\documentclass[a4,12pt]{amsart}  
\usepackage{latexsym}
\usepackage{amssymb}
\usepackage{amsmath}

\usepackage[dvipsnames]{xcolor}
\usepackage{tikz}
\usepackage{enumerate}
\usepackage{mathrsfs}
\usepackage{amsthm}

\usepackage{comment}

\usepackage{url}

\usepackage[left=2.2cm, top=1.7cm,bottom=1.7cm,right=2.2cm]{geometry}

\usepackage{graphicx}

\usepackage[margin=1cm]{caption}

\usepackage{hyperref} 

\usetikzlibrary{shapes.misc}
\usetikzlibrary{decorations.text}

\tikzset{cross/.style={cross out, draw=CadetBlue, minimum size=2*(#1-\pgflinewidth), inner sep=0pt, outer sep=0pt},
cross/.default={4pt}}
\usepackage{subcaption}

\hypersetup{
    colorlinks,
    linkcolor=CadetBlue,
    citecolor=CadetBlue,
    urlcolor=CadetBlue
}

\usepackage{caption}
\usepackage{subcaption}
\numberwithin{equation}{section}

\usepackage[leftcaption]{sidecap}
\sidecaptionvpos{figure}{m}

\usepackage{subcaption}
\usepackage[bottom]{footmisc}

\usepackage{accents}

\usepackage{pgfplots}
\pgfplotsset{compat=1.16}

\tikzset{
    cross/.pic = {
    \draw[rotate = 45] (-#1,0) -- (#1,0);
    \draw[rotate = 45] (0,-#1) -- (0, #1);
    }
}

\definecolor{Cadet2}{rgb}{0.55, 0.65, 0.682}
\definecolor{Cadet3}{rgb}{0.73, 0.68, 0.782}

\DeclareMathOperator{\sgn}{sgn}

\newcommand{\crelaar}{a}

\newtheorem{thm}{Theorem}[section]

\newtheorem{cor}[thm]{Corollary}

\newtheorem{lemma}[thm]{Lemma}
\newtheorem{df}[thm]{Definition}

\usepackage{nomencl}
\makenomenclature

\begin{document}

\title{Continuous Kasteleyn theory for the bead model}
\author{Samuel G. G. Johnston}
\address{Strand Building, King's College London, Strand, London, United Kingdom}

\date{}

\maketitle

\begin{abstract}
Consider the semi-discrete torus $\mathbb{T}_n := [0,1) \times \{0,1,\ldots,n-1\}$ representing $n$ unit length strings running in parallel. 
A bead configuration on $\mathbb{T}_n$ is a point process on $\mathbb{T}_n$ with the property that between every two consecutive points on the same string, there lies a point on each of the neighbouring strings.
In this article we develop a continuous version of Kasteleyn theory to show that partition functions for bead configurations on $\mathbb{T}_n$ may be expressed in terms of Fredholm determinants of certain operators on $\mathbb{T}_n$. 
We obtain an explicit formula for the volumes of bead configurations on $\mathbb{T}_n$. The asymptotics of this formula confirm a recent prediction in the free probability literature. 
Thereafter we study random bead configurations on $\mathbb{T}_n$, showing that they have a determinantal structure which can be connected with exclusion processes. 
We use this machinery to construct a new probabilistic representation of TASEP on the ring.
\end{abstract}

\section{Introduction and overview} \label{sec:introduction} 

\subsection{Interlacing and bead configurations}
Let $S := \{ t_1 < t_2 < \cdots < t_k \}$ and $S' := \{ t_1' < t_2' < \cdots < t_{k'}' \}$ be two collections of real numbers with $|k'-k|\leq 1$. We say that these collections \textbf{interlace} if, between every two consecutive elements of one collection, there is an element of the other collection. In the setting where $k' = k-1$, this reads as saying
\begin{align} \label{eq:interlacer}
t_1 \leq t_1' \leq t_2 \leq  \cdots \leq  t_{k-1} \leq  t_{k-1}' \leq  t_k.
\end{align}
Interlacing is a fundamental phenomenon occuring widely across mathematics. To supply three quick examples here: if $f$ is a polynomial with real roots and $f'$ is its derivative, then the roots of $f$ and $f'$ interlace \cite{fisk}, if $A$ is a Hermitian matrix and $A'$ is any principal minor with one fewer row and column, then the eigenvalues of $A$ and $A'$ interlace \cite{taoRMT}, and if $(p_n)_{n \geq 0}$ (with $\mathrm{deg}(p_n) = n$) is a sequence of orthogonal polynomials with respect to a measure $\mu$ on $\mathbb{R}$, the roots of $p_n$ and $p_{n-1}$ interlace \cite{szego}.

By considering successive interlacing sets (say, by considering the zeros of successive derivatives of the polynomial $f$ above, or by considering the eigenvalues of successive minors of the matrix $A$ above), we obtain a \textbf{bead configuration}: a collection of sets $S_1,\ldots,S_n$ such that each $S_{i+1}$ interlaces $S_i$. We may associate a bead configuration with a subset of $\mathbb{R} \times \{1,\ldots,n\}$.

\emph{Random} bead configurations are ubiquitous across probability, mathematical physics, and random matrix theory \cite{boutillier, ffn, gordenko, JN, MOW, NV, sun}. In spite of their ubiquity, bead configurations are by no means fully understood, and the purpose of the article at hand is to provide a `local solution' to the bead interaction with two main goals in mind: First, we would like to obtain a volume formula for translation invariant bead configurations. Second, we would like to study in detail the various correlations associated with the bead interaction. 
To achieve these goals, we introduce a novel continuous analogue of Kasteleyn theory, a discrete tool in integrable probability that has proven successful over recent decades in solving explicitly numerous statistical physics models such as domino and lozenge tilings \cite{kenyon}.

There are four classes of semi-discrete Abelian groups which provide natural state spaces for translation invariant bead configurations: $\mathbb{R} \times \mathbb{Z}$, $\mathbb{R} \times \mathbb{Z}_n$, $[0,1) \times \mathbb{Z}$ and $[0,1) \times \mathbb{Z}_n$. (Here $\mathbb{Z}_n := \mathbb{Z}/n\mathbb{Z}$ denotes the cyclic group with $n$ elements, and $[0,1)$ the one-dimensional continuous torus.) We retain the greatest flexibility by considering the latter class of these groups, i.e.\ the semi-discrete tori
\begin{align*}
\mathbb{T}_n := [0,1) \times \mathbb{Z}_n,
\end{align*}
which are the most fundamental in that random bead configurations on any of the other groups may be derived using scaling limits of those on $\mathbb{T}_n$. A further advantage of using the tori $\mathbb{T}_n$ is their compactness, and as we will see shortly, we may associate bead configurations on $\mathbb{T}_n$ with closed loops on $\mathbb{T}_n$. 

We think of $\mathbb{T}_n$ as $n$ unit length toric strings wrapped in parallel around a donut, and indexed by $h \in \mathbb{Z}_n$, so that string $h$ lies between $h-1$ and $h+1$ mod $n$. 

\subsection{Bead configurations and occupation processes on the semi-discrete torus} \label{sec:opening}
Reiterating somewhat, for any $k \geq 1$, a \textbf{bead configuration} on the semi-discrete torus $\mathbb{T}_n$ is a collection of $nk$ distinct points $(y_1,\ldots,y_{nk})$ on $\mathbb{T}_n$ (with $y_i = (t_i,h_i)$ for each $i$) such that there are $k$ points on each string $[0,1) \times \{h\}$, and the $k$ points on neighbouring strings interlace, though with certain inequalities taken strictly. More specifically, for each $h \in \mathbb{Z}_n$, if $t_1 < \cdots < t_k$ and $t_1' < \cdots < t_k'$ are the horizontal coordinates of the points on strings $h$ and $h+1$ (mod $n$) respectively, then we have either 
\begin{align} \label{eq:doubleinterlace}
t_1 \leq  t_1 ' < \cdots < t_k \leq  t_k' \qquad \text{or} \qquad t_1' < t_1 \leq  \cdots \leq  t_k' < t_k.
\end{align}
In other words, between every two beads on one string there is a bead on the neighbouring string. The choice of strict and non-strict inequalities in \eqref{eq:doubleinterlace} is motivated by connections with exclusion processes which we make below. See the left panel of Figure \ref{fig:test0} for an example of a bead configuration on $\mathbb{T}_n$.

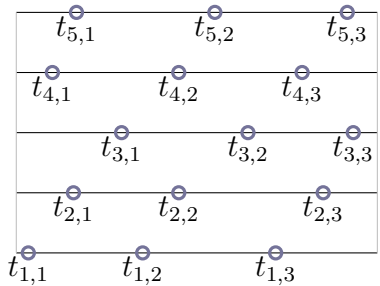
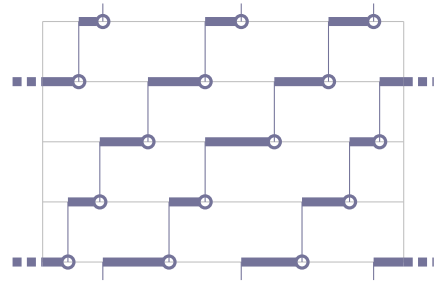
\begin{figure}[h!]
\centering
\begin{subfigure}{.5\textwidth}
\centering
\begin{tikzpicture}[scale=0.8]
\draw (0,0) -- (6,0);
\draw (0,1) -- (6,1);
\draw (0,2) -- (6,2);
\draw (0,3) -- (6,3);
\draw (0,4) -- (6,4);
\draw[lightgray] (0,0) -- (0,4);
\node at (-0.95,2)   (a) {$h \in \mathbb{Z}_n$};
\draw[lightgray, ->] (-0.95,2.5) -- (-0.95,4);
\draw[lightgray, ->] (-0.95,1.5) -- (-0.95,0);
\node at (3,-0.95) (a) {$t \in [0,1)$};
\draw[lightgray, ->] (4.1, -0.95) -- (6, -0.95);
\draw[lightgray, ->] (1.9, -0.95) -- (0, -0.95);
\draw[lightgray] (6,0) -- (6,4);

\draw [very thick, CadetBlue] (0.2,0) circle [radius=0.1];
\node at (0.2,-0.3)   (a) {$t_{1,1}$};
\draw [very thick, CadetBlue] (2.1,0) circle [radius=0.1];
\node at (2.1,-0.3)   (a) {$t_{1,2}$};
\draw [very thick, CadetBlue] (4.31,0) circle [radius=0.1];
\node at (4.31,-0.3)   (a) {$t_{1,3}$};

\draw [very thick, CadetBlue] (0.95,1) circle [radius=0.1];
\node at (0.95,0.7)   (a) {$t_{2,1}$};
\draw [very thick, CadetBlue] (2.7,1) circle [radius=0.1];
\node at (2.7,0.7)   (a) {$t_{2,2}$};
\draw [very thick, CadetBlue] (5.1,1) circle [radius=0.1];
\node at (5.1,0.7)   (a) {$t_{2,3}$};

\draw [very thick, CadetBlue] (1.75,2) circle [radius=0.1];
\node at (1.75,1.7)   (a) {$t_{3,1}$};
\draw [very thick, CadetBlue] (3.85,2) circle [radius=0.1];
\node at (3.85,1.7)   (a) {$t_{3,2}$};
\draw [very thick, CadetBlue] (5.6,2) circle [radius=0.1];
\node at (5.6,1.7)   (a) {$t_{3,3}$};

\draw [very thick, CadetBlue] (0.6,3) circle [radius=0.1];
\node at (0.6,2.7)   (a) {$t_{4,1}$};
\draw [very thick, CadetBlue] (2.7,3) circle [radius=0.1];
\node at (2.7,2.7)   (a) {$t_{4,2}$};
\draw [very thick, CadetBlue] (4.75,3) circle [radius=0.1];
\node at (4.75,2.7)   (a) {$t_{4,3}$};

\draw [very thick, CadetBlue] (1.0,4) circle [radius=0.1];
\node at (1.0,3.7)   (a) {$t_{5,1}$};
\draw [very thick, CadetBlue] (3.3,4) circle [radius=0.1];
\node at (3.3,3.7)   (a) {$t_{5,2}$};
\draw [very thick, CadetBlue] (5.5,4) circle [radius=0.1];
\node at (5.5,3.7)   (a) {$t_{5,3}$};

\draw [white] (2,-0.5) circle [radius=0.01];
\draw [white] (2,4.3) circle [radius=0.1];
\end{tikzpicture}
\caption{A bead configuration on $\mathbb{T}_n$}
  \label{fig:sub10}
\end{subfigure}%
\begin{subfigure}{.5\textwidth}
\centering
\begin{tikzpicture}[scale=0.8]

\draw (0,0) -- (6,0);
\draw (0,1) -- (6,1);
\draw (0,2) -- (6,2);
\draw (0,3) -- (6,3);
\draw (0,4) -- (6,4);
\draw[lightgray] (0,0) -- (0,4);
\node at (-0.95,2)   (a) {$h \in \mathbb{Z}_n$};
\draw[lightgray, ->] (-0.95,2.5) -- (-0.95,4);
\draw[lightgray, ->] (-0.95,1.5) -- (-0.95,0);
\node at (3,-0.95) (a) {$t \in [0,1)$};
\draw[lightgray, ->] (4.1, -0.95) -- (6, -0.95);
\draw[lightgray, ->] (1.9, -0.95) -- (0, -0.95);
\draw[lightgray] (6,0) -- (6,4);

\draw [very thick, CadetBlue] (0.42,0) circle [radius=0.1];
\draw [CadetBlue] (0.42,0) -- (0.42,1); 

\draw [very thick, CadetBlue] (2.1,0) circle [radius=0.1];
\draw [CadetBlue] (2.1,0) -- (2.1,1); 

\draw [very thick, CadetBlue] (4.31,0) circle [radius=0.1];
\draw [CadetBlue] (4.31,0) -- (4.31,1); 

\draw [line width=1.2mm, CadetBlue] (0.42,1) -- (0.85,1); 
\draw [line width=1.2mm, CadetBlue] (2.1,1) -- (2.6,1); 
\draw [line width=1.2mm, CadetBlue] (4.31,1) -- (5.0,1); 

\draw [very thick, CadetBlue] (0.95,1) circle [radius=0.1];
\draw [CadetBlue] (0.95,1) -- (0.95,2); 
\draw [very thick, CadetBlue] (2.7,1) circle [radius=0.1];
\draw [CadetBlue] (2.7,1) -- (2.7,2); 
\draw [very thick, CadetBlue] (5.1,1) circle [radius=0.1];
\draw [CadetBlue] (5.1,1) -- (5.1,2);

\draw [line width=1.2mm, CadetBlue] (0.95,2) -- (1.65,2); 
\draw [line width=1.2mm, CadetBlue] (2.7,2) -- (3.75,2); 
\draw [line width=1.2mm, CadetBlue] (5.1,2) -- (5.5,2); 

\draw [very thick, CadetBlue] (1.75,2) circle [radius=0.1];
\draw [CadetBlue] (1.75,2) -- (1.75,3); 
\draw [very thick, CadetBlue] (3.85,2) circle [radius=0.1];
\draw [CadetBlue] (3.85,2) -- (3.85,3); 
\draw [very thick, CadetBlue] (5.6,2) circle [radius=0.1];
\draw [CadetBlue] (5.6,2) -- (5.6,3); 

\draw [line width=1.2mm, CadetBlue] (1.75,3) -- (2.6,3); 
\draw [line width=1.2mm, CadetBlue] (3.85,3) -- (4.65,3); 
\draw [line width=1.2mm, CadetBlue] (5.6,3) -- (6,3); 
\draw [line width=1.2mm, CadetBlue] (0,3) -- (0.5,3); 

\draw [very thick, CadetBlue] (0.6,3) circle [radius=0.1];
\draw [CadetBlue] (0.6,3) -- (0.6,4);
\draw [very thick, CadetBlue] (2.7,3) circle [radius=0.1];
\draw [CadetBlue] (2.7,3) -- (2.7,4);
\draw [very thick, CadetBlue] (4.75,3) circle [radius=0.1];
\draw [CadetBlue] (4.75,3) -- (4.75,4);

\draw [line width=1.2mm, CadetBlue] (0.6,4) -- (0.9,4); 
\draw [line width=1.2mm, CadetBlue] (2.7,4) -- (3.2,4); 
\draw [line width=1.2mm, CadetBlue] (4.75,4) -- (5.4,4); 

\draw [very thick, CadetBlue] (1.0,4) circle [radius=0.1];
\draw [very thick, CadetBlue] (3.3,4) circle [radius=0.1];
\draw [very thick, CadetBlue] (5.5,4) circle [radius=0.1];

\draw [line width=1.2mm, CadetBlue] (1.0,0) -- (2.0,0); 
\draw [line width=1.2mm, CadetBlue] (3.3,0) -- (4.21,0);
\draw [line width=1.2mm, CadetBlue] (5.5,0) -- (6,0);
\draw [line width=1.2mm, CadetBlue] (0,0) -- (0.32,0);

\draw [CadetBlue] (1.0,-0.3) -- (1.0,0);
\draw [CadetBlue] (1.0,4.3) -- (1.0,4);
\draw [CadetBlue] (3.3,-0.3) -- (3.3,0);
\draw [CadetBlue] (3.3,4.3) -- (3.3,4);
\draw [CadetBlue] (5.5,-0.3) -- (5.5,0);
\draw [CadetBlue] (5.5,4.3) -- (5.5,4);

\draw [line width=1.2mm, CadetBlue, dotted] (-0.5,0) -- (0,0);
\draw [line width=1.2mm, CadetBlue, dotted] (-0.5,3) -- (0,3);
\draw [line width=1.2mm, CadetBlue, dotted] (6,0) -- (6.5,0);
\draw [line width=1.2mm, CadetBlue, dotted] (6,3) -- (6.5,3);

\draw [white] (2,-0.5) circle [radius=0.01];
\draw [white] (2,4.3) circle [radius=0.1];
\end{tikzpicture}
\caption{An occupation process on $\mathbb{T}_n$}
  \label{fig:sub20}
\end{subfigure}
\caption{On the left we have a bead configuration on $n = 5$ strings with $k = 3$ beads per string. On the right we have its associated occupation process, which has occupation number $\ell=2$.}
\label{fig:test0}
\end{figure}

To every bead configuration on $\mathbb{T}_n$ we can associate an occupation process $(X_t)_{t \in [0,1)}$, as in the right panel of Figure \ref{fig:test0}. An \textbf{occupation process} is a right-continuous function $t \mapsto X_t \in \mathcal{P}(\mathbb{Z}_n)$ (the collection of subsets of $\mathbb{Z}_n$) with the property that its discontinuities take the form
\begin{align} \label{eq:b}
y = (t,h) \in \mathbb{T}_n \text{ is a bead} \iff X_t = X_{t-} - \{h\} \cup \{h+1\}.
\end{align}
By right-continuous, we mean on the torus $[0,1)$. There is an integer $\ell \in \{1,\ldots,n-1\}$ such that for all $t \in [0,1)$ we have $\# X_t = \ell$; we call $\ell$ the \textbf{occupation number} of the bead configuration. We say that $y = (t,h) \in \mathbb{T}_n$ is \textbf{occupied} if $h \in X_t$, and that $y \in \mathbb{T}_n$ is \textbf{unoccupied} otherwise.

\begin{figure}
\centering
\begin{tikzpicture}[scale=0.8]
\draw (0,0) -- (6,0);
\draw (0,1) -- (6,1);
\draw (0,2) -- (6,2);
\draw (0,3) -- (6,3);
\draw (0,4) -- (6,4);
\draw[lightgray] (0,0) -- (0,4);
\draw[lightgray] (6,0) -- (6,4);

\draw [very thick, CadetBlue] (0.2,0) circle [radius=0.1];
\draw [very thick, CadetBlue] (2.1,0) circle [radius=0.1];
\draw [very thick, CadetBlue] (4.31,0) circle [radius=0.1];
\draw [very thick, CadetBlue] (0.95,1) circle [radius=0.1];
\draw [very thick, CadetBlue] (2.7,1) circle [radius=0.1];
\draw [very thick, CadetBlue] (5.1,1) circle [radius=0.1];
\draw [very thick, CadetBlue] (1.75,2) circle [radius=0.1];
\node at (1.75,1.6)   (a) {$y_i$};
\draw [very thick, CadetBlue] (3.85,2) circle [radius=0.1];
\draw [very thick, CadetBlue] (5.6,2) circle [radius=0.1];
\draw [<->, fill, CadetBlue] (1.75,2.2) -- (3.85,2.2);
\node at (2.8,2.4)   (a) {$p_i$};
\draw[thick, dotted, CadetBlue] (1.75, 2) -- (1.75, 3.5);
\draw [very thick, CadetBlue] (0.6,3) circle [radius=0.1];
\draw [<->, fill, CadetBlue] (1.75,3.2) -- (2.7,3.2);
\node at (2.25,3.4)   (a) {$q_i$};
\draw [very thick, CadetBlue] (2.7,3) circle [radius=0.1];
\draw [very thick, CadetBlue] (4.75,3) circle [radius=0.1];
\draw [very thick, CadetBlue] (1.0,4) circle [radius=0.1];
\draw [very thick, CadetBlue] (3.3,4) circle [radius=0.1];
\draw [very thick, CadetBlue] (5.5,4) circle [radius=0.1];
\draw [white] (2,-0.5) circle [radius=0.1];
\draw [white] (2,5) circle [radius=0.1];
\end{tikzpicture} 
\caption{The distance $p_i$ from a bead at $y_i = (t_i,h_i)$ to the next bead on string $h_i$, and the distance $q_i$ to the next bead on string $h_i+1$. The \textbf{tilt} $\tau \in (0,1)$ of a configuration is the average of the $q_i$ divided by the average of the $p_i$. If $\ell$ is the occupation number, then $\tau = \ell/n$.}
\label{fig:tiltshow}
\end{figure}
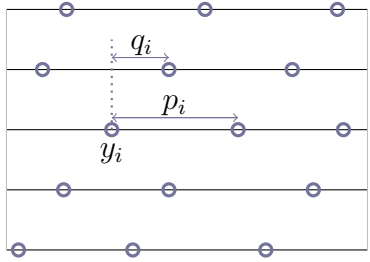

The occupation number is closely related to the \textbf{tilt} of a bead configuration. To define the tilt of a bead configuration $(y_1,\ldots,y_{nk})$, given a bead at $y_i$, let $p_i$ denote the horizontal distance to the next bead on the same string, and let $q_i$ denote the horizontal distance to the next bead on the string above; see Figure \ref{fig:tiltshow}. If $y_i$ is the rightmost bead on a string, take these distances torically. 
The \textbf{tilt} $\tau \in (0,1)$ of a bead configuration is the ratio of the averages of the $q_i$ and of the $p_i$:
\begin{align} \label{eq:tiltdef}
\tau := \frac{ \frac{1}{nk} \sum_{i=1}^{nk} q_i  }{\frac{1}{nk} \sum_{i=1}^{nk} p_i  }.
\end{align}
In fact, since there are $k$ beads on each string, the denominator $\frac{1}{nk} \sum_{i=1}^{nk} p_i$ in \eqref{eq:tiltdef} is equal to $1/k$. The fact that $0 < q_i < p_i$ for each $i$ guarantees $\tau \in (0,1)$.

We now claim that the tilt $\tau \in (0,1)$ and the occupation number $\ell \in \{1,\ldots,n-1\}$ are related via
\begin{align} \label{eq:uy}
\tau = \ell/n.
\end{align}
To see this, first note that since $ \# X_t = \ell$ for each $t$, the total Lebesgue measure of the occupied region of the torus is $\ell$. Alternatively, note that each bead at $y_i = (t_i,h_i)$ is associated with an occupied segment of string of length $q_i$ from $(t_i,h_i+1)$ to $(t_i+q_i,h_i+1)$, taken torically if necessary. The sum of all of these occupied lengths is then $\sum_{i=1}^{nk} q_i = \ell$. On the other hand, $\sum_{i =1}^{nk} p_i$ is simply the total Lebesgue measure of the torus $\mathbb{T}_n$, which is equal to $n$. Plugging these two facts into \eqref{eq:tiltdef}, we obtain \eqref{eq:uy}.

More generally, given an infinite bead configuration (such as Boutillier's \cite{boutillier} bead process on $\mathbb{R} \times \mathbb{Z}$ which we discuss below), one may define the tilt $\tau$ by taking the asymptotics of the ratio in \eqref{eq:tiltdef} as the sum is taken over an asymptotically large region of beads.

\subsection{Partition functions and volumes of bead configurations}

Tying together our definitions so far, an \textbf{$(n,k,\ell)$ configuration} is a bead configuration on $\mathbb{T}_n$ with $k$ beads per string, and such that its associated occupation process $(X_t)_{t \in [0,1)}$ has occupation number $\ell$, or equivalently, the tilt of the configuration is $\ell/n$.

The set of $(n,k,\ell)$ configurations may be associated with a subset
\[\mathcal{W}^{(n)}_{k,\ell} := \{(t_{h,j})_{1 \leq j \leq k, 0 \leq h \leq n-1} \in [0,1)^{nk} : \text{$(n,k,\ell)$ configuration} \}\]
of $[0,1)^{nk}$ by letting $t_{h,j}$ denote the position of the $j^{\text{th}}$ bead on string $h$; see the left panel of Figure \ref{fig:test0}.  Consequently, we can speak of the $nk$-dimensional Lebesgue measure of $\mathcal{W}_{k,\ell}^{(n)}$ and write 
\begin{align} \label{eq:voldef}
\mathrm{Vol}_{k,\ell}^{(n)} :=  \mathrm{Leb}_{nk}( \mathcal{W}_{k,\ell}^{(n)} )  \qquad k \geq 1, 0 \leq \ell \leq n,
\end{align}
for the volume of the set of $(n,k,\ell)$ configurations. When $k \geq 1$, the occupation number $\ell$ always lies in $\{1,\ldots,n-1\}$, and as such $\mathrm{Vol}_{k,0}^{(n)} := \mathrm{Vol}_{k,n}^{(n)} = 0$. 

The empty bead configurations are in correspondence with the constant occupation processes (i.e.\ $X_t = A$ for all $t \in [0,1)$), and since there are $\binom{n}{\ell}$ possible choices for a subset $A$ of $\mathbb{Z}_n$ of cardinality $\ell$, we also set 
\begin{align} \label{eq:voldef2}
\mathrm{Vol}_{0,\ell}^{(n)} := \binom{n}{\ell} \qquad 0 \leq \ell \leq n.
\end{align}
Thus $\mathrm{Vol}_{k,\ell}^{(n)} $ are defined for all $k \geq 0, 0 \leq \ell \leq n$. 

With these definitions at hand, we now state our main results. Define the partition function
\begin{align} \label{eq:pf000}
Z(\lambda,T) := \sum_{k \geq 0} \sum_{ 0 \leq \ell \leq n} T^{nk} e^{-\lambda \ell}\mathrm{Vol}^{(n)}_{k,\ell}.
\end{align}

We endow the torus $\mathbb{T}_n = [0,1) \times \mathbb{Z}_n$ with the natural Lebesgue measure, which gives a measure $b-a$ to each set of the form $(a,b] \times \{h\}$, and hence gives total measure $n$ to $\mathbb{T}_n$. We write $\mathrm{d}y$ for integration with respect to this measure.

The \textbf{Fredholm determinant} $\det (  I + TC  )$ of an operator $C:\mathbb{T}_n \times \mathbb{T}_n \to \mathbb{C}$ is the infinite sum
\begin{align} \label{eq:fredholm}
\det (  I + TC  ) := \sum_{ N =0}^\infty \frac{T^N}{N!} \int_{\mathbb{T}_n^N} \det_{i,j=1}^N C(y_i,y_j) \mathrm{d}y_1 \cdots \mathrm{d}y_N.
\end{align}
Our first main result, Theorem \ref{thm:main}, states that the partition function is a linear combination of Fredholm determinants.

\begin{thm} \label{thm:main}
We have 
\begin{align} \label{eq:prufer}
Z(\lambda,T)= \sum_{ \theta \in \{0,1\}^2 } c_\theta^{\lambda,n} \det( I + T C^{\lambda + \theta_1 \pi \iota, \theta_2} ),
\end{align}
where the sum is over $\theta = (\theta_1,\theta_2) \in \{0,1\}^2$, and for $\beta \in \mathbb{C}-\{0\}$ and $\theta_2 \in \{0,1\}$, $C^{\beta,\theta_2}:\mathbb{T}_n \times \mathbb{T}_n \to \mathbb{C}$ is the translation invariant operator on $\mathbb{T}_n$ given by 
\begin{align} \label{eq:aux00}
C^{\beta,\theta_2} \left( (t,h),(t',h') \right) = e^{ \theta_2 \pi \iota /n} \mathrm{1}_{\{h'=h+1\}} \frac{ e^{ - \beta [t'-t] }}{ 1 - e^{-\beta}},
\end{align}
where $[t'-t] := t' - t + \mathrm{1}_{\{t' < t\}}$ is the residue of $t'-t$ mod $1$. 

The constants $c_\theta^{\lambda,n}$ are given by
\begin{align} \label{eq:cdef}
c_\theta^{\lambda,n} := \frac{1}{2} (-1)^{(\theta_1+1)(\theta_2+n+1)} (1 - (-1)^{\theta_1}e^{-\lambda} )^n.
\end{align}
\end{thm}
Theorem \ref{thm:main} is the bedrock of the article, and may be regarded as a continuous Kasteleyn theorem for the bead process on the torus, in that the partition function $Z(\lambda,T)$ associated with bead configurations on $\mathbb{T}_n$ may be expressed in terms of determinants. 

The benefits of the determinantal expression \eqref{eq:prufer} are twofold. First, we can diagonalise the operators $C^{\beta,\theta_2}$ in order to compute the determinants explicitly, thereby leading to an explicit formula for $Z(\lambda,T)$, and accordingly, the volumes $\mathrm{Vol}_{k,\ell}^{(n)}$. Second, we can sample a random bead configuration according to its contribution to the partition function $Z(\lambda,T)$, and leverage properties of Fredholm determinants to show that such configurations have an explicit determinantal structure. In the former direction, our next two results are explicit formulas for $Z(\lambda,T)$ and $\mathrm{Vol}^{(n)}_{k,\ell}$.

\begin{thm} \label{thm:arrau}
We have
\begin{align} \label{eq:arrau}
Z(\lambda,T) := \frac{1}{2} \sum_{ \theta \in \{0,1\}^2 } (-1)^{(\theta_1+1)(\theta_2+n+1)} \prod_{ z^n = (-1)^{\theta_2} } (e^{Tz} - (-1)^{\theta_1}e^{ - \lambda} ),
\end{align}
where again the sum is over $\theta = (\theta_1,\theta_2)\in\{0,1\}^2$ and the product is over the $n$ complex solutions $z$ to the equation $z^n = (-1)^{\theta_2}$.
\end{thm}

By studying the coefficient of $T^{nk}e^{-\lambda \ell}$ in the expression for $Z(\lambda,T)$ on the right-hand-side of \eqref{eq:arrau}, we are ultimately able to prove the following formula for the volumes of $(n,k,\ell)$ configurations:

\begin{cor} \label{cor:arrau}
The $nk$-dimensional Lebesgue measure of the collection of bead configurations on $\mathbb{T}_n = [0,1)\times \mathbb{Z}_n$ with $k$ beads per string and occupation number $\ell$ is 
\begin{align} \label{eq:vol0}
\mathrm{Vol}^{(n)}_{k,\ell} = \frac{(-1)^{k(\ell+1)}}{(nk)!} \sum_{0 \leq j_1 < \cdots < j_\ell \leq n-1 } ( e^{2 \pi \iota j_1/n} + \cdots + e^{2 \pi \iota j_\ell/n} )^{nk}.
\end{align}
\end{cor}

Our next result gives the fine asymptotics of $\mathrm{Vol}^{(n)}_{k,\ell}$ as $(n,k,\ell)$ tend to infinity jointly under the scaling limit $k =[pn],\ell=[\tau n]$. Since the average distance between consecutive beads in an $(n,k,\ell)$ configuration is $1/k$, we consider the rescaled volumes $k^{nk}\mathrm{Vol}^{(n)}_{k,\ell}$, which may be thought of as measuring volumes of configurations on $n$ strings of length $k$, so that the beads are at unit density. We find that $k^{nk}\mathrm{Vol}^{(n)}_{k,\ell}$ behaves exponentially in the number $nk$ of beads.
\begin{thm}
 \label{thm:cambridge springs}
For $p > 0$ and $\tau \in (0,1)$ let $k = [p n]$ and $\ell = [\tau n]$. Then we have  
\begin{align} \label{eq:ordered9}
\lim_{n \to \infty}  \left( \frac{ e \sin( \pi \ell/n) }{ \pi } \right)^{-nk}k^{nk} \mathrm{Vol}^{(n)}_{k,\ell} = \frac{ e^{ p \pi^2/6}}{ \sqrt{2 \pi p}} \mathcal{P}( e^{ - p q_\tau^+ }) \mathcal{P}( e^{ -p q_\tau^-} ),
\end{align}
where $q_\tau^{\pm} = 2 \pi^2 \pm 2 \pi^2 \iota \frac{\cos(\pi\tau)}{\sin(\pi\tau)}$, and $\mathcal{P}(s) := \sum_{n \geq 0}p(n)s^n$ is the generating function of the integer partitions. 
\end{thm}
The exact form of the constant on the right-hand-side in \eqref{eq:ordered9} appears to be merely a curiosity, with the more important aspect of the formula being the characterisation of the rate of exponential growth of the rescaled volume $k^{nk}\mathrm{Vol}_{k,\ell}^{(n)}$ in the number $nk$ of beads. Indeed, taking logarithmic asymptotics through Theorem \ref{thm:cambridge springs}, we immediately obtain the following corollary characterising the free energy of the bead model:

\begin{cor} \label{cor:free}
For $p > 0$ and $\tau \in (0,1)$ let $k = [p n]$ and $\ell = [\tau n]$. Then we have 
\begin{align} \label{eq:ordered100}
\lim_{n \to \infty} \frac{1}{nk} \log \left\{ k^{nk} \mathrm{Vol}^{(n)}_{k,\ell} \right\} =  1 + \log \sin( \pi \tau) - \log \pi.
\end{align}
That is, the per-bead free energy of the unit-density bead model at tilt $\tau$ is $1 + \log \sin( \pi \tau) - \log \pi$.
\end{cor}

Corollary \ref{cor:free} agrees with a formula recently conjectured by Shlyakhtenko and Tao \cite{ST} in the free probability literature, and provides a statistical physics explanation for Theorem 1.7 of \cite{ST}. We discuss this connection further in Section \ref{sec:free}.

\subsection{Random bead configurations on $\mathbb{T}_n$}
In this section we study random bead configurations and the associated random occupation processes. One natural way to obtain a random bead configuration may be to fix integers $(k,\ell)$, and sample a bead configuration uniformly from the set of bead configurations on $\mathbb{T}_n$ with $k$ beads per string and occupation number $\ell$.  While a simple choice, unfortunately such random configurations do not have a simple correlation structure. 

Instead, it turns out to be more fruitful to instead sample random bead configurations according to their contribution to the partition function $Z(\lambda,T)$. Indeed, we will see shortly that such configurations form determinantal processes. This phenomenon --- i.e.\ wherein a model gains a more tractable correlation structure after the randomisation of a parameter --- is pervasive in integrable probability. Perhaps the most famous example is in Baik, Deift and Johansson \cite{BDJ}, where it is shown that the cycle lengths of a uniformly chosen permutation of $\mathcal{S}_N$, where $N$ is Poisson distributed, have a tractable determinantal structure.\\

As mentioned above, we will sample bead configurations according to their contribution to the partition function $Z(\lambda,T)$. Let us clarify a small point:
\begin{itemize}
\item There is a one-to-one correspondence between non-constant occupation processes and non-empty bead configurations: the beads $(t,h)$ correspond to the discontinuities of $(X_t)_{t \in [0,1)}$ via \eqref{eq:b}.
\item There are $2^n$ occupation processes in correspondence with the empty bead configuration, these are the constant occupation processes, i.e.\ when, for some $A \subseteq \mathbb{Z}_n$, $X_t = A$ for all $t \in [0,1)$. 
\end{itemize}
Recall \eqref{eq:pf000}. 
We now define a probability measure $\mathbf{P}_n^{\lambda,T}$ on occupation processes on $\mathbb{T}_n$ as follows:
\begin{itemize}
\item Sample a pair of random integers $(K,L)$ according to the probability mass function 
\begin{align*}
\mathbf{P}_n^{\lambda,T} ( K=k, L = \ell) := T^{nk} e^{ - \lambda \ell} \mathrm{Vol}_{k,\ell}^{(n)}/Z(\lambda,T).
\end{align*}
\item On the event we obtain $(K,L)=(k,\ell)$ with $k \geq 1$, choose a bead configuration uniformly from the set $\mathcal{W}_{k,\ell}^{(n)}$ of $(n,k,\ell)$ configurations (which as we mentioned in the introduction may be regarded as a subset of $[0,1)^{nk})$, and let $(X_t)_{t \in [0,1)}$ be the associated occupation process.
\item On the event we obtain $(K,L)=(k,\ell)$ with $k =0$, choose a random subset $A \subseteq \mathbb{Z}_n$ with $\# A = \ell$ uniformly from the $\binom{n}{\ell}$ possible, and now set $X_t = A$ for all $t \in [0,1)$.
\end{itemize}
It turns out the probabilities of mixed events such as 
\[\{ \text{$(t_1,h_1)$ is occupied, $(\mathrm{d}t_2,h_2)$ contains a bead} \}\] have a highly tractable structure under the probability measure $\mathbf{P}_n^{\lambda,T}$. This correlation structure is best defined in terms of the \textbf{mixed correlation functions}. Before defining these, we recall that a \textbf{signed (resp.\ complex) probability measure} is a countably additive real-valued (resp.\ complex-valued) function on a sigma algebra whose total mass is one. 

\begin{df}[Mixed correlation functions] \label{df:mixed}
Suppose $\mathbf{P}$ is a complex measure on occupation processes on one of the semi-discrete Abelian groups of the form $\mathbb{T} = R \times Z$, where $R = [0,1)$ or $R = \mathbb{R}$, and $Z = \mathbb{Z}_n$ or $Z = \mathbb{Z}$. Let $(w_1,\ldots,w_N)$ be points of the form $w_i = (\alpha_i,y_i) = (\alpha_i,t_i,h_i)$, where $\alpha_i \in \{b,o,u\}$ and $y_i = (t_i,h_i) \in \mathbb{T}$. Write $i \in \mathcal{B},\mathcal{O},\mathcal{U}$ according to whether $\alpha_i = b/o/u$ respectively. The \textbf{mixed correlation functions} $p_N:( \{b,o,u\} \times \mathbb{T})^N \to \mathbb{C}$ associated with $\mathbf{P}$ are defined via
\begin{align*}
&p_N(w_1,\ldots,w_N) \prod_{i \in \mathcal{B}} \mathrm{d}y_i\\
& := \mathbf{P}( \text{Each $\mathrm{d}y_i$ with $i \in \mathcal{B}$ contains bead, each $y_i$ with $i \in \mathcal{O}$  occ., each $y_i$ with $i \in \mathcal{U}$ unocc.}).
\end{align*} 
\end{df}

The following result says that the probability measure $\mathbf{P}_n^{\lambda,T}$ for occupation processes on $\mathbb{T}_n$ may be realised as a linear combination of four signed probability measures on bead configurations with a rich determinantal structure. 

\begin{thm}\label{thm:detnew}
For $\theta \in \{0,1\}^2$, write $Z^\theta(\lambda,T) :=  c_\theta^{\lambda,n} \det( I + T C^{\lambda+\theta_1 \pi \iota,\theta_2} )$ for the contribution of the $\theta$ term to the partition function in \eqref{eq:prufer}. Then we may write $\mathbf{P}_n^{\lambda,T}$ as an affine combination
\begin{align} \label{eq:Psum0}
\mathbf{P}_n^{\lambda,T} = \sum_{ \theta \in \{0,1\}^2 } \frac{Z^\theta(\lambda,T)}{Z(\lambda,T)} \mathbf{P}_n^{\lambda+\theta_1\pi \iota, \theta_2 ,T},
\end{align}
of signed probability measures. For $\beta$ and $T$ such that $\beta + Tz \notin 2 \pi \iota \mathbb{Z}$ for all $z$ in $\{ z^n = (-1)^{\theta_2} \}$, $\mathbf{P}_n^{\beta, \theta_2 ,T}$ is a complex probability measure on bead configurations with mixed correlation functions
given by 
\begin{align} \label{eq:detnew}
p_N^{\beta,\theta_2,T}( w_1,\ldots,w_N) = \det_{i,j=1}^N H_{\alpha_i}^{\beta,\theta_2,T}(y_j - y_i),
\end{align}
where for $\alpha \in \{b,o,u\}$, $t \in (-1,1)$ and $h \in \mathbb{Z}$ we have
\begin{align} \label{eq:detnewkernel}
H^{\beta,\theta_2,T}_\alpha(t,h) = (-1)^{ \mathrm{1}_{\{\alpha=o\}} } \frac{ T^{\mathrm{1}_{\{\alpha = b\}}} }{n} \sum_{ z^n = (-1)^{\theta_2} } z^{\mathrm{1}_{\{\alpha = b\}}-h} \frac{ e^{ - (\beta+Tz)[t]_\alpha }}{ 1 - e^{ - (\beta+Tz)} } .
\end{align}
Here $[t]_\alpha := t + \mathrm{1}_{\{t<0\}} + \mathrm{1}_{\{\alpha = o\}}\mathrm{1}_{\{t=0\}}$.

When $\beta = \lambda+ \theta_1 \pi \iota$ for $\lambda \in \mathbb{R}$ and $\theta_1 \in \{0,1\}$, $\mathbf{P}_n^{\beta, \theta_2 ,T}$ is a signed probability measure.
\end{thm}
We note that \eqref{eq:Psum0} and \eqref{eq:detnew} entail that the mixed correlation functions $p_N^{\lambda,T}$ of the probability measure $\mathbf{P}_n^{\lambda,T}$ take the form
\begin{align} \label{eq:detnew2}
p_N^{\lambda,T}(w_1,\ldots,w_N) :=\sum_{ \theta \in \{0,1\}^2 } \frac{Z^\theta(\lambda,T)}{Z(\lambda,T)} \det_{i,j=1}^N H^{\lambda+\theta_1 \pi \iota, \theta_2, T}_{\alpha_i} (y_j - y_i).
\end{align}


Theorem \ref{thm:detnew} may be regarded as a continuous generalisation of the complementation principle for discrete determinantal processes \cite{BOO, romik}, which states that if we have a random subset $X$ of a countable set $A$ for which $\mathbb{P}( x_1,\ldots,x_N \in X ) = \det_{i,j=1}^N (K(x_i,x_j))$ for some kernel $K:A \times A \to \mathbb{C}$, then $\mathbb{P}( x_1,\ldots,x_N \notin X ) = \det_{i,j=1}^N ( \mathrm{1}_{x_i =x_j} - K(x_i,x_j) )$. It is possible to show using $\frac{1}{n}\sum_{z^n = (-1)^{\theta_2}} z^h = \mathrm{1}_{h=0}$ for $h \in \mathbb{Z}_n$ that analogously we have
\begin{align*}
H_u^{\beta,\theta_2,T}(t,h) = \mathrm{1}_{(t,h) =(0,0)} - H_o^{\beta,\theta_2,T}(t,h).
\end{align*}
The kernel $H_b^{\beta,\theta_2,T}$ then somehow characterises the differential between the occupied and unoccupied regions.

\subsection{Random bead configurations on $\mathbb{R} \times \mathbb{Z}_n$} 
The remainder of the article is concerned with studying the complex probability measures $\mathbf{P}_n^{\beta,\theta_2,T}$ under a scaling limit where the parameter $T$ controlling the density of beads per string is sent to infinity. After a suitable rescaling, the limiting measures $\mathbf{P}^{n,\ell}$ that arise govern bead configurations on $\mathbb{R} \times \mathbb{Z}_n$, and were recently discovered by Gordenko \cite{gordenko} through studying scaling limits of Young tableaux. Gordenko proved the following:

\begin{thm}[Gordenko \cite{gordenko}] \label{thm:gordenko}
For each $\ell,n$ with $1 \leq \ell \leq n-1$ and a given density of beads per string, there is a unique Gibbs probability measure on bead configurations on $\mathbb{R} \times \mathbb{Z}_n$ with occupation number $\ell$. Under this probability measure, the associated occupation process $(X_t)_{t \in \mathbb{R}}$ is a continuous-time Markov chain, and has the law of $\ell$ independent Poisson walkers on the ring conditioned never to collide.
\end{thm}

With a view to describing the correlations under $\mathbf{P}^{n,\ell}$, consider the following subsets of the roots of unity:

\begin{df} \label{df:rootsets}
Given integers $\ell,n$ with $1 \leq \ell \leq n-1$, let $\theta_2 = n+\ell+1$ mod $2$. We define 
\begin{align*}
\mathcal{L}_{n,\ell} &:= \text{The $\ell$ elements of $\{z^n = (-1)^{\theta_2}\}$ with least real part} \\
\mathcal{R}_{n,\ell} &:= \{z^n = (-1)^{\theta_2}\}-\mathcal{L}_{n,\ell}.
\end{align*}
\end{df}
The parity $\theta_2$ in Definition \ref{df:rootsets} ensures that the two sets are well defined and symmetric with respect to the horizontal axis. See Figure \ref{fig:roots0} for a depiction of the sets $\mathcal{L}_{n,\ell}$ and $\mathcal{R}_{n,\ell}$. 
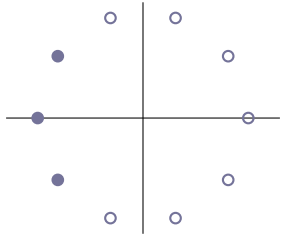
\begin{figure}[h!]
\centering
\begin{tikzpicture}[scale=0.7]
\draw (-2.6,0) -- (2.6,0);
\draw (0,-2.2) -- (0,2.2);
\draw [thick, CadetBlue] (36:2) circle [radius=0.1];
\draw [thick, CadetBlue] (-36:2) circle [radius=0.1];
\draw [thick, CadetBlue] (72:2) circle [radius=0.1];
\draw [thick, CadetBlue] (-72:2) circle [radius=0.1];
\draw [thick, CadetBlue] (108:2) circle [radius=0.1];
\draw [thick, CadetBlue] (-108:2) circle [radius=0.1];
\draw [thick, CadetBlue] (0:2) circle [radius=0.1];
\draw [thick, fill, CadetBlue] (180:2) circle [radius=0.1];
\draw [thick, fill, CadetBlue] (144:2) circle [radius=0.1];
\draw [thick, fill, CadetBlue] (-144:2) circle [radius=0.1];
\end{tikzpicture}
\caption{The sets $\mathcal{L}_{n,\ell}$ (solid circles) and $\mathcal{R}_{n,\ell}$ (hollow circles) depicted for $n = 10$ and $\ell = 3$.}
\label{fig:roots0}
\end{figure}
With Definition \ref{df:rootsets} at hand, for $\alpha \in \{b,o,u\}$ we now define the kernel $H^{n,\ell}_\alpha:\mathbb{R} \times \{-n+1,\ldots,n-1\} \to \mathbb{C}$ by 
\begin{align} \label{eq:kergord}
H^{n,\ell}_\alpha(s,h) = \begin{cases}
(-1)^{\mathrm{1}_{\{\alpha =o\}}} \frac{1}{n} \sum_{ z \in \mathcal{R}_{n,\ell} } z^{\mathrm{1}_{\{\alpha = b\}} -h} e^{-zs} \qquad &\text{if $\{ s > 0\}$ or $\{ \alpha \neq o, s=0\}$},\\
- (-1)^{\mathrm{1}_{\{\alpha =o\}}}  \frac{1}{n} \sum_{ z \in \mathcal{L}_{n,\ell} } z^{\mathrm{1}_{\{\alpha = b\}} -h} e^{-zs} \qquad &\text{if $\{ s<0\}$ or $\{\alpha=o,s=0\}$}.
\end{cases}
\end{align}

The following theorem is an extension and consolidation of results in Gordenko \cite{gordenko} describing the mixed correlation functions of $\mathbf{P}^{n,\ell}$. 

\begin{thm}[Generalisation of Theorems 4 and 5 in \cite{gordenko}] \label{thm:gordnew}
Write $\mathbf{P}^{n,\ell}$ for the law of the unique Gibbs probability measure on bead configurations on $\mathbb{R} \times \mathbb{Z}_n$ with occupation number $\ell$ and with density $\frac{1}{n} \frac{ \sin(\pi \ell/n)}{ \sin( \pi /n)}$ of beads per string. Then the mixed correlation functions of $\mathbf{P}^{n,\ell}$ are given by 
\begin{align} \label{eq:detnewgord}
p_N^{n,\ell}( w_1,\ldots,w_N) = \det_{i,j=1}^N H_{\alpha_i}^{n,\ell}(y_j - y_i).
\end{align}
\end{thm} 
Theorem 4 (resp.\ Theorem 5) of Gordenko \cite{gordenko} corresponds to the special case of Theorem \ref{thm:gordnew} where $\mathcal{B} = \{1,\ldots,N\}$ (resp.\ $\mathcal{O} = \{1,\ldots,N\}$).

In Section \ref{sec:gordenko} we use Theorem \ref{thm:gordnew} to compute the stationary distribution and transition rates of the continuous-time Markov chain $(X_t)_{t \in \mathbb{R}}$ under $\mathbf{P}^{n,\ell}$. These results take shape in the form of the Vandermonde-type function 
\begin{align*}
\Delta(E) := \prod_{1 \leq j < k \leq \ell } | e^{2 \pi \iota h_k/n} - e^{2 \pi \iota h_j/n } | \qquad E = \{h_1,\ldots,h_\ell\},
\end{align*}
on subsets $E$ of $\mathbb{Z}_n$ of cardinality $\ell$. 
\begin{thm} \label{thm:gordnew2}
The transition rates of $(X_t)_{t \in \mathbb{R}}$ under $\mathbf{P}^{n,\ell}$ are given by 
\begin{align} \label{eq:rates0}
\lim_{t \downarrow 0} \frac{1}{t} \mathbf{P}^{n,\ell} ( X_t = E' | X_0 = E ) = \frac{ \Delta(E') }{ \Delta(E)} \qquad E' = E \cup \{h+1\} - \{h\} \text{ for } h+1 \notin E.
\end{align} 
The stationary distribution is given by 
\begin{align*}
\mathbf{P}^{n,\ell}(X_0 = E ) = \Delta(E)^2/n^\ell.
\end{align*}
\end{thm}
The Markov chain $(X_t)_{t \in \mathbb{R}}$ under $\mathbf{P}^{n,\ell}$ is similar to the well-studied totally asymmetric exclusion process (TASEP) on the ring, which is governed by the following dynamics:
under $\mathbf{P}^{n,\ell,\mathrm{TASEP}}$ we have $\ell$ particles at distinct locations of $\mathbb{Z}_n$, and a particle at $h \in \mathbb{Z}_n$ jumps to $h+1$ at rate $1$ if $h+1$ is unoccupied, and otherwise stays put. In other words, in parallel to \eqref{eq:rates0} we have 
\begin{align} \label{eq:ratesT0}
\lim_{t \downarrow 0} \frac{1}{t} \mathbf{P}^{n,\ell,\mathrm{TASEP}}( X_t = E' | X_0 = E ) = 1 \qquad E' = E \cup \{h+1\} - \{h\} \text{ for } h+1 \notin E.
\end{align} 
Despite this innocuous description of the transition rates of TASEP on the ring, it is notoriously difficult to provide formulas for the transition probabilities $\mathbf{P}^{n,\ell,\mathrm{TASEP}}( X_t = E' | X_0 = E )$ for fixed times $t>0$. A recent breakthrough on this front was made by Baik and Liu \cite{BL}, where they provide a complicated but explicit description of these transition probabilities in terms of contour integrals.

With this picture in mind, our last result provides an alternative description of $\mathbf{P}^{n,\ell,\mathrm{TASEP}}$ through $\mathbf{P}^{n,\ell}$ and an exponential martingale. To this end, we define the \textbf{traffic} of a subset $E$ of $\mathbb{Z}_n$ to be the number of elements of $E$ waiting behind a neighbour, i.e.:
\begin{align*}
\mathrm{Traffic}(E) := \# \{ h \in E: h+1 \text{ is also in $E$} \}.
\end{align*}
We now state our final result:

\begin{thm} \label{thm:tasep}
The probability measures $\mathbf{P}^{n,\ell,\mathrm{TASEP}}$ governing TASEP with $\ell$ particles on the ring $\mathbb{Z}_n$ may be recovered from $\mathbf{P}^{n,\ell}$ via the exponential martingale change of measure
\begin{align*}
\frac{\mathrm{d}\mathbf{P}^{n,\ell,\mathrm{TASEP}}}{\mathrm{d}\mathbf{P}^{n,\ell}} \Big|_{\mathcal{F}_t} = \frac{ \Delta(X_0)}{ \Delta(X_t)}  \exp \left\{ \int_0^t \left( \mathrm{Traffic}(X_s) - c_{n,\ell} \right) \mathrm{d}s \right\},
\end{align*}
where $c_{n,\ell} = \ell - \frac{1}{n} \frac{ \sin(\pi \ell/n)}{ \sin( \pi /n) } $ and $\mathcal{F}_t := \sigma ( X_s : 0 \leq s \leq t )$ is the natural filtration of the chain.
\end{thm}
In essence, Theorem \ref{thm:tasep} states that TASEP behaves like a biased version of $(X_t)_{t \in \mathbb{R}}$ under $\mathbf{P}^{n,\ell}$ in which particles are encouraged to spend more time in traffic. Thus TASEP may be recovered via a simple transformation from a process with an explicit integrable structure afforded by Theorem \ref{thm:gordnew}.

\subsection{Further discussion of related work} \label{sec:related}
Bead configurations in their various incarnations have been studied by numerous authors, including Johansson and Nordenstam \cite{JN}, Boutillier \cite{boutillier}, Metcalfe, O'Connell and Warren \cite{MOW}, Sun \cite{sun}, Gordenko \cite{gordenko}, Fleming, Forrester, and Nordenstam \cite{ffn}, and Najnudel and Virag \cite{NV}, not to mention countless other works such as \cite{ANV, ton1, CF, BMRT, FN, mkrtchyan, pete}.

Most notably perhaps, Boutillier \cite{boutillier} showed that at a given density of beads per string, there is a one-parameter family $\{\mathbb{P}^\gamma : \gamma \in (-1,1) \}$ of Gibbs measures for random bead configurations on $\mathbb{R} \times \mathbb{Z}$. The parameter $\gamma$ determines the tilt $\tau \in (0,1)$ of the configuration via the relation $\gamma = \cos \pi \tau$. By sending $n \to \infty$ in Theorem \ref{thm:gordnew} and \eqref{eq:kergord} under the scaling limit $\tau = \ell/n$, and using Cauchy's integral formula, we find that the probability measures $\mathbb{P}^{\gamma}$ have mixed correlation functions with kernel $H_\alpha^\tau :\mathbb{R} \times \mathbb{Z} \to \mathbb{C}$ given by
\begin{align} \label{eq:bouker}
H^{\tau}_\alpha(s,h) = (\mathrm{1}_{s > 0} + \mathrm{1}_{\alpha \neq o, s=0} ) \frac{ \mathrm{1}_{h \geq \mathrm{1}_{\{\alpha = b\}} } (-s)^{ h - \mathrm{1}_{\{\alpha = b\}} } }{(h - \mathrm{1}_{\{\alpha = b\}})!}  + (-1)^{\mathrm{1}_{\{\alpha =o\}}} \int_{\mathcal{C}}  z^{\mathrm{1}_{\{\alpha = b\}} -h} e^{-zs} \frac{\mathrm{d}z}{2 \pi \iota z},
\end{align}
where $\mathcal{C}$ is a contour travelling from $\cos \pi \tau + i \sin \pi \tau$ to $\cos \pi \tau - i \sin \pi \tau$ and staying to the left of the origin.  In other words, with $\gamma = \cos \pi \tau$ and $H^\tau_\alpha$ as in \eqref{eq:bouker}, Boutillier's measure $\mathbb{P}^\gamma$ has mixed correlation functions $p^\tau_N:\{b,o,u\} \times \mathbb{R} \times \mathbb{Z} \to \mathbb{R}$ (Definition \ref{df:mixed}) 
given by
\begin{align} \label{eq:boumixed}
p^\tau_N(w_1,\ldots,w_N) := \det_{i,j=1}^N H_{\alpha_i}^\tau(y_j-y_i).
\end{align}
After reparameterisation, the mixed correlation structure offered here by \eqref{eq:bouker} and \eqref{eq:boumixed} represents an extension of Theorem 2 of \cite{boutillier}, which corresponds to the $\mathcal{O}=\mathcal{U}=\varnothing$ case in the definition of the mixed correlation functions. 

Boutillier \cite{boutillier} noted that the case $\alpha = b$ and $h = 0$ recovers the celebrated sine kernel $H_b(s,0) = \sin( \pi s) /\pi s$. In fact, this observation of Boutillier --- namely that, irrespective of $\tau$, the process along any given string of the doubly infinite bead process behaves like a sine process --- accounts for the ubiquity of the sine process across random matrix theory and mathematical physics: an appearance of a sine process is often underlied by a doubly infinite bead process, be it an exclusion process or Gelfand-Tsetlin pattern. Boutillier mentions \cite[Page 9]{boutillier} that his methods could easily be adapted to obtain the correlation kernels of the bead model on $\mathbb{R} \times \mathbb{Z}_n$ in place of $\mathbb{R} \times \mathbb{Z}$. 

The probability measures $\mathbf{P}^{n,\ell}$ govern bead configurations on $\mathbb{R} \times \mathbb{Z}_n$, and are obtained from a scaling limit of $\mathbf{P}_n^{\beta,\theta_2,T}$ as the density of beads is sent to infinity. It is possible to study an alternative scaling limit, where the number $n$ of strings (rather than the density of beads per string) is sent to infinity, with the other parameters $\beta$ and $T$ fixed. The model that appears in the limit was studied by Metcalfe, Warren and O'Connell \cite{MOW} (see also \cite{metcalfethesis}). Taking $n \to \infty$ in the setting of Theorem \ref{thm:detnew} we obtain kernels of the form 
\begin{equation} \label{eq:MOWEQ}
J_\alpha^{\beta,T}(t,h) = (-1)^{\alpha = o} T^{\mathrm{1}_{\{\alpha = b\}}-h} \int_{ |w| = T} \frac{ \mathrm{d}w}{ 2 \pi \iota w} w^{ \mathrm{1}_{\{\alpha = b\}} - h } e^{ - (\beta+w)[t]_\alpha} (  1-  e^{- (\beta+w)})^{-1}.
\end{equation}
One can solve this contour integral explicitly by observing that there are poles at $w =0$ and at all $w = - \beta + 2 \pi \iota j, j \in \mathbb{Z}$ of modulus at most $T$.  Ultimately, the parameter $T$ only affects the determinantal structure through the number $k$ of poles $|w| < T$ satisfying $\beta+w \in 2 \pi \iota \mathbb{Z}$. This number~ $k$ in turn coincides with the number of beads per string. Thus there is a two-parameter family $\mathbf{P}^{\beta,k}$ of translation-invariant probability measures governing bead configurations on $[0,1) \times \mathbb{Z}$; $k$ is the number of beads per string and $\beta$ controls the tilt. The mixed correlation functions are given by determinants involving the kernels in \eqref{eq:MOWEQ}; c.f.\ \cite[Corollary 5.3]{MOW}.
\vspace{5mm}

We turn to discussing Kasteleyn theory and our continuous approach. Speaking informally, Kasteleyn theory is the art of writing partition functions as determinants. It was pioneered in the 1960s by Dutch physicist Pieter Kasteleyn to provide explicit formulas for partitions functions associated with various models in statistical physics \cite{kast1}, including for instance the number of ways of covering an $n \times n$ chessboard with $2 \times 1$ dominos. We also mention contemporaneous work of Temperley and Fisher \cite{TF}. The tools of (discrete) Kasteleyn theory have been immensely fruitful in recent decades \cite{CKP, KOS}, and continue to be used actively to this day. Introductions to the theory may be found in notes by Kenyon \cite{kenyon,kenyon2} and Toninelli \cite{toninelli}, as well as the recent book of Gorin \cite{gorin2021lectures}.

Finally let us highlight that in the vast majority of prior work on bead configurations, the partition functions and correlation functions are not studied directly, but instead through a scaling limit of a discrete model. Thus one of the innovations of the present article is the direct continuous approach to studying bead configurations using Fredholm determinants. We mention that it is possible to circumvent the continuous Kasteleyn theorem in Theorem \ref{thm:main} and prove the key results of this article, Theorems \ref{thm:arrau} and \ref{thm:detnew}, using a scaling limit of lozenge tilings on the torus. This bespoke discrete approach is recorded in a companion paper, \cite{bead2}, and is substantially less efficient. Take for example the proof of Theorem \ref{thm:arrau}: the discrete approach involves various scaling limits and contour integrals, and takes over ten pages; with the continuous approach, it takes only five.
We hope the reader will regard this as a testament to the efficiency of continuous Kasteleyn theory. 
Finally, as mentioned above, the free energy formula Corollary \ref{cor:free} meets a conjecture in the free probability and statistical physics literature. We postpone discussion of the literature related to this connection to Section \ref{sec:free}.

\subsection{Overview}
The remainder of the paper is structured as follows.

\begin{itemize}
\item In Section \ref{sec:free} we discuss the free energy of the bead model in greater detail, making connections with the asymptotic behaviour of Gelfand-Tsetlin patterns as well as with free probability.
\item In Section \ref{sec:ckastproof} we get started with our proofs of our main results. We begin in this section by developing our continuous analogue of Kasteleyn theory, proving Theorem \ref{thm:main}. Thereafter we diagonalise the continuous Kasteleyn operators to prove Theorem \ref{thm:arrau}.
\item In Section \ref{sec:volproof}, we prove Corollary \ref{cor:arrau}, which gives an explicit formula for the volumes $\mathrm{Vol}_{k,\ell}^{(n)}$ of bead configurations, and afterwards prove Theorem \ref{thm:cambridge springs} on the fine asymptotics of these volumes.
\item In Section \ref{sec:corr} we build further on the continuous Kasteleyn theory of Section \ref{sec:ckastproof}, showing that determinants may be used to detect not just bead configurations, but also the associated occupation processes. This ultimately leads to a proof of Theorem \ref{thm:detnew}.
\item Finally, in Section \ref{sec:gordenko} we study scaling limits of our bead configurations on the torus, establishing connections between random bead configurations on $\mathbb{R} \times \mathbb{Z}_n$ and TASEP on the ring. In particular, in this section we prove Theorems \ref{thm:gordnew}, \ref{thm:gordnew2} and \ref{thm:tasep}.
\end{itemize}

\section{The free energy of the bead model and free probability} \label{sec:free}

We now explain in more detail the connection of the formula \eqref{eq:ordered100} with free probability. This section is independent of the remainder of the paper.

As we mentioned in the introduction, the formula \eqref{eq:ordered100} for the free energy of the bead model agrees with several predictions made in the free probability and statistical physics literature: Based on formal calculations in free probability, Shlyakhtenko and Tao \cite{ST} conjectured the $1 + \log \sin ( \pi \tau ) - \log \pi$ formula explicitly. A less explicit prediction of the form $\log \sin ( \pi \tau) + \mathrm{Const.}~$was made in earlier work by the author and Neil O'Connell \cite{JO} using an argument involving the semicircle law in random matrix theory; this conjectured formula also appears in Gordenko \cite{gordenko}. Sun \cite{sun} gave a sketch argument for an analogous result in a different coordinate system using a scaling limit of the dimer model and an appeal to Legrende duality. In any case, Corollary \ref{cor:free} is the first explicit derivation of the free energy of the bead model based on direct volume computations.

\subsection{Gelfand-Tsetlin patterns}
The relation between bead configurations and free probability begins with Gelfand-Tsetlin patterns, which are fundamental objects in representation theory and algebraic combinatorics.
Gelfand-Tsetlin patterns are essentially bead configurations on $\mathbb{R} \times \{1,\ldots,n\}$ (i.e.\ $n$ infinitely long strings) with one bead on the top string, two beads on the next string down, and so on, so that there are $n$ beads on the bottom string. More precisely, a Gelfand-Tsetlin pattern is a collection of real numbers $(t_{k,j})_{1 \leq j \leq k \leq n}$ satisfying the inequalities $t_{k+1,j} \leq t_{k,j} \leq t_{k+1,j+1}$ for all valid $k,j$. See the left panel of Figure \ref{fig:test} for a depiction of a Gelfand-Tsetlin pattern. 

Consider the set $\mathrm{GT}(s_1,\ldots,s_n) := \{ (t_{k,j})_{1 \leq j \leq k \leq n} \in \mathbb{R}^{n(n+1)/2} : t_{n,j} = s_j \}$ of Gelfand-Tsetlin patterns with a fixed bottom row. By considering just $t_{k,j}$ with $k<n$, $\mathrm{GT}(s_1,\ldots,s_n)$ may be associated with a subset of $\mathbb{R}^{n(n-1)/2}$. Since each $t_{k,j}$ lies in $[s_1,s_n]$, this subset is compact. In fact, it is a consequence of the Weyl dimension formula that the $\frac{n(n-1)}{2}$-dimensional Lebesgue measure of this set is given by
\begin{equation} \label{eq:wdf}
\mathrm{Leb}_{n(n-1)/2} \mathrm{GT}(s_1,\ldots,s_n) = G(n)^{-1} \prod_{ 1 \leq j < k \leq n } (s_k - s_j),
\end{equation}
where $G(n) := \prod_{j =1}^{n-1} j!$ is the Barnes G-function; see e.g.\ \cite{baryshnikov}.

One way in which Gelfand-Tsetlin patterns arise is in the eigenvalue processes of Hermitian matrices. Indeed, as mentioned in the introduction, if $A$ is an $n \times n$ Hermitian matrix with eigenvalues $s_1 \leq \cdots \leq s_n$, then the eigenvalues of its $(n-1)\times (n-1)$ principal minor interlace those of $A$ \cite[Section 1.3]{taoRMT}. Consequently, taking $n$ strings and plotting the $k$ eigenvalues of each $k \times k$ principal minor on the $k^{\text{th}}$ string from the top, we obtain a Gelfand-Tsetlin pattern $(t_{k,j})_{1 \leq j \leq k \leq n}$ from the matrix $A$, which we call \textbf{the eigenvalue process of $A$}. In summary, $t_{k,j}$ is the $j^{\text{th}}$ largest eigenvalue of the $k \times k$ minor.

We now draw on an observation often attributed to Baryshnikov \cite{baryshnikov}, but most likely part of the folklore dating back to Weyl. Suppose we take an $n \times n$ Hermitian random matrix with eigenvalues given by $s_1 \leq \cdots \leq s_n$. 
More specifically, let $U$ be a Haar distributed unitary matrix of dimension $n$, and consider the random matrix $A := U\Lambda U^*$, where $\Lambda$ is the diagonal matrix with entries $(s_1,\ldots,s_n)$. Baryshnikov's observation states that the random Gelfand-Tsetlin pattern associated with the eigenvalue process of $A$ is \emph{uniformly distributed} on the set $\mathrm{GT}(s_1,\ldots,s_n)$. 

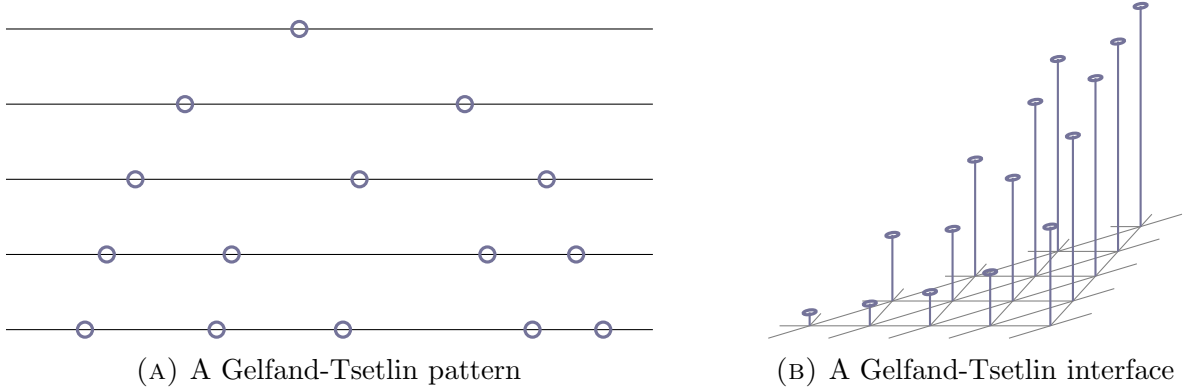
\begin{figure}[h!]
\centering
\begin{subfigure}{.5\textwidth}
  \centering
 \begin{tikzpicture}
\draw (-0.6,0) -- (8,0);
\draw (-0.6,1) -- (8,1);
\draw (-0.6,2) -- (8,2);
\draw (-0.6,3) -- (8,3);
\draw (-0.6,4) -- (8,4);
\draw [very thick, CadetBlue] (6.4,0) circle [radius=0.1];
\draw [very thick, CadetBlue] (2.2,0) circle [radius=0.1];
\draw [very thick, CadetBlue] (7.34,0) circle [radius=0.1];
\draw [very thick, CadetBlue] (0.45,0) circle [radius=0.1];
\draw [very thick, CadetBlue] (3.88,0) circle [radius=0.1];

\draw [very thick, CadetBlue] (2.4,1) circle [radius=0.1];
\draw [very thick, CadetBlue] (5.8,1) circle [radius=0.1];
\draw [very thick, CadetBlue] (0.74,1) circle [radius=0.1];
\draw [very thick, CadetBlue] (6.98,1) circle [radius=0.1];

\draw [very thick, CadetBlue] (1.12,2) circle [radius=0.1];
\draw [very thick, CadetBlue] (4.1,2) circle [radius=0.1];
\draw [very thick, CadetBlue] (6.588,2) circle [radius=0.1];

\draw [very thick, CadetBlue] (1.78,3) circle [radius=0.1];
\draw [very thick, CadetBlue] (5.5,3) circle [radius=0.1];

\draw [very thick, CadetBlue] (3.3,4) circle [radius=0.1];
\end{tikzpicture}
  \caption{A Gelfand-Tsetlin pattern}
  \label{fig:sub1}
\end{subfigure}%
\begin{subfigure}{.5\textwidth}
  \centering
\begin{tikzpicture}[x={(-0.8cm,0cm)}, z={(0cm,0.4cm)}, y={(0.3cm,0.33cm)}]
\draw[gray] (0,0,0) -- (4.5,0,0);
\draw[gray] (0,0,0) -- (0,4.5,0);

\draw [gray] (1,0,0) -- (1,3.5,0);
\draw [gray] (2,0,0) -- (2,2.5,0);
\draw [gray] (3,0,0) -- (3,1.5,0);
\draw [gray] (4,0,0) -- (4,0.5,0);

\draw [gray] (0,1,0,0) -- (3.5,1,0);
\draw [gray] (0,2,0) -- (2.5,2,0);
\draw [gray] (0,3,0) -- (1.5,3,0);
\draw [gray] (0,4,0) -- (0.5,4,0);

\draw [gray] (4.5,-0.5,0) -- (-0.5,4.5,0);
\draw [very thick, CadetBlue] (0,4,7.34) circle [radius=0.1];
\draw [thick, CadetBlue] (0,4,0) -- (0,4,7.34);
\draw [very thick, CadetBlue] (1,3,6.4) circle [radius=0.1];
\draw [thick, CadetBlue] (1,3,0) -- (1,3,6.4);
\draw [very thick, CadetBlue] (2,2,3.88) circle [radius=0.1];
\draw [thick, CadetBlue] (2,2,0) -- (2,2,3.88);
\draw [very thick, CadetBlue] (3,1,2.2) circle [radius=0.1];
\draw [thick, CadetBlue] (3,1,0) -- (3,1,2.2);
\draw [very thick, CadetBlue] (4,0,0.45) circle [radius=0.1];
\draw [thick, CadetBlue] (4,0,0) -- (4,0,0.45);

\draw [gray] (3.5,-0.5,0) -- (-0.5,3.5,0);
\draw [very thick, CadetBlue] (0,3,6.98) circle [radius=0.1];
\draw [thick, CadetBlue] (0,3,0) -- (0,3,6.98);

\draw [very thick, CadetBlue] (1,2,5.8) circle [radius=0.1];
\draw [thick, CadetBlue] (1,2,0) -- (1,2,5.8);

\draw [very thick, CadetBlue] (2,1,2.4) circle [radius=0.1];
\draw [thick, CadetBlue] (2,1,0) -- (2,1,2.4);

\draw [very thick, CadetBlue] (3,0,0.74) circle [radius=0.1];
\draw [thick, CadetBlue] (3,0,0) -- (3,0,0.74);

\draw [gray] (2.5,-0.5,0) -- (-0.5,2.5,0);
\draw [very thick, CadetBlue] (0,2,6.588) circle [radius=0.1];
\draw [thick, CadetBlue] (0,2,0) -- (0,2,6.588);

\draw [very thick, CadetBlue] (1,1,4.1) circle [radius=0.1];
\draw [thick, CadetBlue] (1,1,0) -- (1,1,4.1);

\draw [very thick, CadetBlue] (2,0,1.12) circle [radius=0.1];
\draw [thick, CadetBlue] (2,0,0) -- (2,0,1.12);

\draw [gray] (1.5,-0.5,0) -- (-0.5,1.5,0);
\draw [very thick, CadetBlue] (0,1,5.5) circle [radius=0.1];
\draw [thick, CadetBlue] (0,1,0) -- (0,1,5.5);

\draw [very thick, CadetBlue] (1,0,1.78) circle [radius=0.1];
\draw [thick, CadetBlue] (1,0,0) -- (1,0,1.78);

\draw [gray] (0.5,-0.5,0) -- (-0.5,0.5,0);
\draw [very thick, CadetBlue] (0,0,3.3) circle [radius=0.1];
\draw [thick, CadetBlue] (0,0,0) -- (0,0,3.3);
\end{tikzpicture}
  \caption{A Gelfand-Tsetlin interface}
  \label{fig:sub2}
\end{subfigure}
\caption{On the left we have a Gelfand-Tsetlin pattern with $n=5$. On the right we have the associated Gelfand-Tsetlin interface: the heights of the spikes along the $k^{\text{th}}$ diagonal correspond to the positions of beads on the $k^{\text{th}}$ string.}
\label{fig:test}
\end{figure}

\subsection{Free probability and a variational description of free compression}
Speaking broadly from the perspective of the probabilist,
the field of free probability \cite{voiculescu1998strengthened} is concerned with the asymptotic behaviour of eigenvalues of large random matrices under operations of addition, multiplication and taking minors. 
The operation of free compression characterises the asymptotics of the latter of these operations, taking minors, and will be our focus for the remainder of this section.
We recall from above that the eigenvalue distributions of minors of random matrices are characterised by the positions of beads on higher rows of uniform Gelfand-Tsetlin patterns, and thus free compression is intimately connected with the macroscopic behaviour of such objects.

Free probability gives a (slightly obscure) description of the macroscopic shape of a uniformly chosen Gelfand-Tsetlin pattern whose bottom row $(s_1,\ldots,s_n)$ approximates a probability measure $\mu$ in the sense that $\frac{1}{n} \sum_{i=1}^n \delta_{s_i} \approx \mu$. 
Namely, the empirical distribution of the beads on the $[\crelaar n]^{\text{th}}$ row from the bottom of the pattern approximate a measure $\mu_\crelaar$ of total mass $1-\crelaar$ satisfying the free functional equation
\begin{align} \label{eq:ff}
\crelaar \delta_0 + \mu_\crelaar = (\crelaar\delta_0 + (1- \crelaar) \delta_1) \boxtimes \mu,
\end{align}
where $\boxtimes$ is free multiplicative convolution \cite{MS}. See Metcalfe \cite[Section 1.3]{metcalfe} for further discussion on this front.

In recent work on the \textbf{free entropy} of a probability measure (which we discuss further in Section \ref{sec:FE}), Shlyakhtenko and Tao \cite{ST} obtain an alternative description of the measures $(\mu_\crelaar)_{a \in [0,1]}$ in terms of a variational problem. To describe this here, let $\Delta := \{ 0 \leq t \leq s \leq 1 \}$ and define using $(\mu_\crelaar)_{a \in [0,1]}$ a function $\phi:\Delta \to \mathbb{R}$ by setting, for $u \in [0,1-a]$,
\begin{align} \label{eq:quantile}
\phi(\crelaar+u,u) := \text{$u^{\text{th}}$ quantile of $\mu_a$}.
\end{align}
More explicitly, if $F_\crelaar:\mathbb{R} \to [0,1-a]$ is defined by $F_\crelaar(y) := \int_{-\infty}^y \mu_a(\mathrm{d}r)$, then $\phi(\crelaar+u,u) := \rho_{\mu_a}(u) := F_\crelaar^{-1}(u)$, taken right-continuously.
In particular the diagonal $\phi(u,u) := \rho_\mu(u)$ is the inverse of the distribution function of $\mu$. Shlyakhtenko and Tao \cite{ST} showed in formal calculations that the function $\{\phi(s,t) : 0 \leq t \leq s \leq 1 \}$ satisfies a certain partial differential equation, which in turn characterises $\phi$ as the minimiser of a certain variational problem subject to a boundary condition along the diagonal. 

To describe this variational formulation, define the surface tension
\begin{align} \label{eq:surface}
\sigma(g,\tau) := - \log g - \log \sin ( \pi \tau) - 1 + \log \pi.
\end{align}
Given a function $\phi:\Delta \to \mathbb{R}$, define functions $g^\phi,\tau^\phi:\Delta \to \mathbb{R}$ depending on the gradient of $\phi$ by 
\begin{align} \label{eq:oranga}
g^\phi := \frac{ \partial_s \phi + \partial_t \phi }{ \sqrt{2}} \qquad \text{and} \qquad \tau^\phi := \frac{ \partial_s \phi}{ \partial_s \phi + \partial_t \phi }.
\end{align}
Now set
\begin{align*}
\tilde{\sigma}(\nabla \phi) = \sigma(g^\phi, \tau^\phi).
\end{align*}
We note that $\sigma(g,\tau) = \sigma(g,1-\tau)$, and hence $\tilde{\sigma}(\nabla \phi)$ is symmetric in $\partial_s \phi$ and $\partial_t \phi$.

Theorem 1.7 of \cite{ST} is a variational formulation of the measures $(\mu_a)_{a \in [0,1]}$, stating that the function $\phi:\Delta \to \mathbb{R}$ associated with $(\mu_a)_{a \in [0,1]}$ via \eqref{eq:quantile} minimises the energy integral
\begin{align} \label{eq:costfunction}
\mathcal{I}[\phi] := \int_\Delta \tilde{\sigma}(\nabla \phi) \mathrm{d}s \mathrm{d}t
\end{align}
subject to the boundary condition $\phi(s,s) = \rho_\mu(s)$, where as above, $\rho_\mu(s)$ is the inverse of the distribution function of $\mu$.

\subsection{The statistical physics of Gelfand-Tsetlin patterns}
We now discuss the asymptotic behaviour of Gelfand-Tsetlin patterns from an alternative perspective, providing a statistical physics interpretation of Theorem 1.7 of Shlyakhtenko and Tao \cite{ST}.

We begin by noting that one may reformulate random 
Gelfand-Tsetlin patterns in terms of stochastic interfaces \cite{funaki,sheffield}, as was done in \cite{JO}. A two-dimensional \textbf{stochastic interface} is simply a random function $\phi:\Gamma \to \mathbb{R}$ defined on a subset $\Gamma$ of the two dimensional lattice $\mathbb{Z}^2$ and whose distribution is proportional to $\exp \{  - \sum_{ \langle x, y \rangle } V(\phi(y)-\phi(x)) \}$, where $V:\mathbb{R} \to \mathbb{R}$ is a potential, and the sum is taken over all pairs of directed neighbouring edges in $\Gamma$. 

Given a Gelfand-Tsetlin pattern $(t_{k,j})_{1 \leq j \leq k \leq n}$, we can associate with this pattern a function $\phi_n:\{0 \leq t \leq s \leq 1\} \to \mathbb{R}$ which we call the \textbf{Gelfand-Tsetlin interface} by setting $\phi_n(\frac{k}{n},\frac{j}{n}) := t_{k,j}$, and linearly interpolating. See the right panel of Figure \ref{fig:test}. A uniform Gelfand-Tsetlin pattern with a fixed bottom row is equivalent to a stochastic interface defined on a triangle in the lattice with a fixed diagonal and hard-core interaction potential $V(x) = \infty \mathrm{1}_{x < 0}$. 

Based on Corollary \ref{cor:free}, the volume of a large patch of $N$ beads with average gap $g$ between consecutive beads and tilt $\tau$ has the order $e^{ o(N)} g^Ne^{N(1 + \log \sin \pi \tau - \log \pi )} = e^{ - (1 + o(1))N \sigma(g,\tau)}$, where $\sigma(g,\tau)$ is as in \eqref{eq:surface}. The gap $g$ and tilt $\tau$ of a local patch of beads in the Gelfand-Tsetlin pattern may be recovered from the Gelfand-Tsetlin interface by means of \eqref{eq:oranga}. 

This explains at a heuristic level Shlyakhtenko and Tao's \cite[Theorem 1.7]{ST} variational formulation of the compression measures $(\mu_a)_{a \in [0,1]}$. These measures describe the typical asymptotics of the uniformly chosen Gelfand-Tsetlin pattern whose beads on the bottom row are distributed according to $\mu$. Macroscopically speaking, the beads arrange themselves in such a way as to maximise the volume of the patches of beads across the pattern. Equivalently, the random Gelfand-Tsetlin interface takes on a shape minimising \eqref{eq:costfunction} subject to the boundary condition $\phi(s,s) = \rho_\mu(s)$. 

\subsection{Free entropy} \label{sec:FE}
We now tie our statistical physics discussion back to free entropy, which was introduced and studied by Voiculescu in a series of papers commencing with \cite{V1} and \cite{V2}.  Voiculescu sought to establish parallel notions in free probability to those of entropy and Fisher information in classical probability.

Voiculescu defined the free energy of a probability measure $\mu$ by
\begin{align*}
\tilde{\mathcal{F}}[\mu] := \int_{\mathbb{R} \times \mathbb{R}} \log |t - s | \mu (\mathrm{d}t) \mu(\mathrm{d}s) .
\end{align*}
Shlyakhtenko and Tao \cite{ST} show that after rescaling for variance, the free energy is monotone in the measures $(\mu_a)_{a \in (0,1]}$. Namely, the probability measure $\mu_a$ occuring in \eqref{eq:ff} has variance $a$ times the variance of $\mu$. Accordingly, if $a^{-1/2}_*\mu_a$ denotes the pushforward of this measure under multiplication by $a^{-1/2}$, then $a^{-1/2}_*\mu_a$ has constant variance for $a \in (0,1]$. Shlakhtenko and Tao establish that $\tilde{\mathcal{F}}[a^{-1/2}_*\mu_a]$ is monotone decreasing in $a$.

It turns out that, up to constants, this free energy has an interpretation in terms of the hydrodynamics of Gelfand-Tsetlin patterns. Indeed, as above, let $\rho_\mu:[0,1] \to \mathbb{R}$ be the inverse of the distribution function of $\mu$. By taking logarithmic asymptotics of the Weyl dimension formula \eqref{eq:wdf}, we may define
\begin{align} \label{eq:energi}
\mathcal{F}[\mu] &:= \lim_{n \to \infty} \frac{1}{n^2} \log \mathrm{Leb}_{n(n-1)/2}\mathrm{GT} \left( n \rho_\mu(1/n),n\rho_\mu(2/n),\ldots,n \rho_\mu(n/n) \right) \nonumber \\
&= \frac{1}{2} \int_{\mathbb{R} \times \mathbb{R}} \log |t - s | \mu (\mathrm{d}t) \mu(\mathrm{d}s) + \frac{3}{4},
\end{align}
where we used the fact that $\frac{1}{n^2} \log G(n) = \frac{1}{2} \log n - \frac{3}{4} + o(1)$. 

According to our heuristics so far in this section, should it exist, the minimiser $\phi^*$ of the functional $\mathcal{I}[\phi]$ in \eqref{eq:costfunction} given the fixed boundary condition $\rho_\mu$ should satisfy
\begin{align} \label{eq:minimi}
\mathcal{I}[\phi^*] = - \mathcal{F}[\mu].
\end{align}
It transpires that this is indeed the case. Indeed, we now state a theorem that was proved in more recent work by the author and Joscha Prochno using the explicit formulas of the present article. This theorem links the macroscopic behaviour of Gelfand-Tsetlin patterns with free probability.

\begin{thm}[\cite{JP}] \label{thm:JP}
Let $\mu$ be a probability measure, and let $\rho_\mu:[0,1] \to \mathbb{R}$ denote the right-continuous inverse function of $F_\mu(s) := \int_{-\infty}^s \mu(\mathrm{d}x)$. Let $\phi_n:\Delta \to \mathbb{R}$ be (the linear interpolation of) the stochastic interface associated with a uniformly chosen Gelfand-Tsetlin pattern with bottom row $\rho_\mu(1/n),\ldots,\rho_\mu(n/n)$. Then under a suitable topology the random function $\phi_n$ satisfies a large deviation principle with speed $n^2$ and rate function
\begin{align*}
\mathcal{R}_\mu[\phi] =
\begin{cases}
\mathcal{I}[\phi] + \mathcal{F}[\mu] \qquad &\text{$\phi \in \mathcal{C}^1(\mathring{\Delta})$ with $\phi(s,s) = \rho_\mu(s)$},\\
+\infty \qquad &\text{otherwise}.
\end{cases}
\end{align*}
Here $\mathring{\Delta}$ denotes the interior of $\Delta$, and $\mathcal{I}$ and $\mathcal{F}$ are defined in \eqref{eq:costfunction} and \eqref{eq:energi} respectively. 

In particular, the equation \eqref{eq:minimi} is true, so that $R_\mu[\phi^*] = 0$. 
\end{thm}

Theorem \ref{thm:JP} is an example of a large deviation principle for a stochastic interface. There is a large literature on stochastic interfaces, including on their large deviations. This body of work however usually imposes a strong uniform twice-differentiability conditions on the interaction potential $V$ \cite{DGI, sheffield,funaki,FS,FS1}. The bead interaction potential $V(x) =+\infty \mathrm{1}_{x < 0 }$ is not continuous, let alone twice differentiable, and accordingly, the proof of Theorem \ref{thm:JP} in \cite{JP} requires different tools. Instead, concentration principles from convex geometry are used (as opposed to smoothness argments) to help cement the hard estimates from the integrable probability formulas. (Convex geometry methods for random surfaces have appeared in recent works of Narayanan, Sheffield and Tao \cite{NS, NST}.) In any case, we emphasise that the proof of Theorem \ref{thm:JP} is heavily reliant on the exact formulas we prove in the present article, namely the volume formula given in Theorem \ref{thm:cambridge springs} and the determinantal correlation formula given in Theorem \ref{thm:detnew}. 

\subsection{Further related literature}
We close this section by touching on some related literature.

In the pair of papers \cite{DM1} and \cite{DM2}, Metcalfe and Duse study the asymptotic shapes of discrete Gelfand-Tsetlin patterns, identifying the local correlations in different regions of the pattern.

Perhaps surprisingly, a connected subject of investigation is the asymptotic behaviour of zeros of polynomials under repeated differentiation. Here we recall from the introduction the following observation: if one takes a polynomial with $n$ real roots, then the $n-1$ roots of its derivative interlace those of the original polynomial \cite{sodin}. Thus the zeros of successive derivatives of a polynomial with real roots give rise to a Gelfand-Tsetlin pattern. If one studies for large $n$ the asymptotic distribution of the zeros of the $[\crelaar n]^{\text{th}}$ derivative of a polynomial with $n$ real roots approximating a certain measure $\mu$, then the roots of this $[\crelaar n]^{\text{th}}$ derivative are asymptotically distributed according to the measure $\mu_\crelaar$ appearing in the previous section. The flow of polynomial roots under repeated differentiation is an extremely active area of research \cite{HK, kabluchko, OS, steinerberger}, as is the related emerging field of finite free probability \cite{AGP, GM, MSS}, which Marcus, Spielman and Srivastava have exploited to resolve longstanding open problems in other areas of mathematics \cite{MSS2, MSS3}.

\section{Continuous Kasteleyn theory: proofs}  \label{sec:ckastproof}

Through this section, for $h \in \mathbb{Z}_n$, any occurences of $h-1$ and $h+1$ are to be taken mod $n$. Moreover, for $t,t' \in [0,1)$, we denote by $[t'-t] = t'-t+\mathrm{1}_{\{t' < t\}}$ the residue of $t'-t$ mod $1$.

\subsection{Bead configurations and occupation processes}

We recall from the introduction that a (non-empty) bead configuration $(y_1,\ldots,y_{nk})$ is a collection of $nk$ points on $\mathbb{T}_n$ such that there are $k$ points on each string, and, for each $h \in \mathbb{Z}_n$, if $t_1 < \cdots < t_k$ are the points on string $h$, and $t_1'  < \cdots < t_k'$ are the points on string $h+1$ (mod $n$), then we have either
\begin{align*} 
t_1 \leq  t_1 ' < \cdots < t_k \leq  t_k' \qquad \text{or} \qquad t_1' < t_1 \leq \cdots \leq  t_k' < t_k.
\end{align*}
Our definition of a bead configuration does not depend on the ordering of the points $(y_1,\ldots,y_{nk})$, so that if $(y_1,\ldots,y_{nk})$ is a bead configuration, so is $(y_{\sigma(1)},\ldots,y_{\sigma(nk)})$ for any permutation $\sigma$.

In the introduction, we depicted in Figure \ref{fig:test0} how to every bead configuration we can associate an occupation process $(X_t)_{t \in [0,1)}$ taking values in the collection of subsets of $\mathbb{Z}_n$. We now give a precise definition of this process. Given a bead configuration $(y_1,\ldots,y_{nk})$ and a point $y$ distinct from all $y_i=(t_i,h_i)$ and $y_i+(0,1)=(t_i,h_i+1)$, we 
declare $y = (t,h)$ to be \textbf{occupied}, and write $h \in X_t$, if the first bead on string $h$ after $t$ occurs before the first bead on string $h-1$ after $t$. That is, 
\begin{align} \label{eq:occdef}
h \in X_t \iff  \inf_{i : h_i = h}[t_i-t] <  \inf_{i : h_i = h-1}[t_i-t].
\end{align}
If $(t,h)$ is the location of a bead itself, then $(t,h+1)$ is occupied but $(t,h)$ is unoccupied, unless $(t,h-1)$ is also the location of a bead, in which case $(t,h)$ is occupied. We write $X_t$ for the set of $h$ in $\mathbb{Z}_n$ such that $(t,h)$ is occupied. The cardinality of $X_t$ is constant, and equal to some integer $\ell \in \{1,\ldots,n-1\}$, which we call the occupation number. By considering the cardinality at $t=0$, the occupation number may be written explicitly as 
\begin{align} \label{eq:new}
\ell = \ell(y_1,\ldots,y_{nk}) & =\# \{ h \in \mathbb{Z}_n :  \inf_{i : h_i = h} t_i < \inf_{i : h_i = h-1} t_i \}.
\end{align}
We recall from \eqref{eq:uy} that the occupation number determines the tilt of the configuration via $\tau = \ell/n$.

\subsection{Partition functions} \label{sec:pf}
We recall that an $(n,k,\ell)$ configuration is a bead configuration on $\mathbb{T}_n$ with $k$ beads per string and occupation number $\ell$. 

For $\lambda \in \mathbb{C}$ and $N \geq 1$, we define the function $g_N^\lambda:\mathbb{T}_n^N \to \mathbb{C}$ by 
\begin{align} \label{eq:0gdef}
g_N^{\lambda}(y_1,\ldots,y_N) := 
\begin{cases}
e^{- \lambda \ell} \qquad &\text{$\exists k,\ell: N = nk$, $y_1,\ldots,y_{nk}$ is $(n,k,\ell)$ config.},\\
0 \qquad &\text{otherwise}.
\end{cases}
\end{align}
We emphasise that if $(y_1,\ldots,y_{nk})$ is a bead configuration, so is any reordering, so that $g_N^\lambda$ is invariant under permutations of its entries. Also, $g_N^{\lambda}$ is only nonzero when $N$ is a multiple of $n$. 

We will also define the constant 
\begin{align} \label{eq:crelle}
g_0^\lambda := (1 + e^{-\lambda})^n.
\end{align}
Recall that we write $\mathrm{d}y$ to denote integration with respect to the natural Lebesgue measure on $\mathbb{T}_n$ giving mass $b-a$ to a set of the form $(a,b] \times \{h\}$. Also recall the definition of $\mathrm{Vol}_{k,\ell}^{(n)}$ given in \eqref{eq:voldef}. Note that for $k \geq 1$ we have 
\begin{align} \label{eq:singer1}
\frac{1}{(nk)!} \int_{\mathbb{T}_n^{nk}} \mathrm{1} \{ (y_1,\ldots,y_{nk}) \text{ is an $(n,k,\ell)$ config.}\} \mathrm{d}y_1 \cdots \mathrm{d}y_{nk} = \mathrm{Vol}_{k,\ell}^{(n)},
\end{align}
where the factor $(nk)!$ in front is due to the fact that if $(y_1,\ldots,y_{nk})$ is a bead configuration, so is $(y_{\sigma(1)},\ldots,y_{\sigma(nk)})$ for permutations $\sigma$. As such, for each $k \geq 0$,
\begin{align} \label{eq:singer2}
\frac{1}{(nk)!} \int_{\mathbb{T}_n^{nk}} g_{nk}^\lambda(y_1,\ldots,y_{nk}) \mathrm{d}y_1 \cdots \mathrm{d}y_{nk} = \sum_{ 0 \leq \ell \leq n} e^{ - \lambda \ell} \mathrm{Vol}_{k,\ell}^{(n)},
\end{align}
with the consistency in the $k=0$ case following from \eqref{eq:crelle} and \eqref{eq:voldef2}.

Using \eqref{eq:singer2}, the partition function defined in \eqref{eq:pf000} may then alternatively be written 
\begin{align} \label{eq:umbrella}
Z(\lambda,T) := \sum_{k \geq 0} \sum_{ 0 \leq \ell \leq n} T^{nk} e^{-\lambda \ell}\mathrm{Vol}^{(n)}_{k,\ell} = \sum_{N \geq 0} \frac{T^N}{N!} \int_{\mathbb{T}_n^N} g_N^\lambda(y_1,\ldots,y_N) \mathrm{d}y_1 \cdots \mathrm{d}y_N,
\end{align}
where we recall that $g_N^\lambda = 0$ for all $N$ not divisible by $n$.

We now introduce the foundational idea of this section: $g_N^\lambda$ may be decomposed as a linear combination of functions which have a compliant determinantal structure. For $\theta =(\theta_1,\theta_2) \in \{0,1\}^2$ we define the variant $g_N^{\lambda,\theta}$ of $g_N^\lambda$ by 
\begin{align} \label{eq:0gthetadef}
&g_N^{\lambda,\theta}(y_1,\ldots,y_N) \nonumber \\
&:= 
\begin{cases}
 \frac{1}{2} (-1)^{(\theta_1+k+1)(\theta_2+n+\ell+1)} e^{- \lambda \ell} \qquad &\text{$\exists k,\ell: N = nk$, $y_1,\ldots,y_{nk}$ is $(n,k,\ell)$ config.}\\
0 \qquad &\text{otherwise}.
\end{cases}
\end{align}
Like $g_N^\lambda$, the functions $g_N^{\lambda,\theta}$ are symmetric in their variables, and are zero whenever $N$ is not divisible by $n$. We will also define the constant 
\begin{align} \label{eq:crelle2}
g_0^{\lambda,\theta} := \frac{1}{2} (-1)^{(\theta_1+1)(\theta_2+n+1)} (1 - (-1)^{\theta_1} e^{ - \lambda})^n.
\end{align}
\begin{lemma}
For all $\lambda \in \mathbb{C}$ and $N \geq 0$ we have
\begin{align} \label{eq:gsum}
g^{\lambda}_N = \sum_{ \theta \in \{0,1\}^2} g^{\lambda,\theta}_N.
\end{align}
\end{lemma}
\begin{proof}Inspecting the definitions \eqref{eq:0gdef} and \eqref{eq:0gthetadef}, one can see that when $N \geq 1$ \eqref{eq:gsum} boils down to the \textbf{three-vs-one identity}, which states that 
\begin{align} \label{eq:kid}
\frac{1}{2} \sum_{\theta \in \{0,1\}^2}  (-1)^{(\theta_1+k+1)(\theta_2+n+\ell+1)}  = 1 \qquad k,\ell \in \mathbb{Z}.
\end{align}
To see that \eqref{eq:kid} holds, note that $(\theta_1+k+1)(\theta_2+n+\ell+1)$ is an odd number precisely when $\theta_1 = k $ mod $2$, and $\theta_2 = n+ \ell$ mod $2$, and is an even number for the other three elements $\theta = (\theta_1,\theta_2)$ of $\{0,1\}^2$. It then follows that three of the summands in \eqref{eq:kid} are equal to $+1$, and one is equal to $-1$, proving \eqref{eq:kid} and in turn \eqref{eq:gsum} for $N \geq 1$.

The $N=0$ case of \eqref{eq:gsum} follows immediately from \eqref{eq:crelle} and \eqref{eq:crelle2} (the $\theta_1 =0$ terms cancel out one another). 
\end{proof}

\subsection{Determinantal representations}
The following theorem is the main result of this section, and represents the central step in establishing Theorem \ref{thm:main}.

\begin{thm} \label{thm:ckast}
Let $e^{-\lambda}\neq 1$. For $N \geq 0$ and points $(y_1,\ldots,y_N)$ of $\mathbb{T}_n$ we have 
\begin{align} \label{eq:ckast}
g^{\lambda,\theta}_N(y_1,\ldots,y_N) = \frac{1}{2} (-1)^{(\theta_1+1)(\theta_2+n+1)} (1 - e^{- (\lambda + \theta_1 \pi \iota)} )^n \det\limits_{i,j=1}^{N} C^{\lambda+\theta_1 \pi \iota,\theta_2} (y_j-y_i) ,
\end{align}
where for $\beta \in \mathbb{C}-\{0\}$ and $\theta_2 \in \{0,1\}$, $C^{\beta,\theta_2}:\mathbb{T}_n \to \mathbb{C}$ is given by 
\begin{align} \label{eq:aux0}
C^{\beta,\theta_2} \left( t,h \right) = e^{ \theta_2 \pi \iota /n} \mathrm{1}_{\{h=1\}} \frac{ e^{ - \beta t }}{ 1 - e^{-\beta}}.
\end{align}
\end{thm}
Theorem \ref{thm:ckast} demystifies our perhaps arbitrary seeming choice for the constants $g_0^{\lambda,\theta}$ in \eqref{eq:crelle2}: these constants make Theorem \ref{thm:ckast} consistent in the case $N=0$, with the understanding that a $0 \times 0$ determinant is equal to $1$. Of course, by \eqref{eq:gsum}, the constants $g_0^{\lambda,\theta}$ determine $g_0^{\lambda}$ as it appears in \eqref{eq:crelle}.

In our proof of Theorem \ref{thm:ckast}, it will occasionally lighten notation to write
\begin{align*}
\zeta := e^{ - \lambda}.
\end{align*}

Our work towards proving Theorem \ref{thm:ckast} begins with the following adaptation of an idea from Warren \cite{warren2007dyson}:

\begin{lemma} \label{lem:succession}
Let $\zeta \in \mathbb{C}$, and let $t_1 \leq \cdots \leq t_k$ and $t'_1 \leq \cdots \leq t'_k$ be $2k$ points in $[0,1)$. Then
\begin{align*}
\det \limits_{i,j = 1}^k \left( \zeta^{\mathrm{1}\{ t'_j < t_i\}} \right) = 
\begin{cases}
( 1- \zeta)^{k-1} \qquad &\text{if $t_1 \leq t'_1 < \cdots < t_k \leq t'_k$}\\
(-1)^{k-1} \zeta( 1- \zeta)^{k-1} \qquad &\text{if $ t'_1 < t_1 \leq t'_2 < \cdots \leq  t'_k < t_k$}\\
0 \qquad &\text{otherwise}.
\end{cases}
\end{align*}
\end{lemma}

\begin{proof}
Consider the $k \times k$ matrix $(\zeta^{\mathrm{1}\{t_i<t'_j\}})_{1 \leq i,j \leq k}$. If neither $t_1 \leq t'_1 < \cdots < t_k \leq t'_k$ nor $t'_1 < t_1 \leq t'_2 < \cdots \leq  t'_k < t_k$, then two rows or columns of the matrix are identical, in which case the determinant of this matrix is zero.

Suppose first $t_1 \leq t'_1 < \cdots < t_k \leq t'_k$. Then by subtracting the bottom row of the matrix from each of the other rows to obtain the second equality below, we have 
\begin{align*}
\det \limits_{i,j = 1}^k \left( \zeta^{\mathrm{1}\{ t'_j < t_i\}} \right)   = \det \begin{bmatrix} 
    1 & 1 & \dots  & & 1 \\
    \zeta & 1 & \dots & & \vdots \\
    \vdots & \ddots &  \\
   \zeta &  \zeta & \dots    & 1  & 1 \\
    \zeta &  \zeta & \dots   & \zeta  & 1 
    \end{bmatrix}
=
\det \begin{bmatrix} 
    1-\zeta & 1-\zeta & \dots  & 1 - \zeta & 0 \\
    0 & 1-\zeta & \dots &  1 - \zeta & 0  \\
    \vdots & \ddots &  \\
 0 & \dots  & 0 & 1 -\zeta & 0 \\
    \zeta &  \zeta & \dots   & \zeta  & 1 
    \end{bmatrix}
= (1 - \zeta)^{k-1}.
\end{align*}
We omit the proof of the case $ t'_1 < t_1 \leq t'_2 < \cdots \leq  t'_k < t_k$, which is similar.
\end{proof}

Given a complex parameter $\zeta$, define the auxillary operator $A^\zeta:\mathbb{T}_n \times \mathbb{T}_n \to \mathbb{C}$ by
\begin{align} \label{eq:aux1}
A^\zeta(y,y') = A^\zeta((t,h),(t',h')) := \mathrm{1}_{\{ h' = h+1 \}} \zeta^{\mathrm{1}\{t' <t\}}.
\end{align}
Our next step in proving Theorem \ref{thm:ckast} is the following lemma, which states that determinants involving $A^\zeta$ may be used to detect bead configurations.
\begin{lemma} \label{lem:new det}
Let $\zeta = e^{ - \lambda}$, and for $N \geq 1$ let $(y_1,\ldots,y_N)$ be points of $\mathbb{T}_n$. Then for $k := N/n$ we have
\begin{align} \label{eq:new det}
\det \limits_{i,j =1}^{N} A^\zeta(y_i,y_j) 
= (-1)^{(n-1)k + (k-1) \ell} (1-\zeta)^{n(k-1)} \zeta^\ell \mathrm{1} \{ (y_1,\ldots,y_N) \text{ is bead config.}\}
\end{align}
where $\ell = \ell(y_1,\ldots,y_N)$ is the occupation number defined in \eqref{eq:new}.
\end{lemma}

\begin{proof}

Before delving into the proof, let us recognise immediately that the determinant on the left-hand-side of \eqref{eq:new det} may only be nonzero if $(y_1,\ldots,y_N)$ form a bead configuration: if $h_i = h_j$, and $t_i < t_j$ then the $i^{\text{th}}$ and $j^{\text{th}}$ row of $( A^\zeta(y_i,y_j))_{1 \leq i,j \leq N}$ are identical unless there is some $k$ such that $h_k = h_i+1$ and $t_i \leq t_k < t_j$. Thus \eqref{eq:new det} is plainly true whenever $(y_1,\ldots,y_N)$ is not a bead configuration (since both sides are zero).

We therefore assume without loss of generality for the remainder of the proof that $N=nk$ and $(y_1,\ldots,y_{nk})$ is a bead configuration, ordered so that the vertical coordinates are nondecreasing, and so that points with the same vertical coordinate are ordered so that their horizontal coordinates are increasing. Under this assumption, and using the definition \eqref{eq:aux1}, $( A^\zeta(y_i,y_j))_{1 \leq i,j \leq N}$ is a block matrix with the form 
\begin{align} \label{eq:mrep}
( A^\zeta(y_i,y_j))_{1 \leq i,j \leq N} =
\begin{bmatrix} 
    \mathbf{0} & A_0 & \mathbf{0} & \dots   & \mathbf{0} \\
   \mathbf{0} & \mathbf{0} & A_1 & \dots &  \vdots \\
    \vdots & \ddots &  \\
   \mathbf{0} &  \mathbf{0} & \dots    & \mathbf{0}  & A_{n-2} \\
    A_{n-1} &  \mathbf{0} & \dots   &   & \mathbf{0}
    \end{bmatrix},
\end{align}
where each $A_h$ denotes the $k \times k$ submatrix $A_h := (A^\zeta(y_i,y_j))_{i:h_i=h,j:h_j=h+1}$, and each $\mathbf{0}$ denotes a $k \times k$ matrix of zeros.

We now note that swapping two columns of a matrix changes to parity of the determinant by $-1$. It takes $(n-1)k$ column swaps in order to rearrange the matrix in \eqref{eq:mrep} so that it is block diagonal, i.e.\ $A_0$ occupies the first $k$ rows and $k$ columns of the matrix, $A_1$ occupies the next $k$ rows and next $k$ columns, etc. It then follows from \eqref{eq:mrep} that 
\begin{equation} \label{eq:izzer}
\det_{i,j=1}^N A^\zeta(y_i,y_j) = (-1)^{(n-1)k} \prod_{h \in \mathbb{Z}_n} \det (A_h).
\end{equation} 
Now each matrix $A_h$ is a matrix of the form occurring in Lemma \ref{lem:succession}, where $t_1 \leq \cdots \leq t_k$ are the entries of $(y_1,\ldots,y_N)$ lying on string $h$, and $t'_1 \leq \cdots \leq t'_k$ are the entries lying on string $h+1$. Since $(y_1,\ldots,y_N)$ is a bead configuration, the $t_1 \leq \cdots \leq t_k$ interlace $t'_1 \leq \cdots \leq t'_k$. It follows from Lemma \ref{lem:succession} that for each $h \in \mathbb{Z}_n$,
\begin{align} \label{eq:izzay}
\det(A_h) 
= \begin{cases}
(1-\zeta)^{k-1} \qquad &\text{if } \inf_{i:h_i=h} t_i \leq \inf_{i:h_i=h+1}t_i , \\
(-1)^{k-1}\zeta (1-\zeta)^{k-1} \qquad &\text{if } \inf_{i:h_i=h} t_i > \inf_{i:h_i=h+1}t_i.
\end{cases}
\end{align} 
Using the definition \eqref{eq:new} of $\ell = \ell(y_1,\ldots,y_{nk})$, we immediately see from plugging \eqref{eq:izzay} into \eqref{eq:izzer} that if $(y_1,\ldots,y_{nk})$ is a bead configuration then  
\begin{equation} \label{eq:izzer}
\det_{i,j=1}^N A^\zeta(y_i,y_j) = (-1)^{(n-1)k} ((-1)^{k-1}\zeta)^{\ell} (1-\zeta)^{n(k-1)},
\end{equation} 
completing the proof.
\end{proof}

For $\theta = (\theta_1,\theta_2) \in \{0,1\}^2$, define the operator $B^{\zeta,\theta}:\mathbb{T}_n \times \mathbb{T}_n \to \mathbb{C}$ by 
\begin{align} 
B^{\zeta,\theta}(y,y') = \frac{(-1)^{\theta_2 \mathrm{1}_{\{h'=0\}} } }{1 - (-1)^{\theta_1} \zeta} A^{(-1)^{\theta_1}\zeta} (y,y'), \label{eq:equinox}
\end{align}
where $y = (t,h), y'=(t',h')$, and $A^\zeta$ is defined in \eqref{eq:aux1}. 
We now state and prove our final lemma before proving Theorem \ref{thm:ckast}. 

\begin{lemma} \label{lem:premelan}
With $\zeta = e^{-\lambda}$ we have 
\begin{align} \label{eq:shooman}
 \frac{1}{2} (-1)^{(\theta_1+1)(\theta_2+n+1)} (1 - (-1)^{\theta_1} \zeta)^{n} \det \limits_{i,j=1}^{N} B^{\zeta,\theta} (y_i,y_j)  = g_N^{\lambda,\theta}(y_1,\ldots,y_N).
\end{align}
\end{lemma}

\begin{proof}
For $1 \leq i \leq N$ let $y_i = (t_i,h_i)$. It follows from the definition of $B^{\zeta,\theta}$ and \eqref{eq:new det} that  
\begin{align} \label{eq:new det3}
(1 - (-1)^{\theta_1} \zeta)^{n} \det \limits_{i,j =1}^{N} B^{\zeta,\theta}(y_i,y_j) = (-1)^{(n-1)k+(k-1)\ell + \theta_1 \ell + \theta_2 k } \zeta^\ell \mathrm{1} \{ (y_1,\ldots,y_N) \text{ is bead config.}\},
\end{align}
where $k = N/n$ and $\ell = \ell(y_1,\ldots,y_N)$ is the occupation number.

Now note that modulo $2$ we have
\begin{align} \label{eq:modness}
&(\theta_1+1)(\theta_2+n+1) +(n-1)k + (k-1) \ell + \theta_1 \ell + \theta_2 k \nonumber \\
&= (\theta_1+1)(\theta_2+n+1) + k(\theta_2+n+1)+\ell (\theta_1+1) + k\ell
 \qquad \mathrm{mod}~~ 2 \nonumber \\&
= ( \theta_1 + 1 + k)( \theta_2 + n + 1 + \ell) \qquad \mathrm{mod}~~ 2.
\end{align}
Using \eqref{eq:new det3}, \eqref{eq:modness}, and the definition \eqref{eq:0gthetadef} of $g_N^{\lambda,\theta}$, we obtain the result.
\end{proof}

We now complete the proof of Theorem \ref{thm:ckast}.

\begin{proof}[Proof of Theorem \ref{thm:ckast}]
Since conjugation by a diagonal matrix leaves a determinant unchanged, we make the simple observation that 
\begin{align}
&\det \limits_{i,j =1}^{N} \left(  B^{e^{-\lambda},\theta} (y_i,y_j)   \right) =  \det \limits_{i,j=1}^{N} \left( e^{ \frac{\theta_2 \pi \iota }{n }(h_j-h_i) - ( \lambda+\theta_1 \pi \iota)(t_j - t_i)}  B^{e^{-\lambda},\theta} (y_i,y_j)  \right) \label{eq:rembrandt},
\end{align}
where $(h_j - h_i)$ is a (possibly negative) integer \emph{not} to be taken modulo $n$. 

We now note from \eqref{eq:aux0}, \eqref{eq:aux1} and \eqref{eq:equinox} that
\begin{align*}
e^{ \frac{\theta_2 \pi \iota }{n }(h'-h) - ( \lambda+\theta_1\pi \iota)(t' - t)}  B^{e^{-\lambda},\theta} (y,y') = \frac{ \mathrm{1}_{\{h'= h+1\}} e^{ \frac{ \theta_2 \pi \iota }{n} (h'-h+n\mathrm{1}_{\{h'=0\}}) - ( \lambda + \theta_1  \pi \iota )(t'-t+\mathrm{1}_{\{t' < t\}} ) } }{ 1 - e^{ - (\lambda+\theta_1 \pi \iota ) } } . 
\end{align*}
Now whenever $h'= h+1$ mod $n$ we have $h'-h+n\mathrm{1}_{\{h'=0\}} = 1$. Moreover, $t'-t+\mathrm{1}_{\{t' < t\}}  = [t'-t]$. Therefore 
\begin{align}\label{eq:oranj}
e^{ \frac{\theta_2 \pi \iota }{n }(h'-h) - ( \lambda+\theta_1\pi \iota)(t' - t)}  B^{e^{-\lambda},\theta} (y,y') = \frac{ \mathrm{1}_{\{h'= h+1\}} e^{ \frac{ \theta_2 \pi \iota }{n} - ( \lambda + \theta_1  \pi \iota )[t'-t] }}{ 1 - e^{ - (\lambda+\theta_1 \pi \iota ) } } = C^{\beta,\theta_2}( t'-t,h'-h). 
\end{align}
Plugging \eqref{eq:oranj} and \eqref{eq:rembrandt} into \eqref{eq:shooman}, we obtain the statement of Theorem \ref{thm:ckast}.
\end{proof}

\subsection{Fredholm determinants and diagonalisation}
Recall the definition \eqref{eq:fredholm} of the Fredholm determinant of an operator. Theorem \ref{thm:main} now follows quickly from Theorem \ref{thm:ckast}:
\begin{proof}[Proof of Theorem \ref{thm:main}]
In analogy with \eqref{eq:umbrella} define
\begin{align} \label{eq:umbrellatheta}
Z^\theta(\lambda,T) &:=  \sum_{N \geq 0} \frac{T^N}{N!} \int_{\mathbb{T}_n^N} g_N^{\lambda,\theta}(y_1,\ldots,y_N) \mathrm{d}y_1 \cdots \mathrm{d}y_N,
\end{align}
so that by \eqref{eq:gsum} we have 
\begin{align} \label{eq:Zsum}
Z(\lambda,T) =\sum_{\theta \in \{0,1\}^2} Z^\theta(\lambda,T).
\end{align}
Now by \eqref{eq:ckast} we have
\begin{align} \label{eq:copperone}
Z^\theta(\lambda,T) &:=  \frac{1}{2} (-1)^{(\theta_1+1)(\theta_2+n+1)} (1 - e^{- (\lambda + \theta_1 \pi \iota)} )^n  \sum_{N \geq 0} \frac{T^N}{N!} \int_{\mathbb{T}_n^N}  \det\limits_{i,j=1}^{N} C^{\lambda+\theta_1 \pi \iota,\theta_2} (y_i,y_j) \mathrm{d}y_1 \cdots \mathrm{d}y_N \nonumber \\ 
&=  \frac{1}{2} (-1)^{(\theta_1+1)(\theta_2+n+1)} (1 - e^{- (\lambda + \theta_1 \pi \iota)} )^n \det ( 1 + T C^{\lambda+\theta_1 \pi \iota,\theta_2} ),
\end{align}
where we are abusing notation slightly, and writing $C^{\lambda+\theta_1 \pi \iota,\theta_2} (y_i,y_j) := C^{\lambda+\theta_1 \pi \iota,\theta_2} (y_j-y_i)$ to construct a (translation invariant) operator $C^{\lambda+\theta_1 \pi \iota,\theta_2} : \mathbb{T}_n \times \mathbb{T}_n \to \mathbb{C}$ from $C^{\lambda+\theta_1 \pi \iota,\theta_2} : \mathbb{T}_n \to \mathbb{C}$. 

Using \eqref{eq:Zsum} and \eqref{eq:copperone} together completes the proof of Theorem \ref{thm:main}.
\end{proof}

We now diagonalise the operator $C^{\beta,\theta_2}$ to compute these Fredholm determinants explicitly, thus leading to a proof of Theorem \ref{thm:arrau}.

Define the functions $\phi_{z,m}:\mathbb{T}_n \to \mathbb{C}$ by 
\begin{align} \label{eq:head}
\phi_{z,m}(t,h) := \frac{1}{\sqrt{n}} z^h e^{ - 2 \pi \iota m t }  \qquad z^n = 1, m \in \mathbb{Z}.
\end{align}
The countable collection of functions $\Gamma:= \left\{ \phi_{z,m} : z^n = 1 , m \in \mathbb{Z}\right\}$ form an orthonormal basis of the Hilbert space $\mathcal{L}^2(\mathbb{T}_n)$ of square-integrable functions on $\mathbb{T}_n$ endowed with inner product $\langle \phi, \psi \rangle :=\int_{\mathbb{T}_n} \overline{\phi(y)} \psi(y) \mathrm{d}y $.
They are also eigenfunctions of the operators $C^{\beta,\theta_2}$:

\begin{lemma} \label{lem:eigen}
Each $\phi_{z,m}$ is an eigenfunction of $C^{\beta,\theta_2}$ with eigenvalue $e^{ \theta_2 \pi \iota /n} z(\beta+2\pi \iota m)^{-1}$. 
\end{lemma} 

\begin{proof}
Using the definition \eqref{eq:aux0} of $C^{\beta,\theta_2}$, and writing $[t'-t] = t'-t+\mathrm{1}_{\{t'<t\}}$ for the residue of $t'-t$ mod $1$, we have
\begin{align} \label{eq:stability}
C^{\beta,\theta_2} \phi_{z,m} (t,h) &:= \sum_{ h' \in \mathbb{Z}_n} \int_0^1 C^{\beta,\theta_2} ( (t,h),(t',h')) \phi_{z,m}(t',h') \mathrm{d}t' \nonumber \\
&=e^{ \theta_2 \pi \iota /n}   \sum_{ h' \in \mathbb{Z}_n} \int_0^1 \mathrm{1}_{\{h'=h+1\}} \frac{ e^{ -  \beta  [t'-t] }}{ 1 - e^{- \beta }} z^{h'} e^{ -2 \pi \iota m t'} \mathrm{d}t' \nonumber \\
&=e^{ \theta_2 \pi \iota /n}  \left(\sum_{ h' \in \mathbb{Z}_n} \mathrm{1}_{\{h'=h+1\}} z^{h'} \right) \left(  \int_0^1 \frac{ e^{ -  \beta  [t'-t] }}{ 1 - e^{- \beta }} e^{ -2 \pi \iota m t'} \mathrm{d}t' \right). 
\end{align}
Direct computation tells us that 
\begin{align} \label{eq:350}
\sum_{ h' \in \mathbb{Z}_n} \mathrm{1}_{\{h'=h+1\}} z^{h'} = z^{h+1} \qquad \text{and} \qquad   \int_0^1 \frac{ e^{ -  \beta  [t'-t] }}{ 1 - e^{- \beta }} e^{ -2 \pi \iota m t'} \mathrm{d}t'  = \frac{ e^{ - 2 \pi \iota m t } }{ \beta + 2 m \pi \iota }.
\end{align}
Plugging \eqref{eq:350} into \eqref{eq:stability}, we see that $$C^{\beta,\theta_2} \phi_{z,m}(t,h) = e^{ \theta_2 \pi \iota /n}   z^{h+1}  (\beta+2\pi \iota m)^{-1} e^{ - 2 \pi \iota m t } = e^{ \theta_2 \pi \iota /n}     z(\beta+2\pi \iota m)^{-1} \phi_{z,m}(t,h),$$ sealing the result.
\end{proof}
We now prove Theorem \ref{thm:arrau}.

\begin{proof}[Proof of Theorem \ref{thm:arrau}]
From Lemma \ref{lem:eigen} it follows that each $\phi_{z,m}$ is an eigenfunction of $I+TC^{\beta,\theta_2}$ with eigenvalue 
\begin{align*}
\mu_{z,m} := 1 + \frac{ e^{\theta_2 \pi \iota /n} Tz }{ \beta+ 2 \pi \iota m }.
\end{align*}
Since the operator $I+TC^{\beta,\theta_2}$ has an orthonormal basis of eigenfunctions, it follows that the Fredholm determinant is alternatively characterised through the infinite product
\begin{align} \label{eq:dprod}
\det ( I + T C^{\beta,\theta_2} ) &= \prod_{ z^n =1  } \prod_{m \in \mathbb{Z}} \left( 1 + \frac{ e^{\theta_2 \pi \iota /n} T z }{ \beta+ 2 \pi \iota m }\right) =  \prod_{ z^n = (-1)^{\theta_2} } \prod_{m \in \mathbb{Z}} \left( 1 + \frac{ T z }{ \beta+ 2 \pi \iota m }\right).
\end{align}
For each $z$ with $z^n = (-1)^{\theta_2}$, consider the product over $m \in \mathbb{Z}$ in \eqref{eq:dprod}. By pairing each $m \geq 0$ with $-m$, and accounting for the double counting of the $m=0$ term, we have
\begin{align} \label{eq:dprod2}
\prod_{m \in \mathbb{Z}} \left( 1 + \frac{ T z }{ \beta+ 2 \pi \iota m }\right) =  (1 + Tz/\beta)^{-1} \prod_{ m=0}^\infty \left( 1 + \frac{ 2 \beta T z + T^2 z^2 }{ \beta^2 + 4 \pi^2 m^2} \right).
\end{align}
(The reader concerned about the convergence of the improper infinite product in \eqref{eq:dprod} may resolve to write the determinant as a proper infinite product of determinants $1 + \frac{ 2 \beta T z + T^2 z^2 }{ \beta^2 + 4 \pi^2 m^2}$ over the two-dimensional invariant space spanned by $\phi_{z,m}$ and $\phi_{z,-m}$.)

Appealing to the infinite product identity
\begin{align} \label{eq:dprod3} 
\prod_{m = 0}^\infty \left( 1 + \frac{ \gamma^2 - b^2}{ a^2m^2 + b^2 } \right) = \frac{ \gamma}{a} \frac{ \sinh(\pi \gamma/a)}{\sinh(\pi b/a)} ,
\end{align}
(see e.g.\ \cite{WinNT}) with $a = 2\pi, b=\beta$ and $\gamma=\beta+Tz$, it then follows from \eqref{eq:dprod}, \eqref{eq:dprod2} and \eqref{eq:dprod3} that
\begin{align} \label{eq:det product}
\det ( I + T C^{\beta,\theta_2} ) = \prod_{ z^n = (-1)^{\theta_2} } \frac{ \sinh(\frac{\beta+Tz}{2})}{ \sinh(\frac{\beta}{2})}.
\end{align}
Using the definition of $\sinh$ and exploiting the fact that $\sum_{z^n = (-1)^{\theta_2}}w = 0$, we have 
\begin{align} \label{eq:det product 2}
(1 - e^{ - \beta})^n \det ( I + T C^{\beta,\theta_2} ) = \prod_{ z^n = (-1)^{\theta_2} } (e^{Tz} - e^{ - \beta}).
\end{align}
Using \eqref{eq:det product 2} and \eqref{eq:copperone}, and setting $\beta = \lambda + \theta_1 \pi \iota $, we have
\begin{align} \label{eq:thetaprod}
Z^\theta(\lambda,T) = \frac{1}{2} (-1)^{(\theta_1+1)(\theta_2+n+1)}\prod_{ z^n = (-1)^{\theta_2} } (e^{Tz} - e^{ - (\lambda+\theta_1 \pi \iota )}).
\end{align}
Theorem \ref{thm:arrau} now follows from \eqref{eq:Zsum} and \eqref{eq:thetaprod}.
\end{proof}

\section{Volumes of bead configurations: Proofs of Corollary \ref{cor:arrau} and Theorem \ref{thm:cambridge springs}} \label{sec:volproof}

In this section we prove Corollary \ref{cor:arrau}, giving an explicit formula for the volume $\mathrm{Vol}^{(n)}_{k,\ell}$ of bead configurations on $\mathbb{T}_n$ with $k$ beads per string and occupation number $\ell$. Thereafter we prove Theorem \ref{thm:cambridge springs}, which gives fine asymptotics of the volumes $\mathrm{Vol}^{(n)}_{k,\ell}$. 

\subsection{Proof of Corollary \ref{cor:arrau}}
With a view to proving Corollary \ref{cor:arrau}, let us note that by using the definition \eqref{eq:pf000} of $Z(\lambda,T)$, and setting $\zeta = e^{-\lambda}$, we obtain from Theorem \ref{thm:arrau} the relation
\begin{align} \label{eq:ckast4}
\sum_{k \geq 0} \sum_{ 0 \leq \ell \leq n} T^{nk} \zeta^\ell \mathrm{Vol}^{(n)}_{k,\ell} = \frac{1}{2} \sum_{ \theta \in \{0,1\}^2 } (-1)^{(\theta_1+1)(\theta_2+n+1)} \prod_{ z^n = (-1)^{\theta_2} } (e^{Tz} + (-1)^{\theta_1+1} \zeta). 
\end{align}
We are now ready to prove Corollary \ref{cor:arrau}.

\begin{proof}[Proof of Corollary \ref{cor:arrau}]
By \eqref{eq:ckast4}, $\mathrm{Vol}^{(n)}_{k,\ell}$ is the coefficient of $T^{nk}\zeta^\ell$ in the right-hand-side of \eqref{eq:ckast4}. Expanding the product in the right-hand-side of \eqref{eq:ckast4}, we have 
\begin{align} \label{eq:rog1}
\sum_{k \geq 0} \sum_{ 0 \leq \ell \leq n} T^{nk} \zeta^\ell \mathrm{Vol}^{(n)}_{k,\ell} &= \frac{1}{2} \sum_{ \theta \in \{0,1\}^2 } (-1)^{(\theta_1+1)(\theta_2+n+1)} \sum_{ \ell = 0}^n (-1)^{(\theta_1+1)\ell} \zeta^\ell \sum_{\mathcal{S} \subseteq \{z^n=(-1)^{\theta_2}\}, \# \mathcal{S} = n-\ell } e^{ T \sum_{z \in \mathcal{S}} z},
\end{align}
where the final sum is over subsets $\mathcal{S}$ of the $n^{\text{th}}$ roots of $(-1)^{\theta_2}$ of cardinality $n-\ell$. 

Comparing coefficients of $\zeta^\ell$, for each $0 \leq \ell \leq n$ we have 
\begin{align} \label{eq:chu}
\sum_{k \geq 0}  T^{nk} \mathrm{Vol}^{(n)}_{k,\ell} &= \frac{1}{2} \sum_{ \theta \in \{0,1\}^2 } (-1)^{(\theta_1+1)(\theta_2+n+1+\ell)} \sum_{\mathcal{S} \subseteq \{z^n=(-1)^{\theta_2}\}, \# \mathcal{S} = n-\ell } e^{ T \sum_{z \in \mathcal{S}} z}.
\end{align}
We now expand the right-hand-side of \eqref{eq:chu} as a power series in $T$. To this end, note that we have the equality of sets $\{ z : z^n =(-1)^{\theta_2} \} = \{ ze^{ \theta_2 \pi \iota /n} : z^n = 1 \}$. Using this fact to obtain the first equality below, then expanding as a power series in $T$ to obtain the second, we have 
\begin{align} \label{eq:ancient}
\sum_{\mathcal{S}  \subseteq \{z^n=(-1)^{\theta_2}\}, \# \mathcal{S} = n-\ell }   e^{ T \sum_{z \in \mathcal{S}} z } &= \sum_{\mathcal{S} \subseteq \{z^n=1\}, \# \mathcal{S}=n-\ell } e^{ e^{\theta_2 \pi \iota /n} T \sum_{z \in \mathcal{S}} z } \nonumber \\
&= \sum_{ N = 0}^\infty \frac{T^N}{N!} e^{ \theta_2 \pi \iota N/n} \sum_{\mathcal{S} \subseteq \{z^n=1\}, \# \mathcal{S}=n-\ell  } \left( \sum_{z \in \mathcal{S}} z\right)^N.
\end{align}
By rotational symmetry, $\sum_{\mathcal{S} \subseteq \{z^n=1\}, \# \mathcal{S}=n-\ell } \left( \sum_{z \in \mathcal{S}} z \right)^N$ is zero unless $N$ is a multiple of $n$.
In particular, instead summing over $N = nk$ in \eqref{eq:ancient} we have 
\begin{align} \label{eq:ancient2}
\sum_{\mathcal{S} \subseteq \{z^n=(-1)^{\theta_2}\},\#\mathcal{S}=n-\ell } e^{ T \sum_{z \in \mathcal{S}} z } &= \sum_{ k = 0}^\infty \frac{T^{nk}}{(nk)!} (-1)^{\theta_2 k} \sum_{\mathcal{S} \subseteq \{z^n=1\},\#\mathcal{S}=n-\ell} \left( \sum_{z \in \mathcal{S}} z \right)^{nk}.
\end{align}
Plugging \eqref{eq:ancient2} into \eqref{eq:chu}, and isolating the power of $T^{nk}$, we obtain
\begin{align} \label{eq:japan99}
\mathrm{Vol}^{(n)}_{k,\ell}  = \frac{1}{(nk)!} \left(  \sum_{\mathcal{S} \subseteq \{z^n=1\},\#\mathcal{S}=n-\ell} \left( \sum_{z \in \mathcal{S}} z \right)^{nk} \right) \left( \frac{1}{2} \sum_{ \theta \in \{0,1\}^2 } (-1)^{(\theta_1+1)(\theta_2+n+1+\ell) + \theta_2k } \right).
\end{align}
We now account for the sum over $\theta$ in \eqref{eq:japan99}. Using the fact that $(\theta_1+1)(\theta_2+n+1+\ell) + \theta_2k  = (\theta_1+1+k)(\theta_2+n+1+\ell) -k(n+\ell+1)$ we have
\begin{align} \label{eq:japan1000}
 \frac{1}{2} \sum_{ \theta \in \{0,1\}^2 } (-1)^{(\theta_1+1)(\theta_2+n+1+\ell) + \theta_2k } = \frac{1}{2} (-1)^{k(n+\ell+1)} \sum_{ \theta \in \{0,1\}^2 } (-1)^{(\theta_1+1+k)(\theta_2+n+1+\ell)} = (-1)^{k(n+\ell+1)},
\end{align}
where the latter equality follows from the three-vs-one identity \eqref{eq:kid}. Plugging \eqref{eq:japan1000} into \eqref{eq:japan99} we have 
\begin{align} \label{eq:japan2}
\mathrm{Vol}^{(n)}_{k,\ell} = \frac{(-1)^{k(n+\ell+1)}}{(nk)!} \sum_{\mathcal{S} \subseteq \{z^n=1\},\#\mathcal{S}=n-\ell} \left( \sum_{z \in \mathcal{S}} z \right)^{nk}.
\end{align}
We now make a final adjustment. Note that since the roots of unity sum to zero, for any subset $\mathcal{S}$ of $\{z^n=1\}$ we have 
\begin{align} \label{eq:minus}
\sum_{z \in \mathcal{S}} z = - \sum_{ z \in \{z^n=1\} - \mathcal{S}} z.
\end{align}
Consequently, we can take the internal sum in \eqref{eq:japan2} over subsets $\mathcal{S}$ of $\{z^n=1\}$ with cardinality $\ell$ instead of $n-\ell$, and account for this change by multiplying by a factor of $(-1)^{nk}$. Using this adjustment, and instead writing the roots of unity in terms of integer powers of $e^{2 \pi \iota/n}$, we complete the proof of Corollary \ref{cor:arrau}.
\end{proof}
We remark by \eqref{eq:minus} and the following remark, we have the following symmetry in the volume:
\begin{align} \label{eq:symmetry}
\mathrm{Vol}^{(n)}_{k,\ell} = \mathrm{Vol}^{(n)}_{k,n-\ell}
\end{align}

In fact, \eqref{eq:symmetry} follows directly from the following observation. Let $(y_1,\ldots,y_{nk})$ (here $y_i = (t_i,h_i)$) be a bead configuration with occupation number $\ell$ and $(t_1,\ldots,t_{nk})$ distinct, and set $y'_i := (1-t_i,h_i)$. Then $(y_1',\ldots,y_{nk}')$ is a bead configuration with occupation number $n-\ell$.

\subsection{Proof of Theorem \ref{thm:cambridge springs}}

In this section we give a proof of Theorem \ref{thm:cambridge springs}, concerning the asymptotics of the volumes $\mathrm{Vol}_{k,\ell}^{(n)}$. In fact, we will prove the following stronger version of the result, of which the statement given in Theorem \ref{thm:cambridge springs} is a consequence. 

\begin{thm} \label{thm:cs2}
Let $p=k/n$ and $\tau = \ell/n$. Then whenever $n^{-1/4} \leq \tau \leq 1 - n^{-1/4}$ and $p \geq n^{-1/4}$ we have 
\begin{align} \label{eq:ordered9b}
 \left( \frac{ e \sin( \pi \ell/n) }{ \pi } \right)^{-nk}k^{nk} \mathrm{Vol}^{(n)}_{k,\ell} = \frac{ e^{ p \pi^2/6}}{ \sqrt{2 \pi p}}\left( \mathcal{P}( e^{-pq^-_\tau}) \mathcal{P}( e^{-pq_\tau^+} ) + O\left(e^{ - cp\tau \log(n) }\right) \right).
\end{align}
\end{thm}

We start preparing the proof of Theorem \ref{thm:cs2}. Throughout this section, the symbols $\mathcal{S},\mathcal{S}',\mathcal{C}$ will always refer to a subset of $\mathbb{Z}_n$ of cardinality $\ell$. We note that by \eqref{eq:symmetry}, we may assume without loss of generality that $\tau \in [0,1/2]$. We thus make this assuption in several of our preparatory lemmas. We will use the notation $c,C>0$ for positive universal constants that may change from line to line.

Let us begin by noting that if $\mathcal{S}'$ is a translation of $\mathcal{S}$ in $\mathbb{Z}_n$, i.e. $\mathcal{S}' = \{ j + j_0 : j \in \mathcal{S}\}$ for some $j_0$ in $\mathbb{Z}_n$, then we have $\sum_{j \in \mathcal{S}'} e^{2 \pi \iota j/n} = e^{ 2 \pi \iota j_0/n} \sum_{w \in \mathcal{S} } e^{2 \pi \iota j/n}$, and consequently
\begin{align} \label{eq:sle}
\left( \sum_{j \in \mathcal{S}'} e^{2 \pi \iota j/n} \right)^{nk} = \left( \sum_{j \in \mathcal{S}} e^{2 \pi \iota j/n} \right)^{nk} \qquad \text{whenever $\mathcal{S}'$ is a translation of $\mathcal{S}$ in $\mathbb{Z}_n$}.
\end{align}
We say that $\mathcal{S}$ is \textbf{centered} if 
\begin{align*}
\mathrm{arg} \left( \sum_{j \in \mathcal{S}} e^{ 2 \pi \iota j/n} \right) \in \left[-\frac{\pi}{2n},\frac{3 \pi}{2n}\right).
\end{align*}
Every set $\mathcal{S}$ has a unique centered translation $\mathcal{S}'$, so that using \eqref{eq:sle} in \eqref{eq:vol0} we have
\begin{align} \label{eq:volclass}
\mathrm{Vol}^{(n)}_{k,\ell} =n \frac{(-1)^{k(\ell-1)}}{(nk)!} \sum_{ \mathcal{S} \text{ centered}}\left( \sum_{j \in \mathcal{S}} e^{ 2 \pi \iota j/n} \right)^{nk},
\end{align}
where the sum is over all centered subsets $\mathcal{S}$ of $\mathbb{Z}_n$ of cardinality $\ell$. 

Consider now that to maximise the modulus $|  \sum_{j \in \mathcal{S}} e^{2 \pi \iota j /n}|$ over a set $\mathcal{S}$ of cardinality $\ell$, we should take $\mathcal{S}$ as a set of consecutive elements around the circle. We call such sets \textbf{consecutive}. In this case, given $\mathcal{C}_{j_0} = \{ j_0, j_0+1,\ldots,j_0+\ell-1\}$ we may compute
\begin{align} \label{eq:concsum}
\sum_{j \in \mathcal{C}_{j_0} }e^{2 \pi \iota j/n} = \sum_{ p  = 0}^{\ell-1} e^{ 2 \pi \iota (j_0 + p)/n } = e^{ 2 \pi \iota j_0/n} \frac{ e^{ 2 \pi \iota \ell /n } - 1}{ e^{ 2 \pi \iota /n} - 1 } = e^{ \frac{2 \pi \iota }{n } ( \frac{\ell - 1}{2} + j_0 ) } \frac{ \sin( \pi \ell/n  )}{ \sin ( \pi /n )  }.
\end{align}
Note then that
\begin{align*}
\left(\sum_{j \in \mathcal{C}_{j_0} }e^{2 \pi \iota j/n} \right)^{nk} = (-1)^{k(\ell-1)} \left( \frac{ \sin( \pi \ell/n  )}{ \sin ( \pi /n )  } \right)^{nk},
\end{align*}
which, of course by \eqref{eq:sle}, is independent of the choice of $j_0$.

\begin{figure}[h!]
\centering
\begin{tikzpicture}[scale=0.7]
\draw (-2.6,0) -- (2.6,0);
\draw (0,-2.2) -- (0,2.2);

\draw [thick, fill, CadetBlue] (0:2) circle [radius=0.1];
\draw [thick, fill, CadetBlue] (18:2) circle [radius=0.1];
\draw [thick, fill, CadetBlue] (36:2) circle [radius=0.1];
\draw [thick, CadetBlue] (54:2) circle [radius=0.1];
\draw [thick, fill, CadetBlue] (72:2) circle [radius=0.1];
\draw [thick, CadetBlue] (90:2) circle [radius=0.1];
\draw [thick, fill, CadetBlue] (108:2) circle [radius=0.1];
\draw [thick, CadetBlue] (124:2) circle [radius=0.1];
\draw [thick, CadetBlue] (144:2) circle [radius=0.1];
\draw [thick, CadetBlue] (162:2) circle [radius=0.1];

\draw [thick, fill, CadetBlue] (-18:2) circle [radius=0.1];

\draw [thick, fill, CadetBlue] (-36:2) circle [radius=0.1];

\draw [thick, fill, CadetBlue] (-54:2) circle [radius=0.1];

\draw [thick, CadetBlue] (-72:2) circle [radius=0.1];

\draw [thick, CadetBlue] (-90:2) circle [radius=0.1];

\draw [thick, fill, CadetBlue] (-108:2) circle [radius=0.1];

\draw [thick, CadetBlue] (-124:2) circle [radius=0.1];

\draw [thick, CadetBlue] (-144:2) circle [radius=0.1];

\draw [thick, CadetBlue] (-162:2) circle [radius=0.1];

\draw [thick, CadetBlue] (180:2) circle [radius=0.1];

\end{tikzpicture}
\caption{The centered set $\mathcal{S}$ associated with partitions $\lambda^+ = (2,1)$ and $\lambda^- = (2)$.}
\label{fig:roots66}
\end{figure}
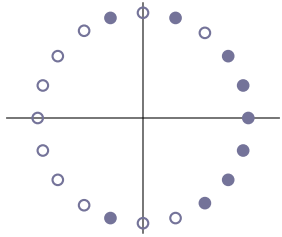

We now consider a useful way of indexing centered sets in terms of integer partitions. An \textbf{integer partition} $\lambda =(\lambda_1,\lambda_2,\ldots)$ is a nonincreasing sequence of nonnegative integers, only finitely many of which are nonzero. The \textbf{length} of $\lambda$ is the number of nonzero elements in the sequence.
The \textbf{mass} of $\lambda$ is  $|\lambda| := \sum_{ j \geq 1} \lambda_j$. 
Let $\mathcal{P}$ denote the set of all integer partitions.

Through the remainder of this section we make the canonical association 
\begin{align*}
\mathbb{Z}_n = \mathbb{Z} \cap (-n/2,n/2].
\end{align*}

One can construct centered sets by first taking a centered consecutive set, and then use a pair of integer partitions to translate the leftmost and rightmost elements of the set. To obtain a rough idea of how this is done, see Figure \ref{fig:roots66}. The following lemma sets up this correspondence explicitly:

\begin{lemma} \label{lem:corred}
The centered sets $\mathcal{S} \subseteq \mathbb{Z} \cap (-n/2,n/2]$ of cardinality $\ell$ are in bijection with a subset of $\mathcal{P}^2$ via the relation $\mathcal{S} \mapsto (\lambda^+,\lambda^-)$ defined by writing
\begin{align*}
\mathcal{S}^+ := \mathcal{S} \cap (0,n/2] = \{ \ell^+ - j +1 + \lambda_j^+ : j = 1,\ldots,\ell^+ \} \qquad \ell^+ = \# \mathcal{S}^+,
\end{align*}
and
\begin{align*}
\mathcal{S}^- := \mathcal{S} \cap (-n/2,0] = \{ - ( \ell^- - j + \lambda_j^-) : j = 1,\ldots,\ell^- \} \qquad \ell^- = \# \mathcal{S}^-.
\end{align*} 
\end{lemma}
\begin{proof}
By construction, to each centered set $\mathcal{S}$ we may associate a quadruplet $(\ell^+,\ell^-,\lambda^+,\lambda^-)$. Different centered sets determine different quadruplets, and in particular, different pairs $(\lambda^+,\lambda^-)$.

Conversely, suppose $(\ell^+,\ell^-,\lambda^+,\lambda^-)$ and $(\tilde{\ell}^+,\tilde{\ell}^-,\lambda^+,\lambda^-)$ are quadruplets associated with two centered sets $\mathcal{S}$ and $\mathcal{S}'$ with the same pair of integer partitions $(\lambda^+,\lambda^-)$. Then by construction, $\mathcal{S}'$ is a translation of $\mathcal{S}$. Since they are both centered, we must have $\mathcal{S} = \mathcal{S}'$, and hence $(\ell^+,\ell^-) = (\tilde{\ell}^+,\tilde{\ell}^-)$ also. It follows that each pair $(\lambda^+,\lambda^-)$ is associated to at most one centered set.

Thus the correspondence between centered sets and pairs of partitions is bijective.
\end{proof}
For suitable partitions $(\lambda^+,\lambda^-)$, we will write $\mathcal{S}(\lambda^+,\lambda^-)$ for the unique centered set associated with $(\lambda^+,\lambda^-)$. There are pairs of partitions individually of mass up to $O(n^2)$ associated to centered sets. When $n$ is large, any pair of partitions each of mass $o_\tau(n)$ is guaranteed to be associated to a centered set. 


Given a centered set $\mathcal{S}$ of cardinality $\ell$, our next result grants us control over the possible ratios $\ell^+/n$ and $\ell^-/n$, where $\ell^+$ and $\ell^-$ are as in the statement of Lemma \ref{lem:corred}.

\begin{lemma} \label{lem:balance}
Let $\mathcal{S}$ be a centered set of cardinality $\ell = \tau n$, where $\tau \in [0,1/2]$. Then the ratios $\tau^+ := \ell^+/n$ and $\tau^- = \ell^-/n$ both satisfy the inequality
\begin{align*}
\tau^\pm \geq c \tau^2 - C/n
\end{align*}
for some universal constants $c,C > 0$ not depending on $\ell,n$ or the centered set $\mathcal{S}$. 
\end{lemma}
\begin{proof}
We will prove that $\tau^+ \geq c \tau^2 - C/n$. The proof of the lower bound for $\tau^-$ is almost identical.

Consider the problem of finding a subset $\mathcal{S}^+$ of $\mathbb{Z} \cap (0,n/2]$ of cardinality $\ell^+$ that maximises the imaginary part of $\sum_{j \in \mathcal{S}^+} e^{2 \pi \iota j / n }$. One would try to select all the $\ell^+$ elements of $\mathcal{S}^+$ in such a way that the roots $e^{2 \pi \iota j/n}$ lies as close to the imaginary unit $\iota$ as possible. Thus, with the supremum below taken over all subsets $\mathcal{S}^+$ of $\mathbb{Z} \cap (0,n/2]$ of cardinality $\ell^+$, we have 
\begin{align*}
\sup_{ \mathcal{S}^+ = \ell^+ } \mathrm{Im} \left( \sum_{j \in \mathcal{S}^+} e^{2 \pi \iota j / n } \right) &= \mathrm{Im} \left(  \sum_{|j| \leq \lfloor \ell^+/2 \rfloor }  \iota e^{ 2 \pi \iota j/n} \right ) + O(1) \\
&=  \frac{n}{2\pi}  \int_{-\pi \tau^+}^{ \pi \tau^+} \cos \phi ~\mathrm{d} \phi + O(1) =  \frac{n}{ \pi} \sin( \pi \tau^+ ) + O(1),
\end{align*}
where the $O(1)$ terms are universal.

Likewise, now consider choosing a subset $\mathcal{S}^-$ of $\mathbb{Z} \cap (-n/2,0]$ of cardinality $\ell^-$ that maximises the imaginary part of $\sum_{j \in \mathcal{S}^-} e^{2 \pi \iota j / n }$.
Since the roots of unity $e^{2 \pi \iota j/n}$ with $j \in \mathcal{S}^-$ must all have negative (or zero) imaginary part, we would try to choose the roots as close to the real axis as possible. It follows that in this case, with the supremum below taken over all subsets $\mathcal{S}^-$ of $\mathbb{Z} \cap (-n/2,0]$ of cardinality $\ell^-$, we have 
\begin{align*}
\sup_{ \mathcal{S}^- = \ell^- } \mathrm{Im} \left( \sum_{j \in \mathcal{S}^-} e^{2 \pi \iota j / n } \right) &= - 2 \mathrm{Im} \left(  \sum_{j = 0}^{\lfloor \ell^-/2 \rfloor } e^{2 \pi \iota j /n} \right) + O(1) \\
&=-   \frac{n}{\pi}  \int_0^{ \pi \tau^-} \sin \phi~\mathrm{d} \phi + O(1) = -  \frac{n}{ \pi} (1 - \cos( \pi \tau^-))  + O(1).
\end{align*}
In particular, for any subset $\mathcal{S}$ of $\mathbb{Z} \cap (-n/2,n/2]$ with $\# \mathcal{S}^\pm = \ell^\pm$ we have
\begin{align*}
\mathrm{Im} \left( \sum_{j \in \mathcal{S}} e^{2 \pi \iota j / n } \right) \leq \frac{n}{\pi} \left( \sin(\pi \tau^+) - (1 - \cos(\pi \tau^-)\right) + O(1),
\end{align*}
where $\tau^\pm = \ell^\pm/n$. 

Since a centered set has $ \mathrm{Im} \left( \sum_{j \in \mathcal{S}} e^{2 \pi \iota j / n } \right) \geq -\pi/2n$, it follows that the ratios $\tau^\pm$ associated with a centered set must satisfy
\begin{align*}
 \frac{n}{\pi} \left( \sin(\pi \tau^+) - (1 - \cos(\pi \tau^-)\right) + O(1) \geq - \pi/2n. 
\end{align*}
If $\tau^+ + \tau^- = \tau$, writing $\tau^+ = \phi$ and $\tau^- = \tau - \phi$ we have
\begin{align*}
\sin( \pi \phi) - (1 - \cos( \pi( \tau - \phi)) \geq O(1/n).
\end{align*} 
Since $x \geq \sin(x)$ and, for some universal $c > 0$, $1 - \cos(x) \geq cx^2$ for $x \in [0,\pi/2]$, we have 
\begin{align*}
\pi \phi \geq c (\tau - \phi)^2 + O(1/n).
\end{align*}  
This ultimately implies
\begin{align*}
\phi \geq c \tau^2 + O(1/n)
\end{align*}
for some, possibly different, universal $c>0$. Setting $\tau^+ = \phi$ completes the proof. 
\end{proof}



With the correspondence set out in Lemma \ref{lem:corred}, we say that a centered set $\mathcal{S}$ associated with a pair $(\lambda^+,\lambda^-)$ of partitions is \textbf{good} if 
\begin{align*}
|\lambda^+| \leq \log n \qquad \text{and} \qquad  |\lambda^-| \leq \log n.
\end{align*}
Otherwise a centered set is \textbf{bad}.
Our next lemma tells us how good sets contribute to the formula \eqref{eq:vol0} for $\mathrm{Vol}^{(n)}_{k,\ell}$. 
\begin{lemma} \label{lem:shalom}
Let $\mathcal{S} = \mathcal{S}(\lambda^+,\lambda^-)$ be a good centered set. Then with $q_\tau^{\pm} = 2 \pi^2 \pm 2 \pi^2 \iota \frac{\cos(\pi\tau)}{\sin(\pi\tau)}$ as in the statement of Theorem \ref{thm:cambridge springs}, and $k=pn$, $\ell = \tau n$, we have 
\begin{align} \label{eq:olem}
\left( \sum_{j \in \mathcal{S} }e^{2 \pi \iota j/n} \right)^{nk} = (-1)^{k(\ell-1)} \left( \frac{ \sin( \pi \ell/n)}{ \sin(\pi/n) }\right)^{nk} \left\{ e^{ - pq_\tau^+ |\lambda^-| - pq_{\tau}^- |\lambda^+|}  +O \left(\frac{p \log^2 n}{ \tau n } \right) \right\}.
\end{align}
\end{lemma}

\begin{proof}
By \eqref{eq:sle} we may instead compute $\left( \sum_{j \in \mathcal{S}' }e^{2 \pi \iota j/n} \right)^{nk} $, where $\mathcal{S}'$ is any translation of $\mathcal{S}$. Thus, to simplify calculations, we consider $\mathcal{S}' := \mathcal{C}(\lambda^+,\lambda^-)$ defined by
\begin{align*}
\mathcal{C}(\lambda^+,\lambda^-):= \{ R_j + \lambda^+_j :1 \leq j \leq \ell^+ \} \cup \{ L_j - \lambda^-_j : 1 \leq j \leq \ell^- \},
\end{align*}
where $L_j := j-1$ and $R_j := \ell - j$. Note that $\mathcal{C}(0,0) := \{0,\ldots,\ell-1\}$ is consecutive.
By definition we have 
\begin{align*}
 \sum_{j \in \mathcal{C}(\lambda^+,\lambda^-)}e^{2 \pi \iota j/n} = \sum_{j \in \mathcal{C}(0,0)}e^{2 \pi \iota j/n}+ \sum_{j \geq 1} \left( e^{ \frac{2 \pi \iota }{n} (R_j + \lambda^+_j)} - e^{ \frac{2 \pi \iota }{n} R_j} \right) + \sum_{j \geq 1} \left( e^{ \frac{2 \pi \iota }{n} (L_j - \lambda^-_j)} - e^{ \frac{2 \pi \iota }{n} L_j} \right) .
\end{align*}
Using \eqref{eq:concsum} with $j_0=0$ to perform the sum over $\mathcal{C}(0,0)$, the fact that $L_j=j-1$ and $R_j=\ell-j$, and the approximation $e^{ix}-1=ix+O(x^2)$, we have
\begin{align*}
 \sum_{j \in \mathcal{C}(\lambda^+,\lambda^-)}e^{2 \pi \iota j/n} = e^{ \frac{2\pi \iota}{n} \frac{\ell-1}{2} } \frac{ \sin( \pi \ell/n)}{ \sin(\pi/n) } + \frac{2 \pi \iota }{n} \left( |\lambda^+| e^{ \frac{2 \pi \iota \ell}{n} } - |\lambda^-| \right) + O\left( \frac{1}{n^2} \sum_{j \geq 1} ((\lambda^+_j)^2 + (\lambda^-_j)^2 ) \right).
\end{align*}
Factoring the leading term, using $\sin(\pi/n) = \pi/n + O(1/n^3)$, and the fact that $\sum_{j \geq 1} (\lambda^\pm)_j^2 \leq |\lambda^\pm|^2 \leq \log^2 n$, we have
\begin{align*}
 \sum_{j \in \mathcal{C}(\lambda^+,\lambda^-)}e^{2 \pi \iota j/n}  = e^{ \frac{2\pi \iota}{n} \frac{\ell-1}{2} } \frac{ \sin( \pi \ell/n)}{ \sin(\pi/n) } \left\{ 1 + \frac{1}{n^2} \frac{2 \pi^2 \iota }{\sin(\pi\ell/n)} \left( e^{ \frac{2 \pi \iota }{n}\frac{\ell}{2} }  |\lambda^+| - e^{ - \frac{2 \pi \iota }{n}\frac{\ell}{2} } |\lambda^-| \right) + O \left(\frac{\log^2 n}{ \tau n^3 } \right) \right\}.
\end{align*}
Taking the $(nk)^{\text{th}}$ power with $k=pn$ and $\ell=\tau n$ we obtain 
\begin{align*} 
&\left(  \sum_{j \in \mathcal{C}(\lambda^+,\lambda^-)}e^{2 \pi \iota j/n}  \right)^{nk}\\
&= (-1)^{k(\ell-1)} \left( \frac{ \sin( \pi \ell/n)}{ \sin(\pi/n) }\right)^{nk}  \exp \left\{ p \frac{2 \pi^2 \iota }{ \sin(\pi \tau) }\left( e^{\iota \pi \tau }  |\lambda^+| - e^{ - \iota \pi \tau } |\lambda^-| \right)    + O \left(\frac{p \log^2 n}{ \tau n } \right)  \right\}.
\end{align*}
The result follows now from observing that
\begin{align*}
\frac{2 \pi^2 \iota  }{ \sin(\pi \tau) } e^{\iota \pi \tau }  = - q_\tau^- \qquad \text{and} \qquad  - \frac{2 \pi^2 \iota  }{ \sin(\pi \tau) }e^{- \iota \pi \tau }   = - q_\tau^+.
\end{align*}
\end{proof}

Where the previous lemma offered a fine estimate of $\sum_{j \in \mathcal{S}} e^{2 \pi \iota j/n}$ for good centered sets, our next lemma offers an upper bound on the modulus of this quantity for \emph{any} centered set.  
\begin{lemma} \label{lem:kontrol}
Let $n^{-1/4} \leq \tau \leq 1/2$. 
Then whenever $\mathcal{S} = \mathcal{S}(\lambda^+,\lambda^-)$ we have 

\begin{align*}
\left| \sum_{j \in \mathcal{S}} e^{2 \pi \iota j/n} \right| \leq \frac{ \sin( \pi \ell/n  )}{ \sin ( \pi /n )  } \exp \left\{  - \frac{c \tau}{n^2} ( |\lambda^+| + |\lambda^-|)  \right\}
\end{align*}
for some universal constants $c$.

\end{lemma}
\begin{proof}
Throughout the proof $c,C>0$ denote universal constants that may change from appearance to appearance. 
We begin by studying the real part of $\sum_{j \in \mathcal{S}} e^{2 \pi \iota j/n} $. Considering first the $\mathcal{S}^+$ part, we have
\begin{align} \label{eq:gould}
\mathrm{Re}  \sum_{j \in \mathcal{S}^+} e^{2 \pi \iota j/n} &= \mathrm{Re} \sum_{j=1}^{\ell^+} e^{2 \pi \iota ( \ell^+ - j +1) /n} - \mathrm{Re} \sum_{j=1}^{\ell^+} \left\{  e^{2 \pi \iota ( \ell^+ - j +1) /n} -  e^{2 \pi \iota ( \ell^+ - j +1 + \lambda_j^+) /n}\right\} \nonumber \\
&\leq \mathrm{Re} \sum_{j=1}^{\ell^+} e^{2 \pi \iota ( \ell^+ - j +1) /n} - \mathrm{Re} \sum_{j=1}^{\lfloor \ell^+/2 \rfloor} \left\{  e^{2 \pi \iota ( \ell^+ - j +1) /n} -  e^{2 \pi \iota ( \ell^+ - j +1 + \lambda_j^+) /n}\right\}.
\end{align}
We now seek to obtain a lower bound on the quantities
\begin{align*}
\mathrm{Re} \left\{  e^{2 \pi \iota ( \ell^+ - j +1) /n} -  e^{2 \pi \iota ( \ell^+ - j +1 + \lambda_j^+) /n}\right\} = \cos(\theta_j) - \cos(\theta_j+s_j), \qquad 1 \leq j \leq \lfloor \ell^+/2 \rfloor,
\end{align*}
where $\theta_j := 2 \pi ( \ell^+ - j +1) /n$ and $s_j = 2 \pi \lambda_j^+/n$. 

There is a universal $c>0$ such that for every $\theta_0 \in (0,\pi/2)$ we have
$\cos \theta - \cos (\theta+s) \geq c \theta_0 s$ whenever $0\leq \theta_0 \leq \theta \leq \theta + s \leq \pi/2$. Set $\theta_0 = 2 \pi (\ell^+ - \lfloor \ell^+/2 \rfloor + 1 )/n$. Note $\pi/2 \geq \theta_j \geq \theta_0$ for all $1 \leq j \leq \lfloor \ell^+/2 \rfloor$. Also note by Lemma \ref{lem:balance} we have $ \theta_0 \geq \ell^+/n \geq c\tau^2 - C/n$. Thus
\begin{align} \label{eq:gould2}
\mathrm{Re} \left\{  e^{2 \pi \iota ( \ell^+ - j +1) /n} -  e^{2 \pi \iota ( \ell^+ - j +1 + \lambda_j^+) /n}\right\} \geq (c \tau^2 - C/n) \lambda_j^+/n, \qquad 1 \leq j \leq \lfloor \ell^+/2 \rfloor.
\end{align}
Using \eqref{eq:gould2} in \eqref{eq:gould}, together with the rough bound $\sum_{j=1}^{\lfloor \ell^+/2 \rfloor} \lambda^+_j \geq \frac{1}{3} |\lambda^+|$, we obtain 
\begin{align} \label{eq:nara1}
\mathrm{Re}  \sum_{j \in \mathcal{S}^+} e^{2 \pi \iota j/n} \leq \mathrm{Re} \sum_{j=1}^{\ell^+} e^{2 \pi \iota ( \ell^+ - j +1) /n} - (c\tau^2 - C/n) |\lambda^+|/n.
\end{align}
A similar calculation for the roots indexed by $\mathcal{S}^-$ tells us that 
\begin{align} \label{eq:nara2}
\mathrm{Re}  \sum_{j \in \mathcal{S}^-} e^{2 \pi \iota j/n} &\leq \mathrm{Re} \sum_{j=1}^{\ell^-} e^{ - 2 \pi \iota ( \ell^- - j) /n} - (c\tau^2 - C/n) |\lambda^-|/n.
\end{align}
Combining \eqref{eq:nara1} and \eqref{eq:nara2} we have
\begin{align} \label{eq:nara5}
\mathrm{Re}  \sum_{j \in \mathcal{S}} e^{2 \pi \iota j/n} &\leq \mathrm{Re} \sum_{j=1}^{\ell^+} e^{2 \pi \iota ( \ell^+ - j +1) /n} + \mathrm{Re} \sum_{j=1}^{\ell^-} e^{ - 2 \pi \iota ( \ell^- - j) /n} - (c\tau^2 - C/n) ( |\lambda^+| + |\lambda^-|)/n.
\end{align}
Note that with the consecutive set $\mathcal{C}_{j_0}$ as in \eqref{eq:concsum}, with $j_0 = -(\ell^--1)$ we have 
\begin{align} \label{eq:campanella}
\mathrm{Re} \sum_{j=1}^{\ell^+} e^{2 \pi \iota ( \ell^+ - j +1) /n} + \mathrm{Re} \sum_{j=1}^{\ell^-} e^{ - 2 \pi \iota ( \ell^- - j) /n} = \mathrm{Re} \sum_{ j \in \mathcal{C}_{j_0}} e^{ 2 \pi \iota j/n} \leq \frac{ \sin(\pi\ell/n)}{\sin(\pi/n)},
\end{align}
where in the final equality we have simply used \eqref{eq:concsum} and $\mathrm{Re}(z) \leq |z|$. 

Using \eqref{eq:campanella} in \eqref{eq:nara5},  we obtain 
\begin{align} \label{eq:nara6}
\mathrm{Re}  \sum_{j \in \mathcal{S}} e^{2 \pi \iota j/n} &\leq \frac{ \sin(\pi\ell/n)}{\sin(\pi/n)} - (c\tau^2 - C/n)( |\lambda^+| + |\lambda^-|)/n.
\end{align}
Finally, since $\mathcal{S}$ is centered, by definition the argument of $z := \sum_{j \in \mathcal{S}} e^{2 \pi \iota j/n}$ lies in $[-\pi/2n,3\pi/2n)$. Now using $|a+bi| \leq |a|+|b|$ together with the fact that $|\lambda^+|+|\lambda^-| \geq 1$ it follows that 
\begin{align} \label{eq:napa0}
\left| \sum_{j \in \mathcal{S}} e^{2 \pi \iota j/n} \right| &\leq \frac{ \sin(\pi\ell/n)}{\sin(\pi/n)} - (c\tau^2 - C_1/n)( |\lambda^+| + |\lambda^-|)/n,
\end{align}
with a possibly different constant $C_1$ replacing $C$ in \eqref{eq:nara6}. To obtain the result as written, using $0 \leq \tau := \ell/n \leq 1/2$, a brief calculation tells us that $\sin(\pi\ell/n)/\sin(\pi/n) \geq c_1 n \tau$ for some universal $c_1 > 0$. In particular, we obtain from \eqref{eq:napa0} 
\begin{align} \label{eq:napa}
\mathrm{Re}  \sum_{j \in \mathcal{S}} e^{2 \pi \iota j/n} &\leq \frac{ \sin(\pi\ell/n)}{\sin(\pi/n)}\left\{ 1  - \left(\frac{c_2\tau}{n^2} - \frac{C_2}{n^3 \tau}\right)( |\lambda^+| + |\lambda^-|  ) \right\},
\end{align}
for possibly different universal constants $c_2,C_2>0$. Using $1-x \leq e^{-x}$, together with the fact that $\tau \geq n^{-1/4}$ implies $\frac{c_2\tau}{n^2} - \frac{C_2}{n^3 \tau} \geq c \tau/n^2$, we obtain the result. 
\end{proof}

We are now ready to prove Theorem \ref{thm:cambridge springs}.

\begin{proof}[Proof of Theorem \ref{thm:cambridge springs}]
By \eqref{eq:volclass} we have
\begin{align} \label{eq:vol008}
\mathrm{Vol}^{(n)}_{k,\ell} = n \frac{(-1)^{k(\ell+1)}}{(nk)!} \left(  \sum_{\mathcal{S} \text{ good, cent.}} \left(  \sum_{j \in \mathcal{S}}e^{2 \pi \iota j/n}  \right)^{nk} + \sum_{\mathcal{S} \text{ bad, cent.} } \left(  \sum_{j \in \mathcal{S}}e^{2 \pi \iota j/n}   \right)^{nk} \right),
\end{align}
where the respective outer sums in \eqref{eq:vol008} are over the good and bad centered subsets of $\mathbb{Z}_n$ of cardinality $\ell$.

First we consider the contribution from good centered sets. For sufficiently large $n$, the good centered sets are in bijection with all pairs of partitions $(\lambda^+,\lambda^-)$ satisfying $|\lambda^+|,|\lambda^-|\leq \log n$. Using Lemma \ref{lem:shalom}, it then follows that 
\begin{align*}
\sum_{\mathcal{S} \text{ good, cent.}} \left( \sum_{j \in \mathcal{S}} e^{2 \pi \iota j/n} \right)^{nk} &= (-1)^{k(\ell-1)}  \left( \frac{ \sin( \pi \ell/n)}{ \sin(\pi/n) }\right)^{nk} \sum_{\lambda^+,\lambda^-} \left\{ e^{ - pq_\tau^+ |\lambda^-| - pq_{\tau}^- |\lambda^+|}  +O \left(\frac{p \log^2 n}{ \tau n } \right) \right\},
\end{align*}
where $\sum_{\lambda^+,\lambda^-}$ denotes a sum taken over all pairs $(\lambda^+,\lambda^-)$ of integer partitions with $|\lambda^+|,|\lambda^-| \leq \log n$. The number of integer partitions $|\lambda| \leq m$ is $e^{O(\sqrt{m})}$ \cite[Theorem 14.7]{apostol}. In particular, there are at most $e^{ O(\sqrt{\log n})}$ pairs $(\lambda^+,\lambda^-)$ of integer partitions satisfying $|\lambda^+|,|\lambda^-| \leq \log n$. Thus
\begin{align*}
\sum_{\mathcal{S} \text{ good, cent.}} \left(  \sum_{j \in \mathcal{S}}e^{2 \pi \iota j/n}  \right)^{nk} &=  (-1)^{k(\ell-1)}  \left( \frac{ \sin( \pi \ell/n)}{ \sin(\pi/n) }\right)^{nk} \left( \mathcal{P}_n( e^{-pq_\tau^+}) \mathcal{P}_n( e^{-pq_\tau^-} ) +R(n) \right),
\end{align*}
where $\mathcal{P}_n( e^{-q} ) := \sum_{ \lambda : |\lambda| \leq \log n} e^{-q |\lambda| } $, and $R(n) = O \left(\frac{p \log^2 n e^{O(\sqrt{\log n } ) } }{ \tau n } \right)$. Using $p,\tau \geq n^{-1/4}$ and taking a generous but simple bound, we have $R(n) =O \left(n^{-1/2} \right)$.  Letting $\mathcal{P}(e^{-q}) := \sum_{ \lambda } e^{-q |\lambda| }$ denote the associated sum without the $|\lambda| \leq \log n$ restriction, we have $\mathcal{P}_n(e^{-q}) = \mathcal{P}(e^{-q}) + O( e^{ - \frac{q}{2} n})$. Putting everything together, we conclude that
\begin{align} \label{eq:ora}
\sum_{\mathcal{S} \text{ good, cent.}} \left(  \sum_{j \in \mathcal{S}}e^{2 \pi \iota j/n}  \right)^{nk} &=  (-1)^{k(\ell-1)}  \left( \frac{ \sin( \pi \ell/n)}{ \sin(\pi/n) }\right)^{nk} \left( \mathcal{P}( e^{-pq_\tau^+} ) \mathcal{P}( e^{-pq_\tau^-} )  +O \left(pn^{-1/2} \right) \right).
\end{align}

We now control the contribution from bad centered sets. Any bad centered set $\mathcal{S}(\lambda^+,\lambda^-)$ has $\log n \leq |\lambda^+|+|\lambda^-| \leq n^2$. Using Lemma \ref{lem:kontrol} to obtain the first equality below, and then using the fact that there are most $e^{ O(\sqrt{m})}$ pairs of partitions $(\lambda^+,\lambda^-)$ with mass $|\lambda^+|+|\lambda^-| = m$ to obtain the second, we see that
\begin{align} \label{eq:ora2}
 \sum_{\mathcal{S} \text{ bad, cent.}} \left| \sum_{j \in \mathcal{S}} e^{2 \pi \iota j/n} \right|^{nk} &\leq \left( \frac{\sin (\pi \ell/n)}{ \sin(\pi /n)} \right)^{nk} \sum_{\log n \leq |\lambda^+|+|\lambda^-| \leq 2n^2} \exp \left\{ - c \tau p ( |\lambda^+| + |\lambda^-|)  \right\} \nonumber \\
&\leq \left( \frac{\sin (\pi \ell/n)}{ \sin(\pi /n)} \right)^{nk} \exp \left\{ - c \tau p \log (n) \right\} 
\end{align}
for some constant $c > 0$. 
By \eqref{eq:ora}, \eqref{eq:ora2}, and \eqref{eq:vol008}, we obtain
\begin{align} \label{eq:vol009}
\left( \frac{ e \sin( \pi \ell/n) }{ \pi } \right)^{-nk}k^{nk} \mathrm{Vol}^{(n)}_{k,\ell} = n \frac{(\frac{k}{e} \frac{\pi}{ \sin(\pi/n)} )^{nk}}{(nk)!} \left( \mathcal{P}( e^{-pc_\tau}) \mathcal{P}( e^{-pq_\tau^+} ) + O\left(e^{ - cp\tau \log(n) }\right) \right).
\end{align}

The result now follows from using Stirling's formula $(nk)! =(1+o(1)) \sqrt{2 \pi nk} (nk/e)^{nk}$, $k=pn$, and the fact that $\left( \frac{ \sin (\pi/n) }{ \pi /n} \right)^{nk} = (1+o(1))e^{ - p \pi^2/6}$.
\end{proof}

\section{Probability measures on bead configurations}  \label{sec:corr}

\subsection{Complex measures on bead processes} \label{sec:random bead}

Recall from the introduction that an occupation process on $\mathbb{T}_n$ is a right-continuous function $t \mapsto X_t \in \mathcal{P}(\mathbb{Z}_n)$ with finitely many discontinuities and the property that every discontinuity takes the form $X_t = X_{t-} \cup \{h+1\} - \{h\}$. In this case we say $(t,h)$ is a bead of $(X_t)_{t \in [0,1)}$. Let $\mathcal{A}_n$ denote the set of occupation processes on $\mathbb{T}_n$.  In this section we study measures on $\mathcal{A}_n$. 

Recall the functions $g_N^\lambda$ and $g_N^{\lambda,\theta}$ defined in Section \ref{sec:pf}. For $\Lambda = \lambda$ or $\Lambda = (\lambda,\theta)$ we define complex measures $\mathbf{Q}_n^{\Lambda,T}$ on occupation processes as follows. Since nonconstant occupation processes are determined by nonempty bead configurations, for $k \geq 1$, and any disjoint subsets $A_1,\ldots,A_{nk}$ of $\mathbb{T}_n$, let
\begin{align} \label{eq:po}
\mathbf{Q}_n^{\Lambda,T} ( nk \text{ beads}, A_1,\ldots,A_{nk} \text{ contain beads}  ) 
:= T^{nk} \int_{A_1 \times \cdots \times A_{nk}} g_{nk}^{\Lambda} ( y_1,\ldots,y_{nk}) \mathrm{d}y_1 \cdots \mathrm{d}y_{nk}.
\end{align}
As for constant occupation processes, we set
\begin{align} \label{eq:po2}
\mathbf{Q}_n^{\Lambda,T}(X_t = A ~\forall t \in[0,1) ) =
\begin{cases}
 e^{ - \lambda \# A} \qquad &\text{if $\Lambda = \lambda$},\\
 \frac{1}{2}(-1)^{(\theta_1+1)(\theta_2+n+\#A+1)}  e^{ - \lambda \# A}\qquad &\text{if $\Lambda = (\lambda,\theta)$}.
\end{cases}
\end{align}
Note that by summing \eqref{eq:po2} over all subsets $A$ of $\mathbb{Z}_n$, \eqref{eq:po2} implies 
\begin{align} \label{eq:po1}
\mathbf{Q}_n^{\Lambda,T}(X_t \text{ is constant}) = g_0^{\Lambda},
\end{align}
where $g_0^{\Lambda}$ is defined in \eqref{eq:crelle} and \eqref{eq:crelle2}.

In summary, $\mathbf{Q}_n^{\lambda,T}$ defines a complex measure on $(\mathcal{A}_n,\mathcal{F})$: the measure on occupation processes with $nk>0$ discontinuities is determined by \eqref{eq:po}, and the measure on constant occupation processes is determined by \eqref{eq:po2}. When $\lambda \in \mathbb{R}$, $\mathbf{Q}_n^{\lambda,T}$ takes nonnegative values, and $\mathbf{Q}_n^{\lambda,\theta,T}$ takes real values.

One can check using \eqref{eq:gsum} in the case $nk > 0$ (and using \eqref{eq:po2} in the no beads case --- the $\theta_1 =0$ terms cancel one another out) that we have the decomposition of measures
\begin{align} \label{eq:Qsum}
\mathbf{Q}_n^{\lambda,T} = \sum_{ \theta \in \{0,1\}^2 } \mathbf{Q}_n^{\lambda,\theta,T}.
\end{align}
Recall from Definition \ref{df:mixed} the definition of the mixed correlation functions. In particular, recall that we use the letters $b$, $o$, and $u$, as shorthand for `bead', `occupied' or `unoccupied'. With this in mind, to  discuss the probability and correlation functions of $\mathbf{Q}_n^{\Lambda,T}$ we define the \textbf{replicated torus} by 
\begin{align*}
\mathbb{T}_{n,*} := \{ b,o,u\} \times \mathbb{T}_n. 
\end{align*}
Given points $(w_1,\ldots,w_N)$ in $\mathbb{T}_{n,*}$ with $w_i = (\alpha_i,t_i,h_i)$, write $i \in \mathcal{B},\mathcal{O},\mathcal{U}$ according to whether $\alpha_i=b,o,u$. We will write $B := \#\mathcal{B}, O := \#\mathcal{O}, U := \# \mathcal{U}$. We now define the mixed indicator function $I:\mathbb{T}_{n,*}^N \to \mathbb{C}$ by setting, in the $\mathcal{B} \neq \varnothing$ case,
\begin{align} \label{eq:gext0}
I(w_1,\ldots,w_N) &:= \mathrm{1} \{ \text{$(y_i : i \in \mathcal{B})$ bead config., $y_i$ is occ./unocc.\ whenever $\alpha_i = o/u$} \}.
\end{align}
By definition, $I(w_1,\ldots,w_N)$ is nonzero only when $B = nk$ for some $k$.
For reasons that will become clear below, in the $\mathcal{B} = \varnothing$ case, we define
\begin{align} \label{eq:gext00}
I(w_1,\ldots,w_N) &:= \mathrm{1} \{ \{ h_i : i \in \mathcal{O} \} \cap \{ h_i : i \in \mathcal{U} \} = \varnothing \}.
\end{align}
We will also define $\ell_o=\ell_o(w_1,\ldots,w_N)$ and $\ell_u = \ell_u(w_1,\ldots,w_N)$ as follows. In the $\mathcal{B} \neq \varnothing$ case 
\begin{align} \label{eq:gext1}
\ell_o = \ell \qquad \text{and} \qquad \ell_u = n-\ell \qquad \text{if $\mathcal{B} \neq \varnothing$},
\end{align}
where $\ell = \ell(y_i : i \in \mathcal{B})$ is the occupation number of the underlying nonempty bead configuration, and otherwise
\begin{align} \label{eq:gext11}
\ell_o = \# \{ h_i : i \in \mathcal{O}\} \qquad \text{and} \qquad \ell_u = \# \{ h_i : i \in \mathcal{U}\} \qquad \text{if $\mathcal{B} = \varnothing$}.
\end{align}
With a view to constructing the mixed correlation functions, we would like to extend the definition of $g_N^\Lambda$ from $\mathbb{T}_n^N$ to $\mathbb{T}_{n,*}^N$ (i.e.\ $g_N^\Lambda:\mathbb{T}_{n,*}^N \to \mathbb{R}$) to encompass occupation, so that we may write
\begin{align*}
& T^B g_N^\Lambda(w_1,\ldots,w_N) \prod_{i \in \mathcal{B}} \mathrm{d}y_i\\
&:= \mathbf{Q}_n^{\Lambda,T} ( \text{$B$ beads}, \text{Beads in $\mathrm{d}y_i : i \in \mathcal{B}$}, \text{$y_i$ is occ./unocc.\ whenever $\alpha_i = o/u$} ).
\end{align*}
More explicitly, given disjoint subsets $A_1,\ldots,A_{nk}$ of $\mathbb{T}_n$, and points $(y_i : i \in \mathcal{O} \sqcup \mathcal{U})$ not lying in the $A_i$, $g_N^\Lambda(w_1,\ldots,w_N)$ are the unique symmetric functions satisfying
\begin{align} \label{eq:lift}
&\mathbf{Q}_n^{\Lambda,T} ( nk \text{ beads}, A_1,\ldots,A_{nk} \text{ contain beads} , \text{$y_i$ is occ./unocc.\ whenever $\alpha_i = o/u$}   ) \nonumber\\
&=  T^{nk}  \int_{A_1 \times \cdots \times A_{nk}} g_{N}^{\Lambda} ((b,y_1),\ldots,(b,y_{nk}),\{(o,y_i):i \in \mathcal{O}\}, \{(u,y_i):i \in \mathcal{U}\} ) \mathrm{d}y_1 \cdots \mathrm{d}y_{nk}.
\end{align}
By symmetric, we mean $g_N^\Lambda(w_{\sigma(1)},\ldots,w_{\sigma(N)}) = g_N^\Lambda(w_1,\ldots,w_N)$ for any permutation $\sigma$. 
Again, $g_N^\Lambda(w_1,\ldots,w_N)$ may only be nonzero when $B = \# \mathcal{B} = \# \{ i:\alpha_i=b \}$ is a multiple of $n$. 

The following lemma provides a direct characterisation of $g_N^\Lambda$:

\begin{lemma} 
With definitions as in \eqref{eq:gext0}, \eqref{eq:gext00}, \eqref{eq:gext1} and \eqref{eq:gext11} we have 
\begin{align} \label{eq:1gdef}
g_N^\lambda(w_1,\ldots,w_N) := e^{ - \lambda \ell_o} (1 + e^{-\lambda})^{n-\ell_o-\ell_u} I(w_1,\ldots,w_N),
\end{align}
and
\begin{align} \label{eq:gext}
g_N^{\lambda,\theta}(w_1,\ldots,w_N) &:=  \frac{1}{2} (-1)^{(\theta_1+k+1)(\theta_2+n+\ell_o+1)} e^{ - \lambda \ell_o} (1 - (-1)^{\theta_1} e^{-\lambda} )^{n- (\ell_o + \ell_u)}  I(w_1,\ldots,w_N) ,
\end{align}
where $k=B/n$ is the number of beads.
\end{lemma}
\begin{proof}
By \eqref{eq:Qsum}, we have $\sum_{\theta \in \{0,1\}^2} g_N^{\lambda,\theta} = g_N^\lambda$. Conversely, if we let $h_N^{\lambda}$ and $h_N^{\lambda,\theta}$ denote respectively the right-hand-sides of \eqref{eq:1gdef} and \eqref{eq:gext} we claim that 
\begin{align} \label{eq:castlin}
\sum_{\theta \in \{0,1\}^2} h_N^{\lambda,\theta} = h_N^\lambda.
\end{align}
To see that \eqref{eq:castlin} holds, in the $\mathcal{B} \neq \varnothing$ (i.e.\ $k>0$) case by definition we have $n-\ell_o-\ell_u=0$, and \eqref{eq:castlin} follows from the three-vs-one identity \eqref{eq:kid}. In the $\mathcal{B}=\varnothing$ case, the terms with $\theta_1 = 0$ cancel out, thus proving \eqref{eq:castlin}. 

Thus it remains to establish \eqref{eq:gext}, since this would also imply \eqref{eq:1gdef}.

In the $\mathcal{B} \neq \varnothing$ case, we have $n- \ell_o - \ell_u = 0$, and \eqref{eq:gext} follows immediately from the definition of $\mathbf{Q}_n^{\lambda,\theta,T}$ in \eqref{eq:po}. 

As for the case where $\mathcal{B} = \varnothing$, by \eqref{eq:po2}, we have
\begin{align*}
 g_N^{\lambda,\theta}(w_1,\ldots,w_N) &:= \mathbf{Q}_n^{\lambda,\theta,T}( X_t \text{ constant},  \text{$y_i$ is occ./unocc.\ whenever $\alpha_i = o/u$} ) \\
&= \sum_{A \subseteq \mathbb{Z}_n} \frac{1}{2}(-1)^{(\theta_1+1)(\theta_2+n+\#A+1)}  e^{ - \lambda \# A} \mathrm{1} \{ Z_O \subseteq A, Z_U \subseteq \mathbb{Z}_n - A \},
\end{align*}
where $Z_O := \{ h_i : i \in\mathcal{O} \}$ and $Z_U := \{ h_i : i \in \mathcal{U} \}$. 
We note that the indicator $\mathrm{1} \{ Z_O \subseteq A, Z_U \subseteq \mathbb{Z}_n - A \}$ may only be nonzero for some $A$ if $Z_O$ and $Z_U$ are disjoint.
Note $\# Z_O = \ell_o$, $\# Z_U = \ell_u$. 
In the setting of the previous line, we may decompose $A$ as the disjoint union $A = Z_O \sqcup B$, where $B \subseteq \mathbb{Z}_n - Z_O - Z_U$. Performing the sum over such $B$, we have
\begin{align*}
 & g_N^{\lambda,\theta}(w_1,\ldots,w_N)  \nonumber \\
&= \frac{1}{2}(-1)^{(\theta_1+1)(\theta_2+n+1)}  (-1)^{(\theta_1+1)\ell_o}  e^{ - \lambda \ell_o} \mathrm{1}_{Z_O \cap Z_U = \varnothing} \sum_{B \subseteq \mathbb{Z}_n-Z_U-Z_O} (-1)^{(\theta_1+1) \# B } e^{ - \lambda \# B } \nonumber \\
&= \frac{1}{2} (-1)^{(\theta_1+1)(\theta_2+n+\ell_o+1)} (1 - (-1)^{\theta_1} e^{-\lambda} )^{n- (\ell_o + \ell_u)} e^{ - \lambda \ell_o} I(w_1,\ldots,w_N),
\end{align*}
proving \eqref{eq:gext} in the case $\mathcal{B} = \varnothing$. 
\end{proof}

Note that the functions $g_N^\lambda:\mathbb{T}_{n,*}^N \to \mathbb{C}$ and $g_N^{\lambda,\theta}:\mathbb{T}_{n,*}^N \to \mathbb{C}$ defined in \eqref{eq:1gdef} and \eqref{eq:gext} are extensions of $g_N^\lambda:\mathbb{T}_{n}^N \to \mathbb{C}$ and $g_N^{\lambda,\theta}:\mathbb{T}_{n}^N \to \mathbb{C}$ defined in \eqref{eq:0gdef} and \eqref{eq:0gthetadef} in the sense that
\begin{align*}
g_N^\Lambda((b,y_1),\ldots,(b,y_N)) = g_N^\Lambda(y_1,\ldots,y_N),
\end{align*}
for either $\Lambda = \lambda$ or $\Lambda = (\lambda,\theta)$.

\subsection{Determinantal representation for $g_N^{\lambda,\theta}$}

The main result of this section states that $g_N^{\lambda,\theta}$ may be represented in terms of determinants involving the operator $C^{\beta,\theta_2}:\mathbb{T}_{n,*} \times \mathbb{T}_{n,*} \to \mathbb{C}$ defined by
\begin{align} \label{eq:Creplica}
C^{\beta,\theta_2}(w,w') := (-1)^{ \theta_2 \mathrm{1}_{\{\alpha = b, h' = 0\}} + \mathrm{1}_{\{\alpha = o\}} } \mathrm{1}_{\{h'=h+\mathrm{1}_{\{\alpha = b\}}\}} \frac{ e^{ - \beta [ t'-t]_\alpha }}{ 1 - e^{ - \beta}} ,
\end{align}
where $w = (\alpha,t,h)$, $w' = (\alpha',t',h')$, and $[t'-t]_\alpha = t' - t + \mathrm{1}_{\{t' < t\}} + \mathrm{1}_{\{\alpha = o, t'=t\}}$. Note that $C^{\beta,\theta_2}(w,w')$ does not depend on $\alpha'$. 

\begin{thm} \label{thm:replicacast} We have 
\begin{align*}
(-1)^{(\theta_1+1)(\theta_2+n+1)} (1-e^{-(\lambda+\theta_1 \pi \iota)})^n \det_{i,j=1}^N C^{\lambda+\theta_1 \pi \iota ,\theta_2}(w_i,w_j) = g_N^{\lambda,\theta}(w_1,\ldots,w_N).
\end{align*}
\end{thm}

Our first step in the proof of Theorem \ref{thm:replicacast} is a generalisation of Lemma \ref{lem:new det} involving the operator $A^\zeta:\mathbb{T}_{n,*}\times \mathbb{T}_{n,*} \to \mathbb{C}$ given by 
\begin{align} \label{eq:blackfall}
A^\zeta(w,w') := (-1)^{ \mathrm{1}_{\{\alpha = o\}}} \mathrm{1}_{\{h'=h+\mathrm{1}_{\{\alpha = b\}}\}} \zeta^{ \mathrm{1}_{\{t'<t\}} + \mathrm{1}_{\{\alpha =o, t' = t\}} }. 
\end{align} 
Like $C^{\beta,\theta}(w,w')$, $A^\zeta(w,w')$ does not depend on $\alpha'$. Note that with $A^\zeta$ defined in \eqref{eq:aux1} as an operator on $\mathbb{T}_n$ rather than $\mathbb{T}_{n,*}$, we have $A^\zeta( (b,y),(b,y')) = A^\zeta(y,y')$. Thus \eqref{eq:blackfall} defines an extension of the operator $A^\zeta$ defined in \eqref{eq:aux1}. 

Our next lemma says that determinants involving $A^\zeta$ may be used to detect bead configurations as well as the associated occupancy and vacancy.
\begin{lemma} \label{lem:new det2}
Let $\zeta \neq 1$. 
Let $w_1,\ldots,w_N$ be elements of $\mathbb{T}_{n,*}$ such that if $w_i = (\alpha_i,t_i,h_i)$, we have $t_1,\ldots,t_N$ distinct. Again write $\mathcal{B}/\mathcal{O}/\mathcal{U} := \{ i : \alpha_i = b/o/u\}$ respectively. Then with the definitions \eqref{eq:gext0}, \eqref{eq:gext00}, \eqref{eq:gext1} and \eqref{eq:gext11} we have
\begin{align} \label{eq:xan}
\det_{i,j=1}^N A^\zeta(w_i,w_j) &= (-1)^{k(n-1)+(k-1)\ell_o}  ( 1 - \zeta)^{N-(\ell_o+\ell_u)}\zeta^{\ell_o} I(w_1,\ldots,w_N),
\end{align}
where $k = \# \mathcal{B}/n$. 
\end{lemma}
\begin{proof}
Note that $(A^\zeta(w_i,w_j))_{i,j =1}^N$ is a block matrix: $A^\zeta(w_i,w_j)$ may only be nonzero if, for some $h \in \mathbb{Z}_n$, we have both $h_i+\mathrm{1}_{\alpha_i=b} = h$ and $h_j = h$.
The determinant of $(A^\zeta(w_i,w_j))_{i,j =1}^N$ may then only be nonzero if each of these blocks is a square, i.e.\ for each $h \in \mathbb{Z}_n$ we have
\begin{align*}
p_h = \# \{ i: h_i+\mathrm{1}_{\alpha_i=b} = h \} = \# \{ j : h_j = h \}.
\end{align*}
This happens if and only if there exists some $k \geq 0$ such that for each $h \in \mathbb{Z}_n$ we have $\# \{ i : \alpha_i = b, h_i = h \} = k$, which we assume is the case for the remainder of the proof. Write
\begin{align*}
O_h := \# \{ i : \alpha_i = o, h_i = h\} \qquad \text{and} \qquad U_h := \# \{ i : \alpha_i = u, h_i = h\}.
\end{align*}
Then $p_h = k + O_h + U_h$.

Suppose now w.l.o.g.\ that the $w_i$ are ordered lexicographically in the $[h,\alpha, t]$ order: i.e.\ first by string $h_i$ in increasing order, then by $\alpha_i$ in $b,o,u$ order, then by $t_i$. By a similar argument to the one used in the proof of Lemma \ref{lem:new det} it takes $(-1)^{k(n-1)}$ row swaps to make the matrix block diagonal, and thus,
\begin{align} \label{eq:totaldet}
\det_{i,j=1}^N A^\zeta(w_i,w_j) = (-1)^{k(n-1)} \prod_{ h \in \mathbb{Z}_n} \det B_h,
\end{align}
where $B_h = ( A^\zeta(w_i,w_j) )_{i : h_i +\mathrm{1}_{\alpha_i =b} = h, j : h_j = h }$. Note that by the definition \eqref{eq:blackfall} of $A^\zeta$ we have
\begin{align*}
\det B_h = (-1)^{ O_h } \det_{i : h_i + \mathrm{1}_{\alpha_i = b } = h , j: h_j = h } \zeta^{ \mathrm{1}\{ t_j < t_i \} + \mathrm{1}\{ \alpha_i = o\} \mathrm{1}\{t_j =t_i\} } .
\end{align*}
where $O_h := \# \{ i \in \mathcal{O} : h_i = h \}$. 
With an eye on using Lemma \ref{lem:succession}, since the $(t_i)$ are distinct we may instead write
\begin{align*}
\det B_h = (-1)^{ O_h } \det_{i : h_i + \mathrm{1}_{\alpha_i = b } = h , j: h_j = h } \zeta^{ \mathrm{1}\{ t_j < t_i - \varepsilon  \mathrm{1}_{\alpha_j = o} \} }
 = (-1)^{ O_h } \det_{i,j=1}^{p_h} \zeta^{ \mathrm{1}\{ r'_j < r_i \} } 
\end{align*}
for some sufficiently small $\varepsilon>0$, where in the final expression above we have introduced the notation
\begin{align*}
\{ r_1,\ldots,r_{p_h} \} := \{ t_i : h_i + \mathrm{1}_{\alpha_i = b} = h \} \qquad \text{and} \qquad \{ r'_1,\ldots,r'_{p_h} \} := \{ t_j - \mathrm{1}_{\alpha_j = o} \varepsilon : h_j = h \}.
\end{align*}
Now according to Lemma \ref{lem:succession}, the determinant $\det_{i,j=1}^{p_h} \zeta^{ \mathrm{1}_{r'_j < r_i} } $ may only be nonzero if we have either 
\begin{align*}
s_1 \leq s_1' < \cdots < s_{p_h} \leq s_{p_h}' \qquad \text{or} \qquad s_1' < s_2 \leq s_2' < \cdots \leq s_{p_h}' < s_{p_h},
\end{align*}
where $s_i = r_{\sigma(i)}$ and $s_i' = r'_{\sigma'(i)}$ are $(r_i)$ and $(r_i')$ listed in increasing order. 
 
\begin{figure}[h!]
\centering
\begin{tikzpicture}[scale=1.2]

\draw[lightgray]  (-0.5,-0.4) -- (-0.5,1.4);
\draw[lightgray]   (9.5,-0.4) -- (9.5,1.4);

\draw[lightgray]  (-0.5,0) -- (9.5,0);
\draw[lightgray]  (-0.5,1) -- (9.5,1);

\draw [very thick, CadetBlue] (0.42,0) circle [radius=0.1];
\draw [CadetBlue] (0.42,0) -- (0.42,1); 

\draw [very thick, CadetBlue] (3.3,0) circle [radius=0.1];
\draw [CadetBlue] (3.3,0) -- (3.3,1); 

\draw [very thick, CadetBlue] (8.31,0) circle [radius=0.1];
\draw [CadetBlue] (8.31,0) -- (8.31,1); 

\draw [line width=1.2mm, CadetBlue] (0.42,1) -- (1.85,1); 
\draw [line width=1.2mm, CadetBlue] (3.3,1) -- (5.1,1); 
\draw [line width=1.2mm, CadetBlue] (8.31,1) -- (9.0,1); 

\draw [very thick, CadetBlue] (1.95,1) circle [radius=0.1];
\draw [CadetBlue] (1.95,1) -- (1.95,1.4); 
\draw [very thick, CadetBlue] (5.21,1) circle [radius=0.1];
\draw [CadetBlue] (5.21,1) -- (5.21,1.4); 
\draw [very thick, CadetBlue] (9.1,1) circle [radius=0.1];
\draw [CadetBlue] (9.1,1) -- (9.1,1.4); 

\node at (1.33,1) [very thick, rectangle,draw,red] (v100) {};
\node at (1.40,0) [very thick, rectangle,draw,red] (v100) {};

\node at (4.33,1) [very thick, rectangle,draw,red] (v100) {};
\node at (4.40,0) [very thick, rectangle,draw,red] (v100) {};

\node at (7.12,1) [very thick, rectangle,draw,blue] (v100) {};
\node at (7.12,0) [very thick, rectangle,draw,blue] (v100) {};

\node at (5.92,1) [very thick, rectangle,draw,blue] (v100) {};
\node at (5.92,0) [very thick, rectangle,draw,blue] (v100) {};

\node at (2.62,1) [very thick, rectangle,draw,blue] (v100) {};
\node at (2.62,0) [very thick, rectangle,draw,blue] (v100) {};

\node at (-2,1) [black] (v109) {String $h$};
\node at (-2,0) [black] (v108) {String $h-1$};

\end{tikzpicture}
\caption{For each $w_i$ of the form $(b,t_i,h_i)$ with $h_i = h$ or $h-1$, we have plotted a purple circle at $(t_i,h_i)$. For each $w_i$ of the form $(o,t_i,h)$, we have plotted points $(t_i-\varepsilon,h)$ and $(t_i,h-1)$; these are depicted with red squares. For each $w_i$ of the form $(t_i,h)$, we have plotted points at $(t_i,h-1)$ and $(t_i,h)$; these are depicted with blue squares. Taking the squares and circles together, the points on string $h$ interlace those on string $h-1$ if and only if the beads (purple circles) on string $h$ and $h-1$ interlace, the red squares on string $h$ lie in occupied regions, and the blue squares on string $h$ lie in unoccupied regions.}
\label{fig:interlacement}
\end{figure}
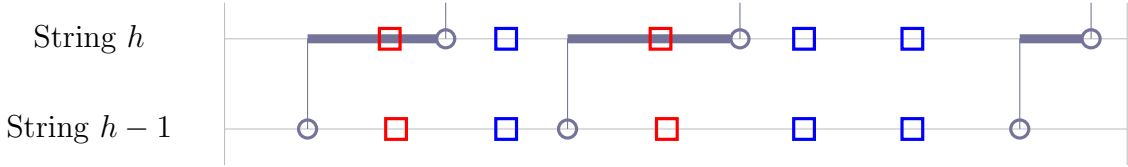

We now study the determinant $\det_{i,j=1}^{p_h} \zeta^{ \mathrm{1}\{r'_j < r_i \} }$ separately in the cases $\mathcal{B} \neq \varnothing$ and $\mathcal{B} = \varnothing$.
\vspace{3mm}

\textbf{Case 1: $\mathcal{B} \neq \varnothing$.}\\
If $\mathcal{B} \neq \varnothing$, then $\# \{ i : \alpha_i = b, h_i = h \} =k>0$. 
After some consideration (the reader may want to consult the example diagram in Figure \ref{fig:interlacement}), interlacing of $s_i$ and $s_i'$ occurs if and only if each of the following occurs:
\begin{itemize}
\item The beads $\{ t_i : \alpha_i = b, h_i = h-1 \}$ on string $h-1$ interlace the beads $\{ t_j : \alpha_j = b, h_j = h \}$ on string $h$.
\item Each $t_i$ with $\alpha_i = o$ and $h_i = h$ is preceded by a bead on string $h-1$ and succeeded by a bead on string $h$. Equivalently, $(t_i,h_i)$ is occupied.
\item Each $t_i$ with $\alpha_i = u$ and $h_i = h$ is succeeded by a bead on string $h-1$ and preceded by a bead on string $h$. Equivalently, $(t_i,h_i)$ is unoccupied.
\end{itemize}
Writing $\Gamma_h$ for the intersection of these three events, it follows that $I(w_1,\ldots,w_N) = \prod_{h \in \mathbb{Z}_n} \mathrm{1}_{\Gamma_h}$. 

It follows from Lemma \ref{lem:succession} that
\begin{align} \label{eq:detAh}
 \det B_h =  (-1)^{ O_h }  \det_{i,j=1}^{p_h} \zeta^{ \mathrm{1} \{r'_j < r_i \} } =  (-1)^{ O_h }  \mathrm{sgn}(\sigma)\mathrm{sgn}(\sigma') ((-1)^{p_h-1}\zeta)^{I_h} (1 - \zeta)^{p_h-1} \mathrm{1}_{\Gamma_h},
\end{align}
where $I_h$ is the indicator function of the event that the smallest $r'_j$ is strictly less than the smallest $r_i$, and $\sigma$ and $\sigma'$ are the permutations required to list $r_i$ and $r_i'$ in increasing order. 

We now take a look at $I_h$ and $\sigma, \sigma'$ on the event $\Gamma_h$.

First we claim that on $\Gamma_h$, $I_h$ is determined by the position of the first \emph{bead} on either string $h-1$ or $h$, i.e.\ $I_h$ is determined by $\{ t_i : i \in \mathcal{B} \text{ with } h_i \in \{h-1,h\} \}$. Indeed, if the first bead on string $h$ (at a point $(t,h)$ say) is strictly before the first bead on string $h-1$, then on $\Gamma_h$ any other $w_i = (\alpha_i,t_i,h)$ with $t_i < t$ must be occupied, and thus since $t'_i = t_i - \varepsilon < t_i$, we have $I_h = 1$. Otherwise, the first bead on string $h-1$ (at a point $(t,h-1)$ say) lies nonstrictly before any bead on string $h$, and consequently on $\Gamma_h$, any point of the form $w_i = (\alpha_i,t_i,h_i)$ with $t_i < t$ must be unoccupied, and hence we have $t_i = t_i'$, so that $I_h = 0$. 

In short, on $\Gamma_h$, $I_h$ is determined by the beads in that 
\begin{align*}
I_h = 1 \iff \inf_{ i : h_i = h, \alpha_i = b } t_i < \inf_{ i : h_i = h-1, \alpha_i = b}.
\end{align*}
In particular, using the definition \eqref{eq:new} of the occupation number $\ell(y_i : i \in \mathcal{B})$, and then \eqref{eq:gext0}, we have
\begin{align} \label{eq:modular}
\sum_{h \in \mathbb{Z}_n} I_h = \ell(y_i : i \in \mathcal{B}) = \ell_0.
\end{align}
We turn to studying $\mathrm{sgn}(\sigma)\mathrm{sgn}(\sigma')$, which is only defined on the event $\Gamma_h$, and which depends on $I_h$. Recall that we have assumed w.l.o.g.\ that the $w_i$ are ordered lexicographically, and with their horizontal coordinates in increasing order.

Considering first the case where $\{ I_h = 1 \}$, on this event, to the left of each $w_i$ of the form $(o,t_i,h)$, since $(t_i,h_i)$ is occupied, there lie the same number of beads on string $h$ and string $h-1$. On the other hand, on $\{I_h =1 \}$, to the left of each $w_i$ of the form $(u,t_i,h)$, there lies one more bead on string $h$ than on string $h-1$. Thus the signs of the reorderings $\sigma$ and $\sigma'$ satisfy
\begin{align*}
\mathrm{sgn}(\sigma)\mathrm{sgn}(\sigma') = (-1)^{U_h} := (-1)^{\# \{ i: \alpha_i = u, h_i = u \}} \qquad \text{when $I_h = 1$}.
\end{align*}
Conversely, on $\{I_h = 0 \}$, each $(o,t_i,h)$ is preceded by one more bead on string $h-1$ than on string $h$, where as each $(u,t_i,h)$ is preceded by the same number of beads on string $h-1$ as on string $h$, and thus
\begin{align*}
\mathrm{sgn}(\sigma)\mathrm{sgn}(\sigma') = (-1)^{O_h} := (-1)^{\# \{ i: \alpha_i = u, h_i = u \}} \qquad \text{when $I_h = 0$}.
\end{align*}
Thus, since $p_h = k + O_h + U_h$ is the number of $w_i$ with $h_i = h$, we may write
\begin{align*}
\mathrm{sgn}(\sigma)\mathrm{sgn}(\sigma') = (-1)^{ I_h (p_h-k-O_h) + (1-I_h) O_h}.
\end{align*}
Plugging this into \eqref{eq:detAh} we obtain
\begin{align} \label{eq:detAh2}
 \det B_h =  (-1)^{ O_h }  \det_{i,j=1}^{p_h} \zeta^{ \mathrm{1}_{t'_j < t_i} } &=  (-1)^{ O_h + I_h (p_h-k-O_h) + (1-I_h)O_h +I_h (p_h-1)  }\zeta^{I_h} (1 - \zeta)^{p_h-1} \mathrm{1}_{\Gamma_h} \nonumber \\
&= (-1)^{ I_h(k-1)} \zeta^{I_h} (1 - \zeta)^{p_h-1} \mathrm{1}_{\Gamma_h}.
\end{align}
Plugging \eqref{eq:detAh2} into \eqref{eq:totaldet}, using \eqref{eq:modular} and the fact that $N = \sum_{h \in \mathbb{Z}_n} p_h$, and recalling that $\ell_o+\ell_u = n$ in the case $\mathcal{B} \neq \varnothing$, we obtain the result in the case $\mathcal{B} \neq \varnothing$.\\
\vspace{3mm}

\textbf{Case 2: $\mathcal{B} = \varnothing$.}\\
In the case where $\mathcal{B} = \varnothing$, we may have $p_h = 0$ for some of the $h$ in $\mathbb{Z}_n$.

Suppose $p_h \neq 0$. Then the determinant $\det_{i,j=1}^{p_h} \zeta^{ \mathrm{1} \{r'_j < r_i \} }$ associated with $B_h$ is nonzero if and only if for those $i$ for which $h_i = h$, we have every $\alpha_i$ is $o$ or every $\alpha_i$ is $u$. 

It follows using Lemma \ref{lem:succession} that if $p_h \neq 0$ we have 
\begin{align*}
 \det B_h =  (-1)^{ O_h }  \det_{i,j=1}^{p_h}\zeta^{ \mathrm{1} \{r'_j < r_i \} }  = (-1)^{O_h + J_h (p_h-1) } \zeta^{J_h}( 1-\zeta)^{p_h-1} (J_h +J'_h)
\end{align*}
where $J_h = \mathrm{1}\{ \alpha_i = o~ \forall ~ i : h_i = h\}$ and $J'_h = \mathrm{1}\{ \alpha_i = u ~\forall ~ i : h_i = h\}$. Note that if $J_h = 1$, $p_h = O_h$ and $U_h = 0$, whereas if $J_h' =1$, $p_h = U_h$ and $O_h = 0$. Thus by considering the cases separately we see that
\begin{align*}
O_h + J_h ( p_h-1) = J_h \qquad \text{mod $2$}.
\end{align*}
Thus $\det B_h = (-\zeta)^{ J_h }( 1-\zeta)^{p_h-1} (J_h +J'_h)$.

Now note that by definition \eqref{eq:gext11}, $\ell_o = \sum_{h \in \mathbb{Z}_n} J_h$ and $\ell_u = \sum_{h \in \mathbb{Z}_n}J_h'$. Plugging this into \eqref{eq:totaldet}, in the case $\mathcal{B}= \varnothing$ we obtain
\begin{align*}
\det_{i,j=1}^N A^\zeta(w_i,w_j) &= \prod_{h \in \mathbb{Z}_n : p_h >0 } (-\zeta)^{J_h} (1 - \zeta)^{p_h-1}(J_h + J_h')\\
&= (-\zeta)^{\ell_o} (1 - \zeta)^{N - (\ell_o+\ell_u)} I(w_1,\ldots,w_N) ,
\end{align*}
where $\ell_o$ and $\ell_u$ are defined in \eqref{eq:gext11}. Since $\mathcal{B} = \varnothing$ is equivalent to $k=0$, this agrees with \eqref{eq:xan} in this case, thereby completing the proof.
\end{proof}

Our next result is now a generalisation of Lemma \ref{lem:premelan}. Define the operator $B^{\zeta,\theta}:\mathbb{T}_{n,*} \times \mathbb{T}_{n,*} \to \mathbb{C}$ by 
\begin{align} \label{eq:corolla}
B^{\zeta,\theta}(w,w') := \frac{(-1)^{\theta_2 \mathrm{1}_{\alpha=b,h'=0} }}{1 - (-1)^{\theta_1}\zeta} A^{(-1)^{\theta_1}\zeta}(w,w').
\end{align}

\begin{lemma} \label{lem:appler}
With $\zeta = e^{ - \lambda}\neq 1$ we have
\begin{align*}
(-1)^{(\theta_1+1)(\theta_2+n+1)} (1-(-1)^{\theta_1}\zeta)^n \det_{i,j=1}^N B^{\zeta,\theta}(w_i,w_j) = g_N^{\lambda,\theta}(w_1,\ldots,w_N).
\end{align*}
\end{lemma}
\begin{proof}
Using the definition \eqref{eq:corolla} in conjunction with Lemma \ref{lem:new det2} we obtain
\begin{align*}
&(-1)^{(\theta_1+1)(\theta_2+n+1)} (1-(-1)^{\theta_1}\zeta)^n \det_{i,j=1}^N B^{\zeta,\theta}(w_i,w_j) \\
&= (-1)^{(\theta_1+1)(\theta_2+n+1) + k(n-1) + \ell_o(k-1) + \theta_1 \ell_o + \theta_2 k } (1-(-1)^{\theta_1}\zeta)^{n-\ell_o-\ell_u} \zeta^{\ell_o} I(w_1,\ldots,w_N).
\end{align*}
The result now follows from \eqref{eq:modness} and \eqref{eq:gext}.
\end{proof}
We are now ready to complete the proof of Theorem \ref{thm:replicacast}.

\begin{proof}[Proof of Theorem \ref{thm:replicacast}]
A calculation tells us that with $\zeta = e^{ - \lambda}$ we have 
\begin{align*}
C^{\lambda+\theta_1 \pi \iota ,\theta_2}(w,w') = e^{ - (\lambda+\theta_1\pi \iota) (t'-t)} B^{\zeta,\theta}(w,w').
\end{align*}
Now note that since $C^{\lambda+\theta_1 \pi \iota ,\theta_2}$ is a diagonal conjugation of $B^{\zeta,\theta}$, we have \[ \det_{i,j=1}^N B^{\zeta,\theta}(w_i,w_j)  =  \det_{i,j=1}^N C^{\lambda+\theta_1 \pi \iota ,\theta_2}(w_i,w_j).\] Now use Lemma \ref{lem:appler}.
\end{proof}

\subsection{Complex probability measures and correlation functions}
In light of \eqref{eq:umbrella}, \eqref{eq:umbrellatheta}, and definitions of $\mathbf{Q}_n^{\Lambda,T}$, the total mass of $\mathbf{Q}_n^{\Lambda,T}$ is $Z(\Lambda,T)$, where
\begin{align*}
Z(\Lambda,T) := 
\begin{cases}
Z(\lambda,T) \qquad &\text{if $\Lambda = \lambda$}\\
Z^\theta(\lambda,T) \qquad &\text{if $\Lambda = (\lambda,\theta)$}.
\end{cases}
\end{align*}
Thus, provided $\lambda,T \in \mathbb{C}$ are such that $Z(\Lambda,T)$ is nonzero, we may define complex measures $\mathbf{P}_n^{\Lambda,T}$ of unit total mass by setting $\mathbf{P}_n^{\Lambda,T} := Z(\Lambda,T)^{-1}\mathbf{Q}_n^{\lambda,T}$, or more explicitly
\begin{align} \label{eq:scalar}
\mathbf{P}_n^{\lambda,T} := \frac{1}{Z(\lambda,T)} \mathbf{Q}_n^{\lambda,T} \qquad \text{and} \qquad \mathbf{P}_n^{\lambda,\theta,T} := \frac{1}{Z^\theta(\lambda,T)} \mathbf{Q}_n^{\lambda,\theta,T}.
\end{align}
That is, $\mathbf{P}_n^{\lambda,T}$ (resp.\ $\mathbf{P}_n^{\lambda,\theta,T}$) is simply $\mathbf{Q}_n^{\lambda,T}$ (resp.\ $\mathbf{Q}_n^{\lambda,\theta,T}$) multiplied by a scalar. When $\lambda \in \mathbb{R}$, $g_N^\lambda$ takes non-negative values, and consequently, $\mathbf{P}_n^{\lambda,T}$ is a genuine probability measure in that it takes values in $[0,1]$. When $\lambda \in \mathbb{R}$, each $\mathbf{P}_n^{\lambda,\theta,T}$ is only a signed measure of unit total mass. 

By \eqref{eq:Qsum}, we can write $\mathbf{P}_n^{\lambda,T}$ as an affine combination of the measures $\mathbf{P}_n^{\lambda,\theta,T}$ by setting
\begin{equation} \label{eq:Psum}
\mathbf{P}_n^{\lambda,T} = \sum_{ \theta \in \{0,1\}^2 } \frac{Z^\theta(\lambda,T)}{Z(\lambda,T)} \mathbf{P}_n^{\lambda,\theta,T}.
\end{equation}
Recall the mixed correlation functions defined in Definition \ref{df:mixed}, and the functions $g_N^\Lambda$ defined in \eqref{eq:lift}. For either $\Lambda = \lambda$ or $\Lambda = ( \lambda,\theta)$, by summing over all possible configurations, these may be written explicitly in terms of the integral formula
\begin{align} 
&p_N^{\Lambda,T}(w_1,\ldots,w_N)  &:= \frac{1}{Z(\Lambda,T)} \sum_{ m \geq 0 } \frac{T^{B+m}}{m!} \int_{\mathbb{T}_n^{m}} g_{N+m}^{\lambda}(w_1,\ldots,w_N, (b,y_{1}),\ldots,(b,y_{m}))  \mathrm{d}y_1 \cdots \mathrm{d}y_{m},\label{eq:coalmine}
\end{align}
where the factor of $1/m!$ is due to the various ways of ordering the unspecified beads, and we note that the sum is supported on those $m$ for which $m+B $ is a multiple of $n$. (Here $B= \# \{i:\alpha_i=b\}$.)

We note the integral in \eqref{eq:coalmine} is taken over $\mathbb{T}_n$. Our next lemma states that the correlation functions may instead be written as an integral over $\mathbb{T}_{n,*}$.
Integration on $\mathbb{T}_{n,*} = \{b,o,u\} \times [0,1) \times \mathbb{Z}_n$ is taken with respect to the natural Lebesgue measure, so that for $0 \leq t < u \leq 1$, $\{\alpha\} \times [t,u) \times \{h\}$ has mass $u-t$.

\begin{lemma} \label{lem:overcharge}
We have
\begin{align*}
p_N^{\Lambda,T}(w_1,\ldots,w_N) :=\frac{e^{-Tn}}{Z(\Lambda,T)} \sum_{ m \geq 0 } \frac{T^{B+m}}{m!} \int_{\mathbb{T}_{n,*}^{m}} g_{N+m}^{\Lambda}(w_1,\ldots,w_N, \zeta_1,\ldots,\zeta_m )  \mathrm{d}\zeta_{1} \cdots \mathrm{d}\zeta_{m}.
\end{align*}
\end{lemma}
\begin{proof}
Expanding the integral on the right-hand-side we have 
\begin{align*}
&\int_{\mathbb{T}_{n,*}^{m}} g_{N+m}^{\Lambda}(w_1,\ldots,w_N, \zeta_1,\ldots,\zeta_m )  \mathrm{d}\zeta_{1} \cdots \mathrm{d}\zeta_{m}\\
&= \sum_{m_b+m_o+m_u = m} \frac{m!}{m_b!m_o!m_u!} \int_{(\{b\} \times \mathbb{T}_n)^{m_b}}\mathrm{d}\zeta_1 \cdots \mathrm{d}\zeta_{m_b}  \int_{(\{o\} \times \mathbb{T}_n)^{m_o}} \mathrm{d}\zeta'_1 \cdots \mathrm{d}\zeta'_{m_o}  \\
&\times\int_{(\{u\} \times \mathbb{T}_n)^{m_u}} \mathrm{d}\zeta''_1 \ldots\mathrm{d}\zeta''_{m_u}~  g_{N+m}^{\Lambda}(w_1,\ldots,w_N, \zeta_1,\ldots,\zeta_{m_b}, \zeta_1',\ldots,\zeta_{m_o}', \zeta_1'',\ldots,\zeta_{m_u}'' ) \\
&= \sum_{m_b+m_o+m_u = m} \frac{m!}{m_b!m_o!m_u!} \int_{(\{b\} \times \mathbb{T}_n)^{m_b}}\mathrm{d}\zeta_1 \cdots \mathrm{d}\zeta_{m_b} g_{N+m_b}^{\Lambda}(w_1,\ldots,w_N, \zeta_1,\ldots,\zeta_{m_b})  (T\ell)^{m_o}(T(n-\ell))^{m_u},
\end{align*}
where the final equality follows from the fact that $g_{N+m}^{\Lambda}$ behaves as an indicator function in $y_i$ with $i \in \mathcal{O}$ being occupied, $y_i$ with $i \in \mathcal{U}$ being unoccupied.

In particular, by summing over $m_o$ and $m_u$ we obtain
\begin{align} \label{eq:canary}
&\sum_{m \geq 0} \frac{T^{B+m}}{m!} \int_{\mathbb{T}_{n,*}^{m}} g_{N+m}^{\Lambda}(w_1,\ldots,w_N, \zeta_1,\ldots,\zeta_m )  \mathrm{d}\zeta_{1} \cdots \mathrm{d}\zeta_{m}\nonumber \\
&= e^{Tn}  \sum_{m_b \geq 0} \frac{T^{B+m_b}}{m_b!} \int_{(\{b\} \times \mathbb{T}_n)^{m_b}}  \mathrm{d}\zeta_{1} \cdots \mathrm{d}\zeta_{m_b} g_{N+m_b}^{\Lambda}(w_1,\ldots,w_N, \zeta_1,\ldots,\zeta_{m_b} ) .
\end{align}
Plugging \eqref{eq:canary} into \eqref{eq:coalmine} we obtain the result.
\end{proof}

\subsection{Inversion}

Recall that $[t'-t]_\alpha = t' - t + \mathrm{1}_{\{t' < t\}} + \mathrm{1}_{\{\alpha = o\}} \mathrm{1}_{t'=t}$. 
In our proof of the next lemma, we will make use of the following identity, which states that for all $\lambda \neq \lambda' \in \mathbb{C}$, and $t,t' \in [0,1)$, and any $\alpha_1,\alpha_2,\alpha_3 \in \{b,o,u\}$ we have
\begin{align} \label{eq:stream}
\int_0^1 \frac{e^{ - \lambda[s-t]_{\alpha_1} }}{1-e^{-\lambda}}  \frac{e^{  - \lambda'[t'-s]_{\alpha_2} }}{1 - e^{ -\lambda'}} \mathrm{d}s = \frac{1}{\lambda-\lambda'} \left[  \frac{ e^{ - \lambda' [t'-t]_{\alpha_3}} }{ 1 - e^{ - \lambda'} }- \frac{e^{ - \lambda [t'-t]_{\alpha_3}}}{1 - e^{-\lambda}}  \right].
\end{align}
To prove \eqref{eq:stream}, first note that the left hand side does not depend on $\alpha_1,\alpha_2$, since these only affect the integrand on a set of measure zero. Also note that the right-hand-side does not depend on $\alpha_3$, since it is easily verified that the limit of the right-hand-side as $t'-t=1$ is equal to that as $t'-t=0$. To actually prove the identity, one may assume without loss of generality using translation invariance that 
$t = 0$. From here it is easy to verify \eqref{eq:stream} by way of a simple calculation using the definition $[t'-s]_\alpha = t'-s+\mathrm{1}_{s > t'} + \mathrm{1}_{\alpha=o}\mathrm{1}_{s=t'}$.

\begin{lemma} \label{lem:inversion}
Define $J^{\beta,\theta_2,T}:\mathbb{T}_{n,*} \times \mathbb{T}_{n,*} \to \mathbb{C}$ by 
\begin{align*}
J^{\beta,\theta_2,T}(w,w') = (-1)^{\mathrm{1}_{\{\alpha=o\}} } \frac{ 1 }{n} \sum_{ z^n = (-1)^{\theta_2} } z^{\mathrm{1}_{\{\alpha = b\}}-h'+h} \frac{ e^{ - (\beta+Tz)[t'-t]_\alpha }}{ 1 - e^{ - (\beta+Tz)} } .
\end{align*}
Then using the shorthand $J := J^{\beta,\theta_2,T}$ and $C = C^{\beta,\theta_2}$ we have
\begin{align*}
J ( I +TC) = C.
\end{align*}
\end{lemma} 
\begin{proof}
First we examine $JC(w,w')$.
Using the definition $JC(w,w') := \int_{\mathbb{T}_{n,*}} \mathrm{d}w'' J(w,w'') C(w'',w')$ we have 
\begin{align} \label{eq:elk}
JC(w,w') &= \sum_{h'' \in \mathbb{Z}_n} \sum_{\alpha'' \in \{b,o,u\} } \int_0^1 \mathrm{d}t''   (-1)^{\mathrm{1}_{\{\alpha=o\}} } \frac{1 }{n} \sum_{ z^n =  (-1)^{\theta_2}} z^{\mathrm{1}_{\{\alpha = b\}}-h''+h} \frac{ e^{ - (\beta+Tz)[t''-t]_\alpha }}{ 1 - e^{ - (\beta+Tz)} } \nonumber \\ 
&\times (-1)^{\mathrm{1}_{\{\alpha''=o\}} +  \theta_2 \mathrm{1}_{\{\alpha'' = b, h' = 0\}} } \mathrm{1}_{\{h'=h''+\mathrm{1}_{\{\alpha''=b\}}\}}  \frac{e^{ - \beta[t'-t'']_{\alpha''} }}{1 - e^{-\beta}}  .
\end{align}
Comparing the terms involving $\alpha'' = o$ and $\alpha'' = u$, we see that thanks to the sign $(-1)^{\mathrm{1}_{\{\alpha''=o\}}}$, they are the negative of one another (but for a set of measure zero in $t''$). Thus they cancel one another out.

It follows that we may restrict our interest to the sum over $\alpha'' = b$, which in turn by virtue of the indicator $\mathrm{1}_{\{h'=h''+\mathrm{1}_{\{\alpha''=b\}}\}}$ guarantees $h'' = h'-1$ mod $n$. Equivalently, $h'' = h'-1+n\mathrm{1}_{\{h'=0\}}$. Plugging this value of $h''=h'-1+n\mathrm{1}_{\{h'=0\}}$ and $\alpha''=b$ into \eqref{eq:elk} and rearranging to obtain the first equality below, and then using \eqref{eq:stream} to obtain the second we have 
\begin{align} \label{eq:Crisp}
JC(w,w') &=   (-1)^{\mathrm{1}_{\{\alpha=o\}} + \theta_2 \mathrm{1}_{\{h' = 0\}} } \frac{1 }{n} \sum_{ z^n = (-1)^{\theta_2}} z^{\mathrm{1}_{\{\alpha = b\}}-h'+1-n\mathrm{1}_{\{h'=0\}}+h} \int_0^1 \frac{ e^{ - (\beta+Tz)[t''-t]_\alpha }}{ 1 - e^{ - (\beta+Tz)} }  \frac{e^{ - \beta[t'-t'']_{b} }}{1 - e^{-\beta}} \mathrm{d}t'' \nonumber \\
&= \frac{1}{T} (-1)^{\mathrm{1}_{\{\alpha=o\}} } \frac{1 }{n} \sum_{ z^n = (-1)^{\theta_2}} z^{\mathrm{1}_{\{\alpha = b\}}-h'+h} \left\{ \frac{ e^{ - \beta[t'-t]_\alpha }}{ 1 - e^{ - \beta} } -  \frac{ e^{ - (\beta+Tz)[t'-t]_\alpha }}{ 1 - e^{ - (\beta+Tz)} }  \right\},
\end{align}
where in the latter equality above we have also used the fact that whenever $z^n = (-1)^{\theta_2}$ we have $z^{n\mathrm{1}_{\{h'=0\}}} = (-1)^{\theta_2 \mathrm{1}_{\{h'=0\}} }$. 

Using the fact that $ \frac{1 }{n} \sum_{ z^n = (-1)^{\theta_2}} z^{\mathrm{1}_{\{\alpha = b\}}-h'+h} = \mathrm{1}_{\{h'=h+\mathrm{1}_{\{\alpha = b\}}\}} (-1)^{\theta_2 \mathrm{1}_{\{h'=0\}} }$, it follows from \eqref{eq:Crisp} that 
\begin{align*} 
JC(w,w') &= \frac{1}{T} ( C(w,w') - J(w,w')).
\end{align*}
Thus $J(I+TC) = C$, completing the proof.
\end{proof}

We are now ready to prove Theorem \ref{thm:detnew}. The proof is based on ideas for (unmixed) correlation functions of discrete determinantal processes (see e.g.\ \cite{romik}).

\begin{proof}[Proof of Theorem \ref{thm:detnew}]
Throughout the proof we lighten notation by writing $C = C^{\beta,\theta_2}$, etc..
The equation \eqref{eq:Psum0} was proved at the beginning of this section in \eqref{eq:Psum}. It remains to show that the correlation functions of each $\mathbf{P}^{\lambda,\theta,T}$ satisfy \eqref{eq:detnew}. To this end, let $\phi:\mathbb{T}_{n,*} \to \mathbb{C}$ be a bounded and measurable test function. We examine the quantity 
\begin{align*}
S_\phi &:= \sum_{N \geq 0} \frac{T^N}{N!} \int_{\mathbb{T}_{n,*}^N} \prod_{i=1}^N \phi(w_i) p_N^{\lambda,\theta,T}( w_1,\ldots,w_N) \mathrm{d}w_1 \cdots \mathrm{d}w_N,
\end{align*}
and show that it is unchanged if the mixed correlation function $p_N^{\lambda,\theta,T}( w_1,\ldots,w_N)$ is replaced by the determinant $\det_{i,j=1}^N H^{\lambda+\theta_1\pi \iota, \theta_2, T}(w_i,w_j)$.

Using Lemma \ref{lem:overcharge} to obtain the initial equality below, we have
\begin{align} \label{eq:open1}
S_\phi &= \sum_{N \geq 0} \frac{T^N}{N!} \int_{\mathbb{T}_{n,*}^N} \prod_{i=1}^N \phi(w_i) \frac{e^{-Tn}}{Z^\theta(\lambda,T)} \sum_{ m \geq 0 } \frac{T^{B+m}}{m!} \int_{\mathbb{T}_{n,*}^{m}} g_N^{\lambda,\theta}(w_1,\ldots,w_N, \zeta_1,\ldots,\zeta_m )  \mathrm{d}\zeta_{1} \cdots \mathrm{d}\zeta_{m} \mathrm{d}w_1 \cdots \mathrm{d}w_N \nonumber \\
&=\frac{e^{-Tn}}{Z^\theta(\lambda,T)}  \sum_{N \geq 0} \frac{T^N}{N!} \int_{\mathbb{T}_{n,*}^N} \prod_{i=1}^N \phi_T(w_i)  \sum_{ m \geq 0 } \frac{T^{m}}{m!} \int_{\mathbb{T}_{n,*}^{m}} g_N^{\lambda,\theta}(w_1,\ldots,w_N, \zeta_1,\ldots,\zeta_m )  \mathrm{d}\zeta_{1} \cdots \mathrm{d}\zeta_{m} \mathrm{d}w_1 \cdots \mathrm{d}w_N, 
\end{align}
where in the final equality above, we have used $B = \# \{ w_i : \alpha_i = b \}$ and set $\phi_T:\mathbb{T}_{n,*} \to \mathbb{C}$ to be
\begin{align*}
\phi_T(\alpha, t,h) := T^{ \mathrm{1}_{\{\alpha = b\}}} \phi(\alpha, t,h).
\end{align*}
Reindexing the sum over $M= N+m$ we obtain 
\begin{align*}
S_\phi = \frac{e^{-Tn}}{Z^\theta(\lambda,T)}  \sum_{M \geq 0} \frac{T^M}{M!} \int_{\mathbb{T}_{n,*}^M} \prod_{i=1}^M(1 +  \phi_T(w_i))  g_M^{\lambda,\theta}(w_1,\ldots,w_M)  \mathrm{d}w_{1} \cdots \mathrm{d}w_{M}.
\end{align*}
Now appealing to Theorem \ref{thm:replicacast}, with $C = C^{\lambda+\theta_1\pi \iota, \theta_2}$ we have 
\begin{align} \label{eq:nuclear}
S_\phi &=  \frac{(-1)^{(\theta_1+1)(\theta_2+n+1)} (1-e^{- (\lambda+\theta_1 \pi \iota)})^n e^{-Tn}}{Z^\theta(\lambda,T)}  \sum_{M \geq 0} \frac{T^M}{M!} \int_{\mathbb{T}_{n,*}^M}   \det_{i,j=1}^M (1 + \phi_T(w_i)) C(w_i,w_j) \mathrm{d}w_{1} \cdots \mathrm{d}w_{M} \nonumber \\
&=  \frac{(-1)^{(\theta_1+1)(\theta_2+n+1)} (1-e^{- (\lambda+\theta_1 \pi \iota)})^n e^{-Tn}}{Z^\theta(\lambda,T)} \det \limits_{\mathbb{T}_{n,*}} ( I + T (I + \phi_T)C ),
\end{align}
where $ \det \limits_{\mathbb{T}_{n,*}}(I+A)$ denotes the Fredholm determinant of an operator $A:\mathbb{T}_{n,*}\times \mathbb{T}_{n,*} \to \mathbb{C}$. 

Note that by setting $\phi = 0$, and using the fact that $S_0 = p_0^{\lambda+\theta_1 \pi \iota,\theta_2,T} = 1$ we have 
\begin{align} \label{eq:nuclear2}
1  =  \frac{(-1)^{(\theta_1+1)(\theta_2+n+1)} (1-(-1)^{\theta_1}\zeta)^n e^{-Tn}}{Z^\theta(\lambda,T)} \det \limits_{\mathbb{T}_{n,*}} ( I + T C ),
\end{align}
thus by \eqref{eq:nuclear} and \eqref{eq:nuclear2} we have 
\begin{align} \label{eq:plant}
S_\phi = \frac{ \det \limits_{\mathbb{T}_{n,*}} ( I + T (I + \phi_T)C )}{  \det \limits_{\mathbb{T}_{n,*}} ( I + T C )}.
\end{align}
Now with $J^{\lambda+\theta_1\pi \iota,\theta_2,T}$ as in Lemma \ref{lem:inversion}, define the operator $H = H^{\lambda+\theta_1\pi \iota,\theta_2,T}:\mathbb{T}_{n,*} \times \mathbb{T}_{n,*} \to \mathbb{C}$ by $H^{\lambda+\theta_1\pi \iota,\theta_2,T}(w,w') = T^{\mathrm{1}_{\{\alpha = b\}}} J^{\lambda+\theta_1\pi \iota,\theta_2,T}(w,w')$. Now set $H_\phi(w,w') := \phi(w)H(w,w')$. Using the definition of $H$ and Lemma \ref{lem:inversion} we have 
\begin{align*}
(I + T H_\phi)(I+TC) =I + T (I + \phi_T)C.
\end{align*}
By the multiplicativity of Fredholm determinants (see e.g.\ \cite[Section 3.4]{AGZ}) we have $\det \limits_{\mathbb{T}_{n,*}}  ( I + T (I + \phi_T)C ) ) =  \det \limits_{\mathbb{T}_{n,*}}  ( I + TH_\phi ) \det \limits_{\mathbb{T}_{n,*}} (I+TC)$, and in particular, by \eqref{eq:plant} we have 
\begin{align} \label{eq:open2}
S_\phi = \det \limits_{\mathbb{T}_{n,*}}  ( I + TH_\phi )= \sum_{N \geq 0} \frac{T^N}{N!} \int_{\mathbb{T}_{n,*}^N} \prod_{i=1}^N \phi(w_i) \det_{i,j=1}^N H(w_i,w_j) \mathrm{d}w_1 \cdots \mathrm{d}w_N.
\end{align}
Comparing \eqref{eq:open1} and \eqref{eq:open2}, and using the fact that $\phi$ is an arbitrary test function, it follows that for any $(w_1,\ldots,w_N)$ we have
\begin{align*}
p_N^{\lambda,\theta,T}(w_1,\ldots,w_N) = \det_{i,j=1}^N H^{\lambda+\theta_1 \pi \iota, \theta_2, T}(w_i,w_j),
\end{align*}
where $H^{\lambda+\theta_1\pi \iota,\theta_2,T}(w,w') = T^{\mathrm{1}_{\{\alpha = b\}}} J^{\lambda+\theta_1\pi \iota,\theta_2,T}(w,w')$ with $J^{\lambda+\theta_1\pi \iota,\theta_2,T}$ from Lemma \ref{lem:inversion}. 

We are almost done. To finish, we note that the slightly different expression given in \eqref{eq:detnew} follows from using the fact that for any $\alpha,\alpha' \in \{b,o,u\}$, $y,y' \in \mathbb{T}_n$, we have $H^{\lambda+\theta_1\pi \iota,\theta_2,T}( (\alpha, y) , (\alpha',y')) = H^{\lambda+\theta_1\pi \iota,\theta_2,T}_\alpha ( y,y')$, where $H^{\lambda+\theta_1\pi \iota,\theta_2,T}_\alpha:\mathbb{T}_n \times \mathbb{T}_n \to \mathbb{C}$ is as in the statement of Theorem \ref{thm:detnew}. (We note from the definition that $H^{\lambda+\theta_1\pi \iota,\theta_2,T}( (\alpha, y) , (\alpha',y'))$ does not depend on $\alpha'$.) 

That completes the proof of Theorem \ref{thm:detnew}. 
\end{proof}

We close this section with a remark on the structure of the probability measures $\mathbf{P}_n^{\lambda,\theta,T}$. We saw in the previous result that the mixed correlation functions of $\mathbf{P}_n^{\lambda,\theta,T}$ take the form 
\begin{align*}
p_N^{\lambda,\theta,T}(w_1,\ldots,w_N) = \det_{i,j=1}^N H^{\lambda+\theta_1 \pi \iota, \theta_2, T}_{\alpha_i}(y_i,y_j).
\end{align*}
From the structure of the correlation functions, we have the equality of complex measures
\begin{align} \label{eq:melta}
\mathbf{P}_n^{\lambda,(\theta_1,\theta_2),T} = \mathbf{P}_n^{\lambda+\theta_1 \pi \iota, (0,\theta_2),T}.
\end{align}
In any case, in the sequel we will write
\begin{align*}
\mathbf{P}_n^{\beta,\theta_2,T} := \mathbf{P}_n^{\beta,(0,\theta_2),T},
\end{align*}
with the understanding that we may plug in $\beta = \lambda +\theta_1 \pi \iota$ to recover the measure $\mathbf{P}_n^{\lambda,\theta,T}$.

\section{Scaling limits and exclusion processes on the ring} \label{sec:gordenko}

\subsection{The horizontal scaling limit}

In the previous section, we constructed probability measures $\mathbf{P}_n^{\beta,\theta_2,T}$ on the set of occupation processes $(X_t)_{t \in [0,1)}$. Theorem \ref{thm:detnew} characterises the mixed correlation structure of these processes. In this section we study the asymptotics of the complex probability measures $\mathbf{P}_n^{\beta,\theta_2,T}$ as the parameter $T$ controlling the density of the beads is sent to infinity, leading to a proof of Theorem \ref{thm:gordnew}. The resulting limit measures are genuine probability measures (as opposed to complex probability measures).

When $T$ becomes large, the density of beads on $\mathbb{T}_n$ under $\mathbf{P}_n^{\beta,\theta_2,T}$ takes the order $T$, and for this reason we consider correlations at a horizontal distance of order $1/T$. More specifically, we consider the scaling limit 
\begin{align} \label{eq:sca}
\beta = pT \qquad y_T = (t,h), y_T'=(t',h') \text{ with } t = \frac{1}{2}+s/T, t' = \frac{1}{2}+s'/T,
\end{align}
where $p$ is a parameter controlling the tilt of the configuration. Assume now that the parameter $p$ is chosen so that $\mathrm{Re}(p+z) \neq 0$ for all $z^n = (-1)^{\theta_2}$.\\

We rescale space using the affine map $\phi:[0,1] \to [-\frac{T}{2},\frac{T}{2}]$ given by $\phi(t) := (t-\frac{1}{2})T$. Now define $\tilde{\mathbf{P}}_n^{pT,\theta_2,T}$
to be the pushforward complex measure of the rescaled occupation process $(X^T_t)_{t \in \mathbb{R}}$ defined by 
\begin{align*}
X_t^T := X_{\phi^{-1}(t)} \qquad \text{ for }t \in [-T/2,T/2] \qquad \text{and} \qquad X_t^T := X_0 \qquad \text{ for } t \in \mathbb{R}-[-T/2,T/2].
\end{align*}
With a view to establishing the convergence in distribution of $(X^T_t)_{t \in \mathbb{R}}$ as $T \to \infty$, the following lemma characterises the convergence in distribution of the relevant kernels at the horizontal scale $1/T$. 

\begin{lemma}
Suppose in the setting of \eqref{eq:sca}, the parameter $p$ is such that the real numbers $\{ \mathrm{Re}(p+z) : z^n = (-1)^{\theta_2} \}$ are nonzero, and exactly $\ell$ of them are negative. For $\alpha \in \{b,o,u\}$, let $H^{\beta,\theta_2,T}_\alpha:\mathbb{T}_n \to \mathbb{C}$ be as in Theorem \ref{thm:detnew}, and let $H^{n,\ell}_\alpha:\mathbb{R} \times \mathbb{Z}_n \to \mathbb{C}$ be as in \eqref{eq:kergord}. Then
\begin{align*}
\lim_{T \to \infty} T^{- \mathrm{1}_{\alpha=b}} H^{pT,\theta_2,T}_\alpha( s/T,h) = e^{ -ps} H^{n,\ell}_\alpha(s,h).
\end{align*}
\end{lemma}
\begin{proof}
By definition, 
\begin{equation*}
T^{- \mathrm{1}_{\alpha=b}} H^{pT,\theta_2,T}_\alpha( s/T,h)  = (-1)^{ \mathrm{1}_{\{\alpha=o\}} } \frac{ 1 }{n} \sum_{ z^n = (-1)^{\theta_2} } z^{\mathrm{1}_{\{\alpha = b\}}-h} \frac{ e^{ - T(p+z)[s/T]_\alpha }}{ 1 - e^{ - T(p+z)} }.
\end{equation*}
Using the definition of $[t]_\alpha$ we have $[s/T]_\alpha := s/T + \mathrm{1}_{s<0} + \mathrm{1}_{\alpha=o, s=0}$, so that
\begin{align} \label{eq:lakes}
T^{- \mathrm{1}_{\alpha=b}} H^{pT,\theta_2,T}_\alpha( s/T,h)  = (-1)^{ \mathrm{1}_{\{\alpha=o\}} } \frac{ 1 }{n} \sum_{ z^n = (-1)^{\theta_2} } z^{\mathrm{1}_{\{\alpha = b\}}-h} e^{ - (p+z)s } 
\frac{ e^{ - T(p+z)( \mathrm{1}_{s < 0 } + \mathrm{1}_{ \alpha = o,s=0} )} }{ 1 - e^{ - T(p+z)} }.
\end{align}
The large $T$ asymptotics of the quotient $e^{ - T(p+z)( \mathrm{1}_{s < 0 } + \mathrm{1}_{ \alpha = o,s=0} )}/( 1 - e^{ - T(p+z)} )$ depend on whether the real part of $p+z$ is positive or negative. Indeed, if $\mathrm{Re}(p+z) > 0$, then 
\begin{align} \label{eq:windermere}
\lim_{T \to \infty} \frac{ e^{ - T(p+z)( \mathrm{1}_{s < 0 } + \mathrm{1}_{ \alpha = o,s=0} )} }{ 1 - e^{ - T(p+z)} } = 1 - (\mathrm{1}_{s < 0 } + \mathrm{1}_{ \alpha = o,s=0}),
\end{align}
whereas if $\mathrm{Re}(p+z)<0$ we have 
\begin{align} \label{eq:windermere2}
\lim_{T \to \infty} \frac{ e^{ - T(p+z)( \mathrm{1}_{s < 0 } + \mathrm{1}_{ \alpha = o,s=0} )} }{ 1 - e^{ - T(p+z)} } = -( \mathrm{1}_{s < 0 } + \mathrm{1}_{ \alpha = o,s=0}). 
\end{align}
By our assumptions on $p$ in the statement of the lemma, with $\mathcal{L}_{n,\ell}$ and $\mathcal{R}_{n,\ell}$ as in Definition \ref{df:rootsets} we have
\begin{align}
\mathcal{L}_{n,\ell}  = \{ z: \mathrm{Re}(p+z) < 0 \} \qquad \text{and} \qquad \mathcal{R}_{n,\ell} = \{ z : \mathrm{Re}(p+z) > 0 \}.
\end{align}
Consequently, using \eqref{eq:windermere} and \eqref{eq:windermere2} in \eqref{eq:lakes}, and then appealing to the definition \eqref{eq:kergord} of $H^{n,\ell}_\alpha$, we complete the proof.
\end{proof}

The following result implies Theorem \ref{thm:gordnew} from the introduction:

\begin{thm}\label{lem:longconv}
As $T \to \infty$, the complex measure $\tilde{\mathbf{P}}_n^{pT,\theta_2,T}$ governing an occupation process $(X_t^T)_{t \in \mathbb{R}}$ converges to $\mathbf{P}^{n,\ell}$. 

Moreover, the probability measures $\mathbf{P}^{n,\ell}$ have mixed correlation functions given in terms of the kernel $H_\alpha^{n,\ell}$.
\end{thm}

\begin{proof}
Since $(X^T_t)_{t \in \mathbb{R}}$ is cadlag and takes values in a finite space, it is sufficient to establish the convergence of the finite dimensional distributions $\tilde{\mathbf{P}}_n^{pT,\theta_2,T}( X(t_1) = E_1,\ldots,X(t_n) = E_n )$. By Theorem \ref{thm:detnew}, these quantities may be given explicitly in terms of determinants of the occupation kernel $H^{pT,\theta_2,T}_o$ associated with $\mathbf{P}_n^{pT,\theta_2,T}$ defined in the statement of Theorem \ref{thm:detnew}. By the previous lemma, at the horizontal scale $1/T$, the kernels $H^{pT,\theta_2,T}_o$ converge to $H^{n,\ell}_o$ as $T \to \infty$, thus establishing the convergence in distribution.

According to the previous lemma, the mixed correlation structure is inherited from the scaling limit; the factor $T^{-\mathrm{1}_{\alpha=b}}$ is to account for the rescaling of horizontal space. The additional factor $e^{-ps}$ in the kernel disappears after diagonal conjugation. Thus the mixed correlation functions of $\mathbf{P}^{n,\ell}$ are given by $H^{n,\ell}_\alpha$.
\end{proof}

We emphasise that the parameter $p$ affects the limiting behaviour of the measures $\mathbf{P}_n^{pT,\theta_2,T}$ \emph{only} through the number of roots $n^{\text{th}}$ roots of $(-1)^{\theta_2}$ whose real parts lie to the left of $-p$. 

\subsection{Occupation probabilities under $\mathbf{P}^{n,\ell}$}
In this section, we show that $\mathbf{P}^{n,\ell}$ is supported on the event $\{ \# X_t = \ell ~ \forall t \in \mathbb{R} \}$.

\begin{lemma} \label{lem:only}
Whenever $E = \{h_1,\ldots,h_p\}$ does not have cardinality $\ell$, for any $t \in \mathbb{R}$ we have $\mathbf{P}^{n,\ell}(X_t = E ) = 0$. 
\end{lemma}
\begin{proof}
Take an enumeration of the roots $\mathcal{L}_{n,\ell} = \{z_1,\ldots,z_\ell\}$ and $\mathcal{R}_{n,\ell} = \{ z_{\ell+1},\ldots,z_n\}$, and suppose $\{h_1,\ldots,h_n\}$ is an enumeration of the set $\mathbb{Z}_n$. Then by Theorem \ref{thm:gordnew}
\begin{align} \label{eq:asd}
\mathbf{P}(X_t = \{h_1,\ldots,h_p\} ) &= \mathbf{P}( (t,h_1),\ldots,(t,h_p) \text{ occ.}, (t,h_{p+1}),\ldots,(t,h_n) \text{ unocc.} ) = \det_{j,k=1}^n K_{j,k},
\end{align}
where 
\begin{align*}
K_{j,k} = \mathrm{1}_{j \leq p} \frac{1}{n} \sum_{z \in \mathcal{L}_{n,\ell}}z^{-(h_k-h_j)} + \mathrm{1}_{j > p }\frac{1}{n} \sum_{z \in \mathcal{R}_{n,\ell}}z^{-(h_k-h_j)} 
= \frac{1}{n} \sum_{ i = 1}^n (\mathrm{1}_{j \leq p, i \leq \ell} + \mathrm{1}_{j > p, i > \ell } ) z_i^{- (h_k-h_j)} .
\end{align*}
Now define
\begin{align*}
B_{j,k} :=( \mathrm{1}_{j \leq p, k \leq \ell} + \mathrm{1}_{j \geq p+1, k \geq \ell+1} ) \frac{z_k^{h_j}}{ \sqrt{n}} \qquad \text{and} \qquad C_{j,k} = \frac{1}{\sqrt{n}} z_j^{-h_k}.
\end{align*}
Note that $K_{j,k} = \sum_{m=1}^n B_{j,m}C_{m,k}$, so that $ \det_{j,k=1}^n K_{j,k} =  \det_{j,k=1}^n B_{j,k} \det_{j,k=1}^n C_{j,k}$. 

The matrix $B_{j,k}$ may only have full rank when $p= \ell$. Thus the determinant of $B_{j,k}$, and accordingly, the determinant of $K_{j,k}$, is nonzero only when $p = \ell$. By \eqref{eq:asd}, it follows that $\mathbf{P}(X_t = \{h_1,\ldots,h_p\} )$ may only be nonzero when $p = \ell$. 
\end{proof}

We now use a well known identity for Vandermonde determinants, (see e.g.\ Macdonald \cite[Section I.3]{macdonald}) which states that for variables $x_1,\ldots,x_m$ we have
\begin{align} \label{eq:XX}
\det_{j,k=1}^m x_j^{k-1} =\prod_{1 \leq j < k \leq m }(x_k-x_j).
\end{align}
The identity \eqref{eq:XX} allows us to compute the stationary probabilities $\mathbf{P}^{n,\ell}(X_t=E)$ explicitly. In this direction, a brief calculation using Definition \ref{df:rootsets} tells us that $\mathcal{L}_{n,\ell} := \{z_1,\ldots,z_\ell\}$ where $z_k := e^{ \frac{2 \pi \iota }{n} \left( \frac{n-\ell+1}{2} + k - 1 \right)}$. Given an $\ell$-tuple $\mathbf{h} = (h_1,\ldots,h_\ell)$ of elements of $\mathbb{Z}_n$, define the $\ell \times \ell$ matrix
\begin{align} \label{eq:verdi}
A^\mathbf{h}_{j,k} := z_k^{h_j}, \qquad 1 \leq j,k \leq \ell.
\end{align}
Define the function $\Delta:\mathbb{Z}_n^\ell \to [0,\infty)$ by 
\begin{align*} 
\Delta(\mathbf{h}) :=  \prod_{1 \leq j < k \leq \ell }\left| e^{ \frac{2 \pi \iota}{n} h_k} - e^{ \frac{2 \pi \iota}{n} h_j} \right|.
\end{align*}
Since $\Delta$ is symmetric in its entries, if $E = \{h_1,\ldots,h_\ell\}$, we will unambiguously write $\Delta(E)$ for $\Delta(\mathbf{h})$. 

We write $\mathbb{Z}_n^{(\ell)}$ for the set of distinct tuples in $\mathbb{Z}_n^\ell$. Clearly $\det (A^\mathbf{h})$ and $\Delta(\mathbf{h})$ are nonzero only when $\mathbf{h} \in \mathbb{Z}_n^{(\ell)}$. 

The following lemma computes the determinant of $A^\mathbf{h}$.
\begin{lemma} \label{lem:rootdet}
For $\mathbf{h} \in \mathbb{Z}_n^{(\ell)}$ we have 
\begin{align*}
\det_{j,k=1}^\ell A^\mathbf{h}_{j,k} = \mathrm{sgn}(\sigma_h)(-1)^{\sum_{j=1}^\ell h_j} \iota^{\ell(\ell-1)/2)} \Delta(\mathbf{h}).
\end{align*}
where $\sigma_h:\{1,\ldots,\ell\} \to \{1,\ldots,\ell\}$ is the permutation ordering the entries of $h$, i.e.\ $h_{\sigma(1)}<\cdots<h_{\sigma(\ell)}$. 
\end{lemma}
\begin{proof}
Using the definition $z_k := e^{ \frac{2 \pi \iota }{n} \left( \frac{n-\ell+1}{2} + k - 1 \right)}$ we have 
\begin{align*}
\det_{j,k=1}^\ell A^\mathbf{h}_{j,k} = e^{ \frac{2 \pi \iota }{n}  \frac{n-\ell+1}{2} \sum_{j=1}^\ell h_j } \det_{j,k=1}^\ell x_j^{k-1},
\end{align*}
where $x_j := e^{ \frac{2 \pi \iota}{n} h_j}$. Applying \eqref{eq:XX} we obtain
\begin{align} \label{eq:apa}
\det_{j,k=1}^\ell A^\mathbf{h}_{j,k} = e^{ \frac{2 \pi \iota }{n}  \frac{n-\ell+1}{2} \sum_{j=1}^\ell h_j } \prod_{1 \leq j < k \leq \ell} \left( e^{ \frac{2 \pi \iota}{n} h_k} - e^{ \frac{2 \pi \iota}{n} h_j} \right).
\end{align}
Note that
\begin{align*}
e^{ \frac{2 \pi \iota}{n} h_k} - e^{ \frac{2 \pi \iota}{n} h_j} = (-1)^{\mathrm{1}_{h_j > h_k}} \iota e^{ \frac{ \pi \iota}{n} ( h_k + h_j ) } \left| e^{ \frac{2 \pi \iota}{n} h_k} - e^{ \frac{2 \pi \iota}{n} h_j} \right|,
\end{align*}
and also that $\prod_{1 \leq j < k \leq \ell} (-1)^{\mathrm{1}_{h_j > h_k}} = \mathrm{sgn}(\sigma_h)$. The result follows from using these observations in \eqref{eq:apa}, and simplifying.
\end{proof}

\begin{cor} \label{cor:ellstat}
For any $t \in \mathbb{R}$, and $h_1,\ldots,h_\ell$ distinct elements of $\mathbb{Z}_n$, we have 
\begin{align} \label{eq:ellconc}
\mathbf{P}^{n,\ell}(X_t=\{h_1,\ldots,h_\ell\}) = n^{-\ell}\Delta(\mathbf{h})^2.
\end{align}
\end{cor}

\begin{proof}
By Lemma \ref{lem:only}, $\mathbf{P}^{n,\ell}$ is supported on the event $\{ X_t \text{ has cardinality $\ell$} \}$, so that the events $\{ X_t = \{h_1,\ldots,h_\ell\} \}$ and $\{ \{h_1,\ldots,h_\ell\} \subseteq X_t \}$ are almost-surely equivalent. Using this fact to obtain the first equality below, and then Theorem \ref{thm:gordnew} to obtain the second, we have
\begin{align*}
\mathbf{P}^{n,\ell}( X_t = \{h_1,\ldots,h_\ell \} ) &= \mathbf{P}^{n,\ell}(  \{h_1,\ldots,h_\ell\} \subseteq X_t  ) \\
&= \det \limits_{j,k = 1}^\ell H^{n,\ell}_o(0,h_k-h_j)\\
&= \det \limits_{j,k=1}^\ell \left( \frac{1}{n} \sum_{z \in \mathcal{L}_{n,\ell}} z^{h_j-h_k} \right) \\
&= n^{-\ell} \det_{j,k=1}^\ell A^{\mathbf{h}}_{j,k} \det_{j,k=1}^\ell \overline{A^{\mathbf{h}}_{j,k}},
\end{align*}
where $A^{\mathbf{h}} := z_k^{h_j}$ is as in \eqref{eq:verdi}, and $\overline{A^{\mathbf{h}}_{j,k}}$ denotes complex conjugation. The result now follows from applying Lemma \ref{lem:rootdet}.
\end{proof}

We mention here that the stationary distribution found in Corollary \ref{cor:ellstat}, which takes the form $\mathbb{P}( X_t = \{ h_1,\ldots,h_\ell\} ) = C_{n,\ell} \Delta(\mathbf{h})^2$, may be regarded as a discrete analogue of the eigenvalue distribution of the circular unitary\ ensemble. Namely, the $\ell$ eigenvalues $e^{i\theta_1},\ldots,e^{i\theta_\ell}$ of a Haar distributed $\ell$-by-$\ell$ unitary matrix have a joint density function of the form $C_\ell \prod_{1 \leq i < j \leq \ell} |e^{i \theta_k} - e^{i \theta_j} |^2$ for $\theta_j$ and $\theta_k$ varying continuously in $[0,2\pi)$.

\subsection{Basic densities under $\mathbf{P}^{n,\ell}$} \label{sec:basicfacts}

We take a moment to perform some basic calculations for probabilities under $\mathbf{P}^{n,\ell}$ using Theorem \ref{thm:gordnew}. Setting $\mathcal{O} = \{1\}$ and $\mathcal{B} = \mathcal{U} = \varnothing$ in Theorem \ref{thm:gordnew}, we can compute the probability that $y_1$ is occupied as 
\begin{align} \label{eq:density1}
\mathbf{P}^{n,\ell}( y_1 = (t_1,h_1) \text{ is occupied} ) = H_o^{n,\ell}(0,0) = \frac{1}{n} \sum_{ z \in \mathcal{L}_{n,\ell}} z^0 = \frac{\ell}{n},
\end{align}
which reflects the fact that $\ell$ of the $n$ strings are occupied at any given $t$. A similar calculation tells us that $\mathbf{P}^{n,\ell}( y_1 = (t_1,h_1) \text{ is unoccupied} ) =  H_u^{n,\ell}(0,0)  = \frac{n - \ell}{n}$. 

Setting $\mathcal{B} = \{1\}$ and $\mathcal{O} = \mathcal{U} = \varnothing$ we can compute the probability $\mathrm{d}y_1$ contains a bead. Indeed, since $\mathcal{R}_{n,\ell}$ consists of $n-\ell$ consecutive roots around the origin, we have 
\begin{align*}
\mathbf{P}^{n,\ell}( \mathrm{d}y_1 \text{ contains a bead} ) = H_b^{n,\ell}(0,0)   \mathrm{d}y_1 = \left( \frac{1}{n} \sum_{ z \in \mathcal{R}_{n,\ell}} z\right) \mathrm{d}y_1  = \frac{1}{n} \frac{ \sin(\pi \ell/n)}{ \sin( \pi /n)} \mathrm{d}y_1 ,
\end{align*}
where the final equality above is most easily seen by noting that the sum $\sum_{z \in \mathcal{R}_{n,\ell}} z$ is a positive real number with the same modulus (\eqref{eq:concsum}) as the geometric sum $\sum_{ j =0}^{n-\ell-1} e^{2 \pi \iota j /n}$, i.e.\ $\sin(\pi \ell/n)/\sin(\pi/n)$. In other words, the density of beads per string under $\mathbf{P}^{n,\ell}$ is $\frac{1}{n} \frac{ \sin(\pi \ell/n)}{ \sin( \pi /n)}$.

\subsection{Transition probabilities under $\mathbf{P}^{n,\ell}$}

Let $E = \{h_1,\ldots,h_\ell\}$ and $E' = \{h_1',\ldots,h_\ell'\}$ be subsets of $\mathbb{Z}_n$. Then with the case $\mathcal{O} = \{1,\ldots,2\ell\}$, $\mathcal{B} = \mathcal{U} = \varnothing$ of Theorem \ref{thm:gordnew} we have 
\begin{align*}
\mathbf{P}^{n,\ell}( X_0 = E, X_t = E' ) = \det_{j,k=1}^{2\ell} H^{n,\ell}_o(y_j,y_k),
\end{align*}
where for $1 \leq j \leq \ell$, $y_j := (0,h_j)$, and for $\ell+1 \leq j \leq 2\ell$ we have $y_j := (t,h'_{j-\ell})$. 
In particular, using Corollary \ref{cor:ellstat} we have the transition probabilities
\begin{align} \label{eq:2l}
\mathbf{P}^{n,\ell}( X_t = E' | X_0 = E ) = \frac{n^\ell}{\Delta(E)^2} \det_{j,k=1}^{2\ell} H^{n,\ell}_o(y_j,y_k) \qquad t > 0.
\end{align}
However, the representation \eqref{eq:2l} is not ideal for computing the transition \emph{rates} $\lim_{t \downarrow 0} \frac{1}{t} \mathbf{P}^{n,\ell}( X_t = E' | X_0 = E )$. These can be computed however using the mixed correlation structure in Theorem~ \ref{thm:gordnew}, as in the proof of the following lemma.

\begin{lemma} \label{lem:rates}
Let $\mathbf{h} = (h_1,\ldots,h_\ell)$ and $\mathbf{h}' := (h'_1,\ldots,h'_\ell)$ be $\ell$-tuples of distinct elements of $\mathbb{Z}_n$, with 
\begin{align*}
h_j' = h_j \text{  for $j \leq \ell-1$} \qquad \text{and} \qquad h'_\ell := h_\ell+1,
\end{align*}
where $h_\ell+1$ is to be taken mod $n$. Then
\begin{align*}
\lim_{t \downarrow 0} \frac{1}{t} \mathbf{P}^{n,\ell}( X_t = \{h'_1,\ldots,h'_\ell\} | X_0 = \{h_1,\ldots,h_\ell\} ) = \frac{\Delta(\mathbf{h}')}{\Delta(\mathbf{h})}.
\end{align*}
\end{lemma}

\begin{proof}
Write $E = \{h_1,\ldots,h_\ell\}$ and $E' = \{ h_1,\ldots,h_{\ell-1},h_{\ell}+1\}$. Since $\{X_0 = E , X_t = E' \}$ implies a bead lies on the segment $(h_\ell \mathbf{e}_2,h_\ell \mathbf{e}_2 + t \mathbf{e}_1]$ of $\mathbb{R} \times \mathbb{Z}_n$, using Theorem \ref{thm:gordnew} to obtain the second equality below we have 
\begin{align} \label{eq:glasses}
\mathbf{P}^{n,\ell}(X_0=E,X_t = E') &= \mathbf{P}^{n,\ell}((0,h_1),\ldots,(0,h_\ell) \text{ are occ.}, (h_\ell \mathbf{e}_2,h_\ell \mathbf{e}_2 + t \mathbf{e}_1] \text{ contains a bead}) + o(t) \nonumber \\
&= t \det_{i,j=1}^{\ell+1} H^{n,\ell}_{\alpha_i}( y_j - y_i) + o(t),
\end{align}
where 
\begin{align*}
y_i = (0,h_i), \alpha_i = o \text{ for all $1 \leq i \leq \ell$ and } y_{\ell+1} = (t,h_\ell), \alpha_{\ell+1}= b.
\end{align*}
(The choice $y_{\ell+1} = (t,h_\ell)$ is not important, we could just have easily chosen $y_{\ell+1}=(t/2,h_\ell)$, say.) Now with $K_{j,k} := H^{n,\ell}_o(y_k-y_j)$ and $J_{\ell+1,k} := H^{n,\ell}_b(y_k-y_{\ell+1})$, we may write 
\begin{align} \label{eq:cubitt}
\det_{j,k=1}^{\ell+1} H^{n,\ell}_{\alpha_i} ( y_j - y_i) =  \det 
\begin{pmatrix}
  K_{j,k} & K_{j,\ell} & K_{j,\ell+1}\\
 K_{\ell,j} & K_{\ell,\ell} & K_{\ell,\ell+1} \\
  J_{\ell+1,k}  &J_{\ell+1,\ell}  & J_{\ell+1,\ell+1}
  \end{pmatrix},
\end{align}
where the right-hand-side of \eqref{eq:cubitt} depicts the determinant of an $((\ell-1) +1 + 1)$ by $((\ell-1) +1 + 1)$ block matrix.

It turns out that the final two columns of the matrix in \eqref{eq:cubitt} are very similar to one another. In order to capture this, we need to study the continuity of the kernels $H_\alpha^{n,\ell}(s,h)$ near $s=0$. To this end, one can verify using the definition \eqref{eq:kergord} and the equality $\frac{1}{n} \sum_{ z \in \mathcal{R}_{n,\ell}} z^h + \frac{1}{n} \sum_{z \in \mathcal{L}_{n,\ell}} z^h = \mathrm{1}_{ h = 0}$ that for any $\alpha,h$ the left and right limits of $H_\alpha^{n,\ell}(s,h)$ near $s = 0$ satisfy
\begin{align} \label{eq:downjump}
\lim_{s \downarrow 0} H_\alpha^{n,\ell}(s,h) - \lim_{s \uparrow 0} H_\alpha^{n,\ell}(s,h)  = (-1)^{\mathrm{1}_{\{\alpha =o \}}} \mathrm{1}_{\{ h = \mathrm{1}_{\{ \alpha=b\} } \} }.
\end{align}
In particular, $H_\alpha^{n,\ell}$ is continuous on $\mathbb{R}\times \mathbb{Z}_n - \{(0,  \mathrm{1}_{\{ \alpha=b\} }  )\}$. Moreover, $H_\alpha^{n,\ell}$ is continuous on one side at its discontinuity: $H_\alpha^{n,\ell}$ is right-continuous at $(0,  \mathrm{1}_{\{ \alpha=b\} }  )$ for $\alpha=b,u$, and left-continuous at $(0,  \mathrm{1}_{\{ \alpha=b\} }  ) = (0,0)$ for $\alpha = o$. 

We now apply these continuity notions to study the matrix in \eqref{eq:cubitt}, with subsequent $o(1)$ terms referring to the limit as $t \downarrow 0$. First note that since $H_o^{n,\ell}$ is continuous on $\mathbb{R} \times \mathbb{Z}_n -(0,0)$, it follows that
\begin{align} \label{eq:blu1}
K_{j,\ell+1} =K_{j,\ell}+o(1) \text{ for $j<\ell$}.
\end{align}
Note that $H_o^{n,\ell}$ is left-continuous at $(0,0)$, and by \eqref{eq:downjump}, has a downward jump of size $-1$ at $(0,0)$. Thus 
\begin{align} \label{eq:blu2}
\qquad K_{\ell,\ell+1}=-1+K_{\ell,\ell}+o(1).
\end{align}
Finally, since $H^{n,\ell}_b(t,h)$ is continuous at $(t,h) = (0,0)$, we have 
\begin{align} \label{eq:blu3}
J_{\ell+1,\ell+1} = J_{\ell+1,\ell}+o(1).
\end{align}
Using \eqref{eq:blu1}, \eqref{eq:blu2} and \eqref{eq:blu3} in \eqref{eq:cubitt} we have
\begin{align} \label{eq:cubitt2}
\det_{j,k=1}^{\ell+1} H^{n,\ell}_{\alpha_i} ( y_j - y_i) =  \det 
\begin{pmatrix}
  K_{j,k} & K_{j,\ell} & K_{j,\ell}\\
 K_{\ell,j} & K_{\ell,\ell} & - 1 +  K_{\ell,\ell} \\
  J_{\ell+1,k}  &J_{\ell+1,\ell}  & J_{\ell+1,\ell}
  \end{pmatrix} +o(1)
\end{align}
as $t \downarrow 0$. Subtracting the $\ell^{\text{th}}$ column from the $(\ell+1)^{\text{th}}$ in the matrix in \eqref{eq:cubitt2} to obtain the first equality below, and then expanding the determinant through the $(\ell,\ell+1)$ entry to obtain the second, we have 
\begin{align} \label{eq:cubitt3}
\det_{j,k=1}^{\ell+1} H^{n,\ell}_{\alpha_i} ( y_j - y_i) =  \det 
\begin{pmatrix}
  K_{j,k} & K_{j,\ell} & 0\\
 K_{\ell,j} & K_{\ell,\ell} & -1  \\
  J_{\ell+1,k}  & J_{\ell+1,\ell}  & 0
  \end{pmatrix} +o(1) = \det \begin{pmatrix}
  K_{j,k} & K_{j,\ell} \\
  J_{\ell+1,k}  & J_{\ell+1,\ell} 
  \end{pmatrix} +o(1),
\end{align}
where the latter determinant above depicts that of an $(\ell-1)+1$ by $(\ell-1)+1$ block matrix. 

Now note that for $j < \ell, k \leq \ell$,
\begin{align*}
K_{j,k} :=  H^{n,\ell}_o(y_k-y_j)  = H_o^{n,\ell}(0,h_k-h_j) = \frac{1}{n} \sum_{z \in \mathcal{L}_{n,\ell}}z^{h_j-h_k},
\end{align*}
and for $k \leq \ell$
\begin{align*}
J_{\ell+1,k} :=  H^{n,\ell}_b(y_k-y_{\ell+1})  = H^{n,\ell}_b(-t,h_{\ell}+1-h_j) = \frac{1}{n} \sum_{ z \in \mathcal{L}_{n,\ell} } z^{1+h_\ell-h_j} + o(1).
\end{align*}
Define the matrix $(Q_{j,k})_{1 \leq j,k \leq \ell}$ by 
\begin{align*}
\text{$Q_{j,k} := K_{j,k}$ for $j < \ell$ and $Q_{\ell,k} := J_{\ell+1,k}$}.
\end{align*}
We see that we may write $Q_{j,k} = \frac{1}{n} \sum_{ z \in \mathcal{L}_{n,\ell} } z^{h_j'-h_k} + o(1)$, where $h_j'$ are as in the statement of theorem. Then $Q = \frac{1}{n}(A^\mathbf{h})(A^{\mathbf{h}'})^*$, where $A^\mathbf{h}$ and $A^{\mathbf{h}'}$ are as in \eqref{eq:verdi}, and $(A^{\mathbf{h}'})^*$ denotes conjugate transpose. Assuming without loss of generality that the ordering of $h'$ is the same as that of $h$ (this is certainly the case if we assume w.l.o.g.\ that $h_\ell \leq n-2$), by \eqref{eq:cubitt3} and Lemma \ref{lem:rootdet} we have
\begin{align*}
\det_{j,k=1}^{\ell+1} H^{n,\ell}_{\alpha_i} ( y_j - y_i)  = \det_{j,k=1}^\ell Q_{j,k} = n^{-\ell}\det(A^\mathbf{h}) \overline{ \det(A^{\mathbf{h}'}) } = n^{-\ell} \Delta(\mathbf{h}) \Delta(\mathbf{h}').
\end{align*} 
In particular, by \eqref{eq:glasses} we have $\mathbf{P}^{n,\ell}(X_0=E,X_t = E') = tn^{-\ell} \Delta(\mathbf{h}) \Delta(\mathbf{h}') + o(t)$. Dividing through by $\mathbf{P}^{n,\ell}(X_0 = h)$ in \eqref{eq:ellconc} we obtain the result.
\end{proof}

\subsection{Eigenfunctions of Markov chains} \label{sec:generator}
Our work in this section is related to that of K\"onig, O'Connell and Roch \cite{KOR}, who show under moment conditions that the Vandermonde determinant in $k$ variables is harmonic for any random walk in $\mathbb{R}^k$. Their observation is somehow dual to our setting, where we work in continuous-time but discrete space. 

Let $\mathbf{e}_j = (0,\ldots,0,1,0,\ldots)$ be the element of $\mathbb{Z}_n^\ell$ with a $1$ in the $j^{\text{th}}$ slot and zeros in the other slots. For $\mathbf{h} =(h_1,\ldots,h_\ell)$ in $\mathbb{Z}_n^{(\ell)}$, we take $\mathbf{h}+\mathbf{e}_j$ mod $n$ coordinatewise.

Let $\mathcal{P}_\ell(\mathbb{Z}_n)$ denote the collection of subsets of $\mathbb{Z}_n$ of cardinality $\ell$. We can associate the set of functions $\{ f:\mathcal{P}_\ell(\mathbb{Z}_n) \to \mathbb{R} \}$ with the set $\{ f: \mathbb{Z}^{(\ell)}_n \to \mathbb{R} , \text{$f$ symmetric} \}$ of functions $f$ on $\mathbb{Z}^{(\ell)}_n$ satisfying $f(h_1,\ldots,h_\ell) = f(h_{\sigma(1)},\ldots,h_{\sigma(\ell)})$ for permutations $\sigma$. 

According to Lemma \ref{lem:rates}, the generator of the Markov process $(X_t)_{t \in \mathbb{R}}$ under $\mathbf{P}^{n,\ell}$ is the linear operator $\mathcal{G}^{\mathrm{Gordenko},n,\ell}$ acting on $\{ f: \mathbb{Z}^{(\ell)}_n \to \mathbb{R} , \text{$f$ symmetric} \}$ 

and defined by 
\begin{align} \label{eq:gen}
\mathcal{G}^{n,\ell}f(\mathbf{h}) :=  \sum_{ j : h_k \neq h_j +1 \forall k} \frac{\Delta(\mathbf{h}+\mathbf{e}_j)}{\Delta(\mathbf{h})}( f(\mathbf{h}+\mathbf{e}_j) - f(\mathbf{h})).
\end{align}
We now introduce TASEP (the totally asymmetric exclusion process) on $\mathbb{Z}_n$ to be the continuous-time Markov chain $(X_t)_{t \in \mathbb{R}}$ taking values in $\mathcal{P}_\ell(\mathbb{Z}_n)$ and with the generator
\begin{align} \label{eq:genTASEP}
\mathcal{G}^{\mathrm{TASEP},n,\ell}f(\mathbf{h}) := \sum_{ j : h_k \neq h_j +1 \forall k} ( f(\mathbf{h}+\mathbf{e}_j) - f(\mathbf{h})).
\end{align}

We now prepare the proof of Theorem \ref{thm:tasep}, which states that the probability law of TASEP on $\mathbb{Z}_n$ can be recovered from the probability laws $\mathbf{P}^{n,\ell}$ via a simple exponential transform. To this end, for $\mathbf{h} \in \mathbb{Z}_n^{(\ell)}$ let us write
\begin{align*}
\mathbf{P}^{n,\ell}_\mathbf{h} = \mathbf{P}_h^{n,\ell,\mathrm{Gordenko}} \qquad \text{and} \qquad \mathbf{P}^{n,\ell,\mathrm{TASEP}}
\end{align*}
for the respective probability laws of the noncolliding walkers and of TASEP associated with $\ell$ walkers on $\mathbb{Z}_n$ starting from a configuration $X_0 = \{h_1,\ldots,h_\ell\}$. 

Recall from the introduction we defined the \textbf{traffic} of a configuration $\mathbf{h} \in \mathbb{Z}_n^{(\ell)}$ to be the number of elements of $\{h_1,\ldots,h_\ell\}$ \emph{waiting in traffic}:
\begin{align*}
\mathrm{Traffic}(h_1,\ldots,h_\ell) := \# \{ (j,k) : 1 \leq j \neq k \leq \ell : h_k = h_j + 1 \}.
\end{align*}
Alternatively, writing $E$ for the set $\{h_1,\ldots,h_\ell\}$, we may also write 
\begin{align*}
\mathrm{Traffic}(E) := \# \{ h \in E: h+1 \in E \text{ also}\}.
\end{align*}
Before proving Theorem \ref{thm:tasep}, we recall some general facts about Markov chain generators and their associated exponential transforms. Suppose $\{ \mathbf{P}_x : x \in V \}$ is a collection of probability laws for a Markov chain $(X_t)_{t \geq 0 }$ taking values in a finite set $V$. Let $\mathbb{R}^V := \{ \phi : V \to \mathbb{R}\}$. The generator $\mathcal{G}$ of $(X_t)_{t \geq 0}$ is the linear operator on $\mathbb{R}^V$ given by 
\begin{align*}
\mathcal{G}\phi(x) := \lim_{t \downarrow 0} \frac{1}{t} \mathbf{P}_x \left[ \phi(X_t) - \phi(x) \right].
\end{align*}
Suppose we have a positive function $\Delta:V \to (0,\infty)$ of $\mathcal{G}$ satisfying
\begin{align*}
\mathcal{G} \Delta(x) = p(x) \Delta(x).
\end{align*}
Then the stochastic process $M_t := \frac{\Delta(X_t)}{\Delta(X_0)} e^{- \int_0^t p(X_s) \mathrm{d}s }$ is a positive unit-mean $\mathbf{P}_x$-martingale for each $x \in V$, and thus we may define a change of measure
\begin{align*}
\frac{ \mathrm{d}\mathbf{P}_x^\Delta }{ \mathrm{d} \mathbf{P}_x } \Bigg|_{\mathcal{F}_t} := M_t.
\end{align*}

The measures $\mathbf{P}_x$ and $\mathbf{P}_x^\Delta$ are absolutely continuous with respect to one another.

For $x \neq y$ in $V$, if we set $q(x,y) := \lim_{t \downarrow 0} \frac{1}{t} \mathbf{P}_x ( X_t = y )$, then 
\begin{align} \label{eq:tran}
 \lim_{t \downarrow 0} \frac{1}{t} \mathbf{P}^\Delta_x ( X_t = y ) = \frac{\Delta(y)}{\Delta(x)}q(x,y).
\end{align}
Theorem \ref{thm:tasep} then follows from a statement about applying the generator of TASEP to the function $\Delta:\mathbb{Z}_n^{(\ell)} \to (0,\infty)$. Our main computation is the following lemma.

\begin{lemma} \label{lem:eigengen}
For $\mathbf{h} = (h_1,\ldots,h_\ell)$ in $\mathbb{Z}^{(\ell)}_n$, define the function $\phi(\mathbf{h}) := \det_{j,k=1}^\ell A^\mathbf{h}_{j,k}$ where $A^\mathbf{h}_{j,k}$ is the matrix introduced in \eqref{eq:verdi}. Then
\begin{equation*}
\sum_{ j = 1}^\ell  \phi( \mathbf{h}+\mathbf{e}_j) = - \mu_{n,\ell} \phi(\mathbf{h}) \qquad \text{where} \qquad \mu_{n,\ell} := \frac{ \sin(\pi \ell/n)}{ \sin(\pi/n) }.
\end{equation*}
\begin{proof}
A brief calculation expanding the determinant tells us that 
\begin{align*}
\sum_{ i = 1}^\ell  \phi(h_1,\ldots,h_{i-1},h_i + 1,h_{i+1}, \ldots h_\ell) = \sum_{ i = 1}^\ell  \det \limits_{j,k=1}^\ell \left( z_k^{h_j+\mathrm{1}_{j=i}} \right) = \left( \sum_{z \in \mathcal{L}_{n,\ell}} z \right) \phi(h_1,\ldots,h_\ell).
\end{align*}
The result now follows from noting that $\sum_{z \in \mathcal{L}_{n,\ell}} z  =  - \frac{ \sin(\pi \ell/n)}{ \sin(\pi/n) } = - \mu_{n,\ell}$, which can be seen easily from the fact that $ \sum_{z \in \mathcal{L}_{n,\ell}} z$ is a negative real number with the same modulus as the geometric sum $\sum_{ j=0}^{\ell-1} e^{2 \pi \iota j /n }$. 
\end{proof}
\end{lemma}

We are now ready to prove Theorem \ref{thm:tasep}.

\begin{proof}[Proof of Theorem \ref{thm:tasep}]
Computing $\mathcal{G}^{n,\ell,\mathrm{TASEP}} \Delta(E)$ with $E = \{h_1,\ldots,h_\ell\}$ we have
\begin{align} \label{eq:eriksen}
\mathcal{G}^{n,\ell,\mathrm{TASEP}} \Delta(\mathbf{h}) = \sum_{ j : h_k \neq h_j +1 \forall k} ( \Delta(\mathbf{h}+\mathbf{e}_j) - \Delta(\mathbf{h})) = \sum_{j = 1}^\ell \Delta(\mathbf{h}+\mathbf{e}_j)  - (\ell - \mathrm{Traffic}(\mathbf{h})) \Delta(\mathbf{h}).
\end{align}
We consider the first term on the right-hand-side of \eqref{eq:eriksen}. Using Lemma \ref{lem:rootdet} and the definition of $\phi$ in Lemma \ref{lem:eigengen} we have
\begin{align} \label{eq:matlab}
\sum_{ j = 1}^\ell \Delta(\mathbf{h}+\mathbf{e}_j) = \sum_{j =1}^\ell \frac{ \phi(\mathbf{h}+\mathbf{e}_j) }{ \mathrm{sgn}(\sigma_{\mathbf{h}+\mathbf{e}_j} ) (-1)^{ \sum_{ k=1}^\ell (h_k + \mathrm{1}_{j=k} ) } \iota^{\ell(\ell-1)/2} }.
\end{align}
For convenience we assume w.l.o.g.\ that in our enumeration of $E$ we have $h_1 < \cdots < h_\ell \leq n-2$. Then $\sigma_h$ in the setting of Lemma \ref{lem:rootdet} is the identity permutation, and moreover if for some $1 \leq j \leq \ell$, $\mathbf{h}+\mathbf{e}_j$ is a tuple of distinct elements, then $\sigma_{\mathbf{h}+\mathbf{e}_j}$ is also the identity permutation. It follows that $\mathrm{sgn}(\sigma_{\mathbf{h}+\mathbf{e}_j}) = \sgn(\sigma_{\mathbf{h}})=1$ for all $j$ with $\phi(\mathbf{h}+\mathbf{e}_j) \neq 0$ in \eqref{eq:matlab}. Using this fact to obtain the first equality below, then Lemma \ref{lem:eigengen} to obtain the second, we have 
\begin{align} \label{eq:matlab2}
\sum_{ j = 1}^\ell \Delta(\mathbf{h}+\mathbf{e}_j) &= - (-1)^{ \sum_{ k=1}^\ell h_k  } \iota^{-\ell(\ell-1)/2} \sum_{j =1}^\ell  \phi(\mathbf{h}+\mathbf{e}_j) \nonumber \\
&=  (-1)^{ \sum_{ k=1}^\ell h_k  } \iota^{-\ell(\ell-1)/2}  \mu_{n,\ell} \phi(\mathbf{h})  = \mu_{n,\ell} \Delta(\mathbf{h}),
\end{align}
where in the final equality above we simply repackaged everything using Lemma \ref{lem:rootdet}.

Plugging \eqref{eq:matlab2} into \eqref{eq:eriksen}, we have
\begin{align} \label{eq:eriksen2}
\mathcal{G}^{n,\ell,\mathrm{TASEP}} \Delta(\mathbf{h}) = \sum_{ j : h_k \neq h_j +1 \forall k} ( \Delta(\mathbf{h}+\mathbf{e}_j) - \Delta(\mathbf{h})) = ( \mu_{n, \ell} - \ell + \mathrm{Traffic}(\mathbf{h})) \Delta(\mathbf{h}).
\end{align}
In particular, by the discussion preceding Lemma \ref{lem:eigengen}, $M_t := \frac{ \Delta(X_t)}{ \Delta(X_0) } \exp \left\{ - \int_0^t ( \mu_{n, \ell} - \ell + \mathrm{Traffic}(X_s))\mathrm{d}s \right\}$ is a $\mathbf{P}^{n,\ell,\mathrm{TASEP}}$-Martingale, and if we define a change of measure $\mathbf{P}^{n,\ell}$ by $\mathrm{d} \mathbf{P}^{n,\ell}/\mathrm{d} \mathbf{P}^{n,\ell,\mathrm{TASEP}} = M_t$, then according to \eqref{eq:tran} we have
\begin{align*}
\lim_{t \downarrow 0} \frac{1}{t} \mathbf{P}^{n,\ell}( X_t = E' | X_0 = E) = \frac{ \Delta(E')}{\Delta(E) } \lim_{t \downarrow 0} \frac{1}{t}  \mathbf{P}^{n,\ell,\mathrm{TASEP}}( X_t = E' | X_0 = E),
\end{align*}
in concordance with \eqref{eq:gen} and \eqref{eq:genTASEP}, so that $\mathbf{P}^{n,\ell}$ is indeed Gordenko's change of measure. 

Reciprocating the change of measure, we complete the proof of Theorem \ref{thm:tasep}.
\end{proof}

\section*{Acknowledgements}
The author is immensely grateful to Elia Bisi, Dominik Schmid and Piotr Dyszewski for their valuable comments, and to Neil O'Connell for pointing out several references.\\

\noindent
This research was supported in the early stages by the EPSRC funded Project EP/S036202/1 \emph{Random fragmentation-coalescence processes out of 
equilibrium}.

\end{document}